%% file: PseudoDirStr.tex
\documentclass[11pt]{amsart}
\usepackage{geometry}                
\usepackage{upgreek}
\usepackage{graphicx}
\usepackage{amsmath}%
\usepackage{amsfonts}%
\usepackage{amssymb}
\usepackage{amsthm}
\usepackage{epstopdf}
\usepackage{hyperref}
\usepackage{hyperref}
\usepackage[capitalize]{cleveref}
\usepackage{tikz}
\usetikzlibrary{matrix,arrows}

\theoremstyle{plain}
\newtheorem{theorem}{Theorem}[section]
\newtheorem{lemma}{Lemma}[section]
\newtheorem{corollary}{Corollary}[section]
\newtheorem{proposition}{Proposition}[section]

\theoremstyle{definition}
\newtheorem{definition}{Definition}[section]

\newtheorem{example}{Example}[section]

\theoremstyle{remark}
\newtheorem{remark}{Remark}[section]

\numberwithin{equation}{section}

\crefname{pluralequation}{Eqs.}{Eqs.}
\Crefname{pluralequation}{Eqs.}{Eqs.}

\newcommand{\lb}{[\![}
\newcommand{\rb}{]\!]}
\newcommand{\Cour}[1]{\lb #1\rb}
\newcommand{\la}{\langle}
\newcommand{\ra}{\rangle}
\newcommand{\mf}{\mathfrak}
\newcommand{\mbb}{\mathbb}
\newcommand{\mbf}{\mathbf}
\newcommand{\mc}{\mathcal}
\newcommand{\on}{\operatorname}
\newcommand{\Lied}{\mathcal{L}}

\newcommand{\g}{\mathfrak{g}}
\newcommand{\h}{\mathfrak{h}}
\renewcommand{\d}{\mathrm{d}}
\newcommand{\gr}{\on{gr}}
\newcommand{\ann}{\on{ann}}
\newcommand{\mmat}[2][3em]{\matrix (#2) [matrix of math nodes, row sep=#1,
  column sep=#1, text height=1.5ex, text depth=0.25ex]}
\tikzset{node distance=2cm, auto}

\title{Pseudo-Dirac structures}
\author{David Li-Bland}

\begin{document}
\begin{abstract}
A Dirac structure is a Lagrangian subbundle of a Courant algebroid, $L\subset\mbb{E}$, which is involutive with respect to the Courant bracket. In particular, $L$ inherits the structure of a Lie algebroid. In this paper, we introduce the more general notion of a pseudo-Dirac structure: an arbitrary subbundle, $W\subset\mbb{E}$, together with a \emph{pseudo-connection} on its sections, satisfying a natural integrability condition. As a consequence of the definition, $W$ will be a Lie algebroid. Allowing non-isotropic subbundles of $\mbb{E}$ incorporates non-skew tensors and connections into Dirac geometry. Novel examples of pseudo-Dirac structures arise in the context of quasi-Poisson geometry, Lie theory, generalized K\"ahler geometry, and Dirac Lie groups, among others. Despite their greater generality, we show that pseudo-Dirac structures share many of the key features of Dirac structures. In particular, they behave well under composition with Courant relations.
\end{abstract}
\maketitle
\tableofcontents

\section{Introduction}
Dirac structures were first introduced by Courant and Weinstein \cite{Courant:tm} in their study of constrained mechanical systems. They were further generalized in \cite{ManinTriplesBi}, and have since found applications in myriad settings, from generalized complex geometry \cite{Gualtieri:2007wh}, to the theory of moment maps \cite{Bursztyn:2009wi}.

In their most basic form, Dirac structures encode a Hamiltonian structure on a manifold $M$ as a \emph{Lagrangian} subbundle 
$L\subset \mbb{T}M$
of the \emph{Pontryagin} bundle $\mbb{T}M=TM\oplus T^*M$, satisfying some integrability condition. By \emph{Lagrangian}, we mean that $L=L^\perp$ with respect to the natural symmetric pairing 
$$\la(X,\alpha),(Y,\beta)\ra=\la\alpha,Y\ra+\la\beta,X\ra,\quad (X,\alpha),(Y,\beta)\in TM\oplus T^*M.$$
In particular, any two-form $\omega\in\Omega^2(M)$ corresponds to the Lagrangian subbundle
$$\gr(\omega^\flat):=\big\{\big(X,\omega(X,\cdot)\big)\mid X\in TM\big\}\subset \mbb{T}M,$$
while any bivector field $\pi\in\mf{X}^2(M)$  corresponds to the Lagrangian subbundle
$$\gr(\pi^\sharp):=\big\{\big(\pi(\alpha,\cdot),\alpha\big)\mid \alpha\in T^*M\big\}\subset \mbb{T}M.$$
Arbitrary Lagrangian subbundles of $\mbb{T}M$ interpolate between these two extremes. 

  Courant described a bracket, which is usually referred to as the \emph{Courant bracket} in the literature,\footnote{In fact, this bracket was introduced by Irene Dorfman in the context of two dimensional variational problems \cite{Dorfman:1993us}. Courant worked with its skew symmetrization instead.}  
 \begin{equation}\label{eq:DorfBrk}\Cour{(X,\alpha),(Y,\beta)}=\big([X,Y],\Lied_X\beta-\iota_Yd\alpha\big)\end{equation}
  on the space of sections $\Gamma(\mbb{T}M)$. 
  A \emph{Dirac structure} is defined to be a Lagrangian subbundle $L\subseteq \mbb{T}M$ which is involutive with respect the Courant bracket. As examples, $\gr(\omega^\flat)$ is a Dirac structure if and only if $\omega$ defines a presymplectic structure on $M$, while $\gr(\pi^\sharp)$ is a Dirac structure if and only if $\pi$ defines a Poisson structure on $M$.

In general, the Courant bracket \labelcref{eq:DorfBrk} does not define a Lie bracket  on the sections of $\mbb{T}M$, since it fails to be skew symmetric. 
In order to obtain a Lie bracket from the Courant bracket, Courant restricted the bracket to sections of Lagrangian subbundles $L\subseteq \mbb{T}M$, which ensures skew symmetry of the bracket. In this paper, we will take a different approach to obtain a Lie bracket from the Courant bracket: we modify the Courant bracket itself.


We introduce the notion of a \emph{pseudo-Dirac structure} in the Courant algebroid $\mbb{T}M$: a subbundle $$W\subseteq \mbb{T}M$$ together with a map $$\nabla:\Omega^0(M,W)\to\Omega^1(M,W^*)$$ satisfying certain axioms, including
\begin{subequations}\label[pluralequation]{eq:PsConIntAll}
\begin{align}
\label{eq:PsConInt}\nabla f\sigma&=f\nabla\sigma+\d\!f\otimes \la\sigma,\cdot\ra,&\sigma\in\Gamma(W),\\
\label{eq:derPairInt}\d\la\sigma,\tau\ra&=\la\nabla\sigma,\tau\ra+\la\sigma,\nabla\tau\ra,&\sigma,\tau\in\Gamma(W).
\end{align}
\end{subequations}
Since \cref{eq:PsConIntAll} reduce to the definition of a metric connection when $\la\cdot,\cdot\ra\rvert_W$ is non-degenerate, we call $\nabla$ a \emph{pseudo-connection}. 
\Cref{eq:derPairInt} guarantees that the modification to the Courant bracket, given by 
\begin{equation}\label{eq:ModBrkInt}[\sigma,\tau]:=\Cour{\sigma,\tau}-\la\nabla\sigma,\tau\ra,\end{equation}
is skew symmetric. 
More importantly, in \cref{thm:LieSubIsVBDir}, we prove that \cref{eq:ModBrkInt} endows $W$ with the structure of a Lie algebroid. In particular, the sections of $W$ form a Lie algebra.

Since pseudo-Dirac structures $W\subseteq \mbb{T}M$ need not be Lagrangian subbundles, they are able to encode more data than Dirac structures. 
 As a simple example, which arises in the context of generalized K\"ahler geometry \cite{Gualtieri:2004wh}, suppose $g\in S^2(TM)$ is a pseudo-Riemannian metric on $M$ then
$$\gr(\omega^\flat+g^\flat):=\big\{\big(X,\omega(X,\cdot)+g(X,\cdot)\big)\mid X\in TM\big\}$$
is never a Dirac structure, but it
is a pseudo-Dirac structure with respect to the Levi-Civita connection if and only if $\omega$ is closed. More generally, it is a pseudo-Dirac structure with respect to a given metric connection, $\nabla$, if and only if its torsion is $2\d\omega$ (c.f. \cref{ex:MetCon}). 

As a second example, suppose that  $\g$ is a Lie algebra endowed with a non-degenerate invariant quadratic form, and $\rho:\g\times M\to M$ is an action of $\g$ with coisotropic stabilizers; that is, the stabilizer $\g_m:=\on{ker}\big(\rho(\cdot,m)\big)$ of any point $m\in M$ satisfies $\g_m^\perp\subseteq \g_m$. Then 
$T^*M\subset \mbb{T}M$ is a pseudo-Dirac structure for the pseudo-connection $$\nabla \alpha=\sum_i \rho(\xi_i) \d\alpha\big(\rho(\xi^i)\big),\quad \alpha\in\Omega^1(M),$$
where $\{\xi_i\}\subset\g$ and $\{\xi^i\}\subset\g$ are basis in duality.
This scenario occurs, for instance, when $M=\mc{L}_\g$ is the variety of Lagrangian (or coisotropic) subalgebras of $\g$, or when $M=G/P$ is a flag variety (here $P$ is a parabolic subgroup of a semi-simple algebraic group, $G$).

Further examples of pseudo-Dirac structures include quasi-Poisson structures, action Courant algebroids, and Lie algebroid structures on $T^*M$.

 One of the key aspects of Courant's framework is the ability to impose constraints on a Dirac structure. More precisely, he described a procedure to \emph{restrict}\footnote{Courant calls this procedure \emph{reduction} in \cite{Courant:1990uy}. We refer to it as \emph{restriction} (as is also done in \cite{Jotz:2008wn}, for instance) or \emph{pull-back} (as is done in \cite{LiBland:2009ul}, for instance) to distinguish it from more general reduction procedures.} a Dirac structure $L\subseteq \mbb{T}M$ to a submanifold $S\subseteq M$, 
 \begin{equation}\label{eq:restProc}\big(L\subseteq \mbb{T}M\big)\to \big(L_S\subseteq \mbb{T}S),\end{equation} where $L_S\subseteq \mbb{T}S$ is computed by the formula 
\begin{equation}\label{eq:restForm}L_S:=\frac{L\cap TS\oplus T^*M\rvert_S}{L\cap \ann(TS)}\subseteq \mbb{T}S.\end{equation}
 Courant showed that  \cref{eq:restForm} describes a Dirac structure (under some cleanness assumptions).

\Cref{prop:clasLagTDSB} states that  subbundles of $\mbb{T}M$ endowed with a pseudo-connection are in one-to-one correspondence with Lagrangian subbundles of the Courant algebroid $\mbb{T}TM$ which are linear with respect to the vector bundle structure on $TM$.
 Thus it is quite natural to study them. Similarly, pseudo-Dirac structures in $\mbb{T}M$ are in one-to-one correspondence with (linear) Dirac structures in $\mbb{T}TM$.

As a consequence, it is easy to impose constraints on a pseudo-Dirac structure in $\mbb{T}M$:
Suppose that $S\subseteq M$ is the constraint submanifold. Courant's restriction procedure \labelcref{eq:restProc} allows us to restrict Dirac structures in $\mbb{T}TM$ to Dirac structure in $\mbb{T}TS$:
\begin{center}
\begin{tikzpicture}
\node (B) at (-4,-1) [text width=4cm,align=center] {Dirac structures in $\mbb{T}TM$};
\node (D) at (4,-1) [text width=4cm,align=center] {Dirac structures in $\mbb{T}TS$};
\draw [->] (B) edge node  [text width=4cm,align=center,below] {Courant's restriction procedure \labelcref{eq:restProc}} (D);
\end{tikzpicture}
\end{center} 
Composing this restriction procedure with the equivalence between (linear) Dirac structures in $\mbb{T}TM$ and pseudo-Dirac structures in $\mbb{T}M$ allows us to impose constraints on pseudo-Dirac structures in $\mbb{T}M$:
\begin{center}
\begin{tikzpicture}
\node (A) at (-4,1) [text width=4cm,align=center] {pseudo-Dirac structures in $\mbb{T}M$};
\node (B) at (-4,-1) [text width=4cm,align=center] {(linear) Dirac structures in $\mbb{T}TM$};
\node (C) at (4,1) [text width=4cm,align=center] {pseudo-Dirac structures in $\mbb{T}S$};
\node (D) at (4,-1) [text width=4cm,align=center] {(linear) Dirac structures in $\mbb{T}TS$};
\draw [->] (B) edge node  [text width=4cm,align=center,below] {Courant's restriction procedure \labelcref{eq:restProc}} (D);
\draw [->,dashed] (A) edge node [text width=4cm,align=center] {restriction procedure for pseudo-Dirac structures} (C);
\draw [<->] (A) edge (B);
\draw [<->] (C) edge (D);
\end{tikzpicture}
\end{center}

In this way, pseudo-Dirac structures share the key features of Dirac structures, while encoding a greater variety of structures. Moreover, we should remark that while in the introduction we only described pseudo-Dirac structures in the standard Courant algebroid, $\mbb{T}M$, in the bulk of the paper, we will describe pseudo-Dirac structures in arbitrary Courant algebroids.

The structure of this paper is as follows. In \cref{sec:Pre} we recall some background material, including the general definition of a Courant algebroid structure on a vector bundle, $\mbb{E}\to M$. In \cref{sec:PseuDir} we define pseudo-Dirac structures in $\mbb{E}$, and describe a number of examples. In \cref{sec:LieSubAlg}, we relate pseudo-Dirac structures in $\mbb{E}$ with Dirac structures in $T\mbb{E}$. Finally, in \cref{sec:FBimage} we describe the behavior of pseudo-Dirac structures under composition with Courant relations, and describe further examples.

\subsection{Acknowledgements}
D. Li-Bland would like to thank Eckhard Meinrenken for his advice throughout this project, as well as Marco Gualtieri, Mathieu Stienon and Madeleine Jotz for helpful conversations. D. Li-Bland was supported in part by a NSERC CGS-D Grant as well as by NSF Award No. DMS-1204779.

\section{Preliminaries}\label{sec:Pre}

\subsection{Relations}\label{sec:LinRel}

A smooth relation $S\colon M_1\dasharrow M_2$ between manifolds is an immersed submanifold 
$S\subseteq M_2\times M_1$. The \emph{transpose relation} $S^\top \colon M_2\dasharrow M_1$ consists of all $(m_1,m_2)$ such that 
$(m_2,m_1)\in R$.

We will write $$m_1\sim_S m_2,$$ if $(m_2,m_1)\in S\subseteq M_2\times M_1$, and for functions $f_i\in C^\infty(M_i)$, we will write $$f_1\sim_S f_2$$ if $f_2\oplus(-f_1)$ vanishes on $S$.
\begin{example}
The identity relation $M_\Delta: M\dasharrow M$ is given by the diagonal 
$$M_\Delta:=\{(m,m)\in M\times M\mid m\in M\}.$$
\end{example}

Given smooth relations $S\colon M_1\dasharrow M_2$ and 
$S'\colon M_2\dasharrow M_3$, the set-theoretic composition $S'\circ S$ is the image of 
\begin{equation}\label{eq:intersection}
 S'\diamond S=(S'\times S)\cap (M_3\times (M_2)_\Delta \times M_1)
 \end{equation}
under projection to $M_3\times M_1$. Here $(M_2)_\Delta\subset M_2\times M_2$ denotes the diagonal.
\subsubsection{$\mc{VB}$-relations}
If $V_1,V_2$ are two vector bundles over $M_1$ and $M_2$ respectively, then a \emph{$\mc{VB}$-relation} $R:V_1\dasharrow V_2$ is a subbundle
$R\subseteq V_2\times V_1$  along a submanifold $S\subset M_2\times M_1$.

We define $\ker(R)\subseteq p_{M_1}^*V_1,$ and $\on{ran}(R)\subseteq p_{M_2}^*V_2$ to be the kernel and range of the bundle map $$R\to p_{M_2}^*V_2,\ 
(v_2,v_1)\mapsto v_2$$ (where $p_{M_i}\colon S\to M_i,\ (m_2,m_1)\mapsto m_i$). 
If $\sigma_i\in\Gamma(V_i)$, then we write $$\sigma_1\sim_R\sigma_2$$ whenever $(\sigma_2,\sigma_1)\rvert_S\in \Gamma(R)$.

We define the $\mc{VB}$-relation $\ann^\natural(R):V_1^*\dasharrow V_2^*$ by $$\ann^\natural(R)=\{(\mu_2,-\mu_1)\mid (\mu_2,\mu_1)\in \ann(R)\}\subseteq V_2^*\times V_1^*.$$

We will frequently find the following Lemma useful.
\begin{lemma}[\cite{LiBland:2011vqa}]\label{lem:ann}
For any relations $R\colon V_1\dasharrow	 V_2$ and $R'\colon V_2\dasharrow V_3$, 
one has  
$$\ann^\natural(R'\circ R)=\ann^\natural(R')\circ \ann^\natural(R).$$ 
\end{lemma}
The proof can be found in \cite{LiBland:2011vqa}.

\begin{example}\label{ex:GrAdd}
Let $p:V\to M$ be a vector bundle. Then the $\mc{VB}$-relation $$\gr(+):V\times V\dasharrow V$$ defined by $$(v_1,v_2)\sim_{\gr(+)} v_1+v_2,$$ whenever $p(v_1)=p(v_2)$ is called the \emph{graph of addition}. A (fibrewise) linear function $f \in C^\infty(V)$ satisfies  $$f\oplus f\sim_{\gr(+)}f$$ while a  fibrewise constant function satisfies $$f\oplus 0\sim_{\gr(+)}f.$$
\end{example}

%

\subsection{Lie algebroids and Courant algebroids}

\subsubsection{Lie algebroids}
Lie algebroids, which are a common generalization of Lie algebras and the tangent bundle, will be an important concept in this paper. We recall their definition and basic properties, and refer the reader to \cite{moerdijk03,Mackenzie05} for more comprehensive expositions.

\begin{definition}
A \emph{Lie algebroid} is a vector bundle $A\to M$ together with a Lie bracket $[ \cdot,\cdot ]$
on its space of sections $\Gamma(A)$ and a 
bundle map $\mbf{a}:A\to TM$ called the \emph{anchor map} 
%
such that the following Leibniz identity is satisfied:
$$[\sigma,f\tau]=(\mbf{a}(\sigma)\cdot f)\tau+f[\sigma,\tau],\quad \sigma,\tau\in\Gamma(A), f\in C^\infty(M).$$

\begin{remark}
Note that the anchor map is determined by the Lie bracket. Indeed $$(\mbf{a}(\sigma)\cdot f)\sigma=[\sigma,f\sigma], \quad \sigma\in\Gamma(A),f\in C^\infty(M).$$ Additionally, the anchor map intertwines the Lie brackets:
$$\mbf{a}[\sigma,\tau]=[\mbf{a}(\sigma),\mbf{a}(\tau)].$$
\end{remark}

\end{definition}

\begin{example}[The tangent bundle]\label{ex:TangBundLie}
The tangent bundle $A=TM\to M$ is a Lie algebroid, where the bracket on $\Gamma(TM)=\mf{X}(M)$ is the Lie bracket of vector fields, and the anchor map $\mbf{a}:TM\to TM$ is the identity map.

\end{example}

\begin{example}[Lie algebras and action Lie algebroids]
Any Lie algebra $\g$, regarded as a vector bundle over the point, is a Lie algebroid. 

More generally, if $\g$ acts on a manifold $M$ via the Lie algebra morphism $\rho:\g\to \mf{X}(M)$, then $\g\times M$ is a Lie algebroid with bracket 
\begin{equation}
\label{eq:actionlie} [\xi_1,\xi_2]_{\g\times M}=[\xi_1,\xi_2]_\g+\Lied_{\rho(\xi_1)}\xi_2-\Lied_{\rho(\xi_2)}\xi_1,\quad \xi_1,\xi_1\in\Gamma(\g\times M)\cong C^\infty(M,\g).\end{equation} The anchor map $\mbf{a}:\g\times M\to TM$ is defined to extend the map $\rho:\g\to\mf{X}(M)$ on constant sections. With this structure, $\g\times M$ is called the \emph{action Lie algebroid} for the action of $\g$ on $M$.
\end{example}

Next we recall the notion of a Lie subalgebroid of a Lie algebroid, due to Higgins and Mackenzie \cite{Higgins:1990gq}.

\begin{definition}\label{def:SubLA}
Let $A\to M$ be a Lie algebroid. A subbundle 
$B\subseteq A$ over a submanifold $S\subseteq M$ is called a Lie subalgebroid if for any two sections 
$\sigma,\tau\in\Gamma(A)$ satisfying $\sigma\rvert_S,\tau\rvert_S\in \Gamma(B),$
$$[\sigma,\tau]\rvert_S\in\Gamma(B).$$

\end{definition}

\begin{definition}\label{def:MorphLA}
Suppose that $A_1\to M_1$ and $A_2\to M_2$ are two Lie algebroids.
\begin{itemize}
 \item A morphism of vector bundles $\Phi:A_1\to A_2$ is called a Lie algebroid morphism if its graph $\on{gr}(\Phi)\subset A_2\times A_1$ is a Lie subalgebroid \cite{Higgins:1990gq}.

\item Suppose $\Psi:\psi^* A_2\to A_1$ is a base-preserving map of vector bundles, where $\psi^*A_2$ is the pullback of $A_2$ along a map $\psi:M_1\to M_2$. If $\on{gr}(\Psi)\subset A_2\times A_1$ is a Lie subalgebroid, then the vector bundle relation $\gr(\Psi):A_1\dasharrow A_2$ is said to be a \emph{comorphism} of Lie algebroids \cite{Higgins:1993bc}.

\item More generally, a relation $R:A_1\dasharrow A_2$ is called an $\mc{LA}$ relation if $R\subseteq A_2\times A_1$ is a Lie subalgebroid. In particular, for any $\sigma_i,\tau_i\in\Gamma(A_i)$ satisfying $\sigma_1\sim_R\sigma_2$ and $\tau_1\sim_R\tau_2$, we have
$$[\sigma_1,\tau_1]\sim_R[\sigma_2,\tau_2].$$
\end{itemize}

%

%
%
%
\end{definition}

\subsubsection{Courant algebroids}
Dirac structures were introduced by Courant \cite{Courant:1990uy} as a unified framework from which to study constrained mechanical systems. Liu-Weinstein-Xu \cite{ManinTriplesBi}
generalized Courant's original set-up, replacing $\mbb{T}M$ with a more general notion of a \emph{Courant algebroid}
$\mbb{E}\to M$.  We recall the basic theory, and refer the reader to \cite{Dorfman:1993us,LetToWein,Roytenberg99,Severa:2005vla,Roytenberg:2002,Uchino02,Bursztyn:2009wi} for more details.

\begin{definition}\label{def:CA}
A \emph{Courant algebroid} over a manifold $M$ is a vector bundle $\mbb{E}\to M$, together with a bundle 
map $\mbf{a}\colon \mbb{E}\to TM$ called the \emph{anchor}, a bundle metric\footnote{In this paper, we take 
`metric' to mean a non-degenerate symmetric bilinear form which is not necessarily positive definite.} $\la\cdot,\cdot\ra$, and 
 a bilinear bracket $\Cour{\cdot,\cdot}$ on its space of sections $\Gamma(\mbb{E})$. These are required 
to satisfy the following axioms, for all sections $\sigma_1,\sigma_2,\sigma_3\in\Gamma(\mbb{E})$:
\begin{enumerate}
\item[c1)] $\Cour{\sigma_1,\Cour{\sigma_2,\sigma_3}}=\Cour{\Cour{\sigma_1,\sigma_2},\sigma_3}
+\Cour{\sigma_2,\Cour{\sigma_1,\sigma_3}}$, 
\item[c2)] $\mbf{a}(\sigma_1)\la \sigma_2,\sigma_3\ra=\la \Cour{\sigma_1,\sigma_2},\,\sigma_3\ra+\la \sigma_2,\,\Cour{\sigma_1,\sigma_3}\ra$,
\item[c3)] $\Cour{\sigma_1,\sigma_2}+\Cour{\sigma_2,\sigma_1}=\mbf{a}^*(\d \la \sigma_1,\sigma_2\ra)$.
\end{enumerate}
Here $\mbf{a}^*\colon T^*M\to \mbb{E}^*\cong\mbb{E}$ is the dual map to $\mbf{a}$. The axioms c1)-c3) imply various other properties, 
in particular
\begin{enumerate}
\item[c4)] $\Cour{\sigma_1,f\sigma_2}=f\Cour{\sigma_1,\sigma_2}+\mbf{a}(\sigma_1)(f)\sigma_2$,
\item[c5)] $\Cour{f\sigma_1,\sigma_2}=f\Cour{\sigma_1,\sigma_2}-\mbf{a}(\sigma_2)(f)\sigma_1+\la \sigma_1,\sigma_2\ra \mbf{a}^*(\d f)$, 
\item[c6)] $\mbf{a}(\Cour{\sigma_1,\sigma_2})=[\mbf{a}(\sigma_1),\mbf{a}(\sigma_2)]$,
\end{enumerate}
for sections $\sigma_i\in\Gamma(\mbb{E})$ and functions $f\in C^\infty(M)$. We
will refer to the bracket $\Cour{\cdot,\cdot}$ as the \emph{Courant
bracket} (some authors refer to $\Cour{\cdot,\cdot}$ as the Dorfman bracket (after Dorfman \cite{Dorfman:1993us}, who introduced it) and its skew-symmetric part as the Courant bracket). 

For any Courant algebroid $\mbb{E}$, we denote by $\overline{\mbb{E}}$ the Courant algebroid with the same 
bracket and anchor, but with the bundle metric, $\la\cdot,\cdot\ra$, negated.\footnote{To see that axiom c3) holds for $\overline{\mbb{E}}$, it can be useful to rewrite it as $$\la\Cour{\sigma_1,\sigma_2}+\Cour{\sigma_2,\sigma_1},\sigma_3\ra=\mbf{a}(\sigma_3)\cdot\la \sigma_1,\sigma_2\ra,$$ an equation whose validity is manifestly preserved when the bundle metric is negated. Note that  we abused notation when writing the original equation, denoting  the composition $T^*M\xrightarrow{\mbf{a}^*}\mbb{E}^*\xrightarrow{\la\cdot,\cdot\ra}\mbb{E}$ simply by $\mbf{a}^*$. }

\end{definition}

A subbundle $E\subseteq \mbb{E}$ along a submanifold $S\subseteq M$ is called \emph{involutive} if it has the 
property 
\[ \sigma_1|_S,\ \sigma_2|_S\in\Gamma(E) 
\Rightarrow \Cour{\sigma_1,\sigma_2}|_S\in \Gamma(E),\]  
for any $\sigma_1,\sigma_2\in\Gamma(\mbb{E})$.
It is important to note that this does not define a bracket on sections of $E$, in general.

We let $L^\perp\subseteq \mbb{E}$ denote the orthogonal complement of $L$ with respect to the fibre metric.
A subbundle $L\subset \mbb{E}$ is called \emph{Lagrangian} if $L^\perp=L$, and coisotropic if $L^\perp\subseteq L$.
An involutive Lagrangian subbundle $E\subseteq \mbb{E}$  along $S\subseteq M$ is called
a \emph{Dirac structure along $S$}.   

A Dirac structure, $E$, along $S=M$ is
simply called a Dirac structure, and the pair $(\mbb{E},E)$ is called a \emph{Manin pair} \cite{PonteXu:08,Bursztyn:2009wi}. In this case, the restriction of the Courant bracket and the anchor map endows $E$ with the structure of a Lie algebroid \cite{Courant:1990uy,ManinTriplesBi}.
Dirac structures were introduced by
Courant \cite{Courant:1990uy} and Liu-Weinstein-Xu
\cite{ManinTriplesBi}. The notion of a Dirac structure along a
submanifold goes back to \v{S}evera \cite{LetToWein} and was developed
in \cite{Alekseev:2002tn,Bursztyn:2009wi,PonteXu:08}.

\begin{example}[The standard Courant algebroid]\label{ex:StdCourAlg}
The \emph{standard Courant algebroid} over $M$ is $\mbb{T} M=TM\oplus T^*M$ with anchor map the projection to the first factor and bilinear form 
$\la (X,\alpha),(Y,\beta)\ra=\la\beta,X\ra+\la \alpha,Y\ra$. The Courant bracket reads
\begin{equation}\label{eq:StdCourBrack} \Cour{(X,\alpha),(Y,\beta)}=([X,Y],\Lied_{X}\beta-\iota_{Y}d\alpha),\end{equation}
for vector fields $X,Y\in\mf{X}(M)$ and 1-forms $\alpha,\beta\in\Omega^1(M)$.

Both $TM$ and $T^*M$ are Dirac structures in $\mbb{T} M$. 

The standard Courant algebroid was introduced by Courant \cite{Courant:1990uy}.
\end{example}

\begin{example}[Poisson structures]\label{ex:StdCourAlgPoisSymp}
A manifold $M$ is called a \emph{Poisson manifold} if the space of functions $C^\infty(M)$ is equipped with a Lie bracket $$\{\cdot,\cdot\}:C^\infty(M)\times C^\infty(M)\to C^\infty(M)$$ satisfying the Leibniz rule, $$\{f,gh\}=\{f,g\}h+g\{f,h\}, \quad f,g,h\in C^\infty(M).$$
Consequently, for any function $f\in C^\infty(M)$, the operator $\{f,\cdot\}:C^\infty(M)\to C^\infty(M)$ is a derivation, and so it defines a vector field $X_f=\{f,\cdot\}$ on $M$ called the \emph{Hamiltonian vector field} associated to $f$. By skew symmetry the Poisson bracket is also a derivation in the first variable. 

Since $\{f,g\}$ depends only on the differentials $\d\! f$ and $\d g$, there exists a bivector field $\pi\in\mf{X}^2(M)$ such that \begin{equation}\label{eq:bivField}\{f,g\}=\pi(\d\! f,\d g).\end{equation} We let $\pi^\sharp:T^*M\to TM$ be the associated skew-symmetric map, that is $\pi^\sharp(\d\! f)=X_f$.

Given a bivector field, $\pi\in\mf{X}^2(M)$, then the graph $\on{gr}(\pi^\sharp)\subset \mbb{T}M$ of the associated skew symmetric map $\pi^\sharp:T^*M\to TM$ is a Dirac structure if and only if \cref{eq:bivField} defines a Poisson structure \cite{Courant:1990uy,Courant:tm}.

\end{example}

\begin{example}[Exact Courant algebroids]\label{ex:exactCA}
The theory of \emph{exact Courant algebroids} was developed by \v{S}evera \cite{Severa:2001}.
A Courant algebroid $\mbb{E}\to M$ is called \emph{exact}, if the sequence $$0\to T^*M\xrightarrow{\mbf{a}^*}\mbb{E}\xrightarrow{\mbf{a}}TM\to0$$ is  exact. In this case, any Lagrangian splitting $s:TM\to\mbb{E}$ of this sequence defines a closed 3-form $\gamma\in\Omega^3(M)$ by the formula 
\begin{equation}\label{eq:SplitExtGam}\gamma(X,Y,Z)=\la\Cour{s(X),s(Y)},s(Z)\ra,\quad X,Y,Z\in\Gamma(TM)\end{equation}
Additionally, the splitting defines a trivialization $$(\mbf{a}\times s^*):\mbb{E}\to TM\times_M T^*M,$$ which intertwines the bundle metric on $\mbb{E}$ with the natural pairing on $TM\oplus T^*M$. 

Under this identification, the Courant bracket becomes
\begin{equation}\label{eq:twistCourBrack} \Cour{(X,\alpha),(Y,\beta)}_\gamma=([X,Y],\Lied_{X}\beta-\iota_{Y}d\alpha+\iota_X\iota_Y\gamma),\end{equation}
for vector fields $X,Y\in\mf{X}(M)$ and 1-forms $\alpha,\beta\in\Omega^1(M)$.

More generally, for any closed 3-form $\gamma\in\Omega^3_{cl}(M)$, we shall denote the Courant algebroid on the pseudo-euclidean bundle $TM\oplus T^*M$ equipped with bracket \labelcref{eq:twistCourBrack} by $\mbb{T}_\gamma M$. Thus the splitting $s:TM\to \mbb{E}$ identifies $\mbb{E}\cong \mbb{T}_\gamma M$.

 Up to isomorphism, exact Courant algebroids are classified by the de Rham cohomology class $[\gamma]\in H^3_{dR}(M)$, often referred to as the \emph{\v{S}evera class} of the Courant algebroid.
\end{example}

\begin{example}[Quadratic Lie algebras]\label{ex:ActCourAlg}
A Lie algebra together with an invariant metric is called a \emph{quadratic Lie algebra}.
Courant algebroids over a point correspond to \emph{quadratic Lie algebras}.

Suppose $\mf{d}$ is a quadratic Lie algebra,
acting on a manifold $M$. Let $\rho\colon \mf{d}\times M\to TM$ be the
action map. Let $\mbb{E}=\mf{d}\times M$ with anchor map $\mbf{a}=\rho$ and with
the bundle metric coming from the metric on $\mf{d}$. As shown in
\cite{LiBland:2009ul}, the Lie bracket on constant sections $\mf{d}\subseteq
C^\infty(M,\mf{d})=\Gamma(\mbb{E})$ extends to a Courant bracket if and only if
the action has coisotropic stabilizers, i.e. $\ker(\rho(\cdot,m))\subseteq
\mf{d}$ satisfies $\ker(\rho(\cdot,m))^\perp\subseteq \ker(\rho(\cdot,m))$ for any $m\in M$. Explicitly, for $\sigma_1,\sigma_2\in \Gamma(\mbb{E})=C^\infty(M,\mf{d})$ the
Courant bracket reads (see \cite[$\mathsection$ 4]{LiBland:2009ul})
\begin{equation}
\label{eq:actioncourant} \Cour{\sigma_1,\sigma_2}=[\sigma_1,\sigma_2]+\Lied_{\rho(\sigma_1)}\sigma_2-\Lied_{\rho(\sigma_2)}\sigma_1+\rho^*\la \d\sigma_1,\sigma_2\ra.\end{equation}
Here $\rho^*\colon T^*M\to \mf{d}\times M$ is the dual map to the action map, 
using the metric to identify $\mf{d}^*\cong \mf{d}$. 
Note that the first three terms give the Lie algebroid bracket \labelcref{eq:actionlie} for the action Lie algebroid $\mf{d}\times M$. The correction term \begin{equation}\label{eq:ActCourCorr}\rho^*\la \d\sigma_1,\sigma_2\ra\end{equation} turns the Lie algebroid bracket into a 
Courant bracket. We refer to $\mf{d}\times M$ with bracket \labelcref{eq:actioncourant} as an \emph{action Courant algebroid}.
\end{example}

\subsubsection{Courant Morphisms and Relations}

\begin{definition}[Courant morphisms and relations \cite{Alekseev:2002tn}]\label{def:CARel}
Let $\mbb{E}_1,\mbb{E}_2$ be two Courant algebroids over $M_1$ and $M_2$, respectively. A \emph{Courant relation} $R:\mbb{E}_1\dasharrow\mbb{E}_2$ is a Dirac structure $R\subseteq \mbb{E}_2\times\overline{\mbb{E}_1}$ along a submanifold $S\subset M_2\times M_1$.
If $S$ is the graph of a map $M_1\to M_2$, then $R$ is called a \emph{Courant morphism}.
\end{definition}

As a consequence of the definition, if $\sigma_i\in \Gamma(\mbb{E}_i)$ and $\tau_i\in\Gamma(\mbb{E}_i)$ satisfy 
$\sigma_1\sim_R\sigma_2$, and $\tau_1\sim_R\tau_2$, then
\begin{equation*}\begin{split}
\Cour{\sigma_1,\tau_1}&\sim_R\Cour{\sigma_2,\tau_2},\\
\la\sigma_1,\tau_1\ra&\sim_R\la\sigma_2,\tau_2\ra.
\end{split}
\end{equation*}

\begin{example}\label{ex:diagDirStr}
Suppose $\mbb{E}$ is a Courant algebroid over $M$. Then the diagonal $\mbb{E}_\Delta\subseteq \mbb{E}\times\overline{\mbb{E}}$ is a Dirac structure with support along the diagonal $M_\Delta\subseteq M\times M$. The corresponding Courant relation, $$\mbb{E}_\Delta:\mbb{E}\dasharrow\mbb{E}$$ is just the identity map.
\end{example}

\begin{example}[Standard lift of a relation]\label{ex:StdDiracStr}
Let $S\subset M$ be an embedded submanifold. Then $TS\oplus \ann(TS)\subseteq \mbb{T} M$ is a Dirac structure along $S$. Moreover, if $S:M_1\dasharrow M_2$ is a relation, then $R_S:=TS\oplus\ann^\natural(TS)\subseteq \mbb{T}M_2\times \overline{\mbb{T} M_1}$ defines a Courant relation $$R_S:\mbb{T}_{\gamma_1}M_1\dasharrow \mbb{T}_{\gamma_2}M_2,$$
whenever $(\gamma_2\oplus-\gamma_1)\rvert_S=0$.
\end{example}

\begin{example}[{\cite[Proposition 1.6]{LiBland:2009ul}} the diagonal morphism]\label{ex:DiagMorph}
Suppose $\mbb{E}\to M$ is a Courant algebroid, then
\begin{equation}\label{eq:DiagMorph}R_{\on{diag}}:(\mbb{T}M,TM)\dasharrow (\mbb{E}\times\overline{\mbb{E}},\mbb{E}_\Delta)\end{equation}
is a Morphism of Manin pairs over the diagonal embedding $M\to M\times M$, where
\begin{equation}\label{eq:DiagMorph2}(v,\mu)\sim_{R_{\on{diag}}} (x,y)\Leftrightarrow v=\mbf{a}(x),\quad x-y=\mbf{a}^*\mu.\end{equation}
\end{example}

If two Courant relations $R_1:\mbb{E}_1\dasharrow \mbb{E}_2$ and $R_2:\mbb{E}_2\dasharrow\mbb{R}_3$ compose \emph{cleanly}, then their composition $R_2\circ R_1:\mbb{E}_1\dasharrow \mbb{R}_3$ is a Courant relation (see \cite[Proposition~1.4]{LiBland:2011vqa}). Furthermore: 

\begin{proposition}[{\cite[Proposition~1.4]{LiBland:2011vqa}}]
Suppose that $\mbb{E}\to M$ and $\mbb{F}\to N$ are Courant algebroids and
$R:\mbb{E}\dasharrow\mbb{F}$ is a Courant relation.
\begin{itemize}
\item If $E\subseteq \mbb{E}$ is a Dirac structure which composes cleanly with $R$ and the subbundle $F\subseteq \mbb{F}$ supported on all of $N$ satisfies $F=R\circ E$, then $F\subseteq \mbb{F}$ is a Dirac structure.
\item If $F\subseteq \mbb{F}$ is a Dirac structure which composes cleanly with $R$ and the subbundle $E\subseteq \mbb{E}$ supported on all of $M$ satisfies $E=F\circ R$, then $E\subseteq \mbb{E}$ is a Dirac structure.
\end{itemize}
\end{proposition}

\section{Pseudo-Dirac structures}\label{sec:PseuDir}
In this section we describe our main results and examples. Proofs of the results will be presented later, in \cref{sec:LieSubAlg}, after describing the tangent prolongation of a Courant algebroid \cite{Boumaiza:2009eg}, and the theory of double structures.

\begin{definition}\label{def:pseuCon}
Suppose $\mbb{E}\to M$ is a vector bundle with a bundle metric $\la\cdot,\cdot\ra$, and $W\subseteq\mbb{E}$ is a subbundle. A map $\nabla:\Omega^0(M,W)\to \Omega^1(M, W^*)$ satisfying 
\begin{subequations}\label[pluralequation]{eq:pseudoCon}
\begin{align} 
\label{eq:NabDer}\nabla(f\sigma)&=f\nabla\sigma+\d\!f\otimes \la\sigma,\cdot\ra,\\
\label{eq:NabPairDer}\d\la\sigma,\tau\ra&=\la\nabla\sigma,\tau\ra+\la\sigma,\nabla\tau\ra,
\end{align}
\end{subequations}
 is called a \emph{pseudo-connection} for $W$. 
 
 Given a smooth map $\phi:N\to M$,  the \emph{pull-back pseudo-connection} $\phi^*\nabla$ for $\phi^*W\subseteq \phi^*\mbb{E}$ is defined by $$(\phi^*\nabla)_X\phi^*\sigma=\phi^*(\nabla_{d\phi(X)}\sigma),$$ where $\sigma\in\Gamma(W)$ and $X\in TN$.
\end{definition}

\begin{remark}
When $W\subseteq \mbb{E}$ is a quadratic subbundle (i.e. $W^\perp\cap W=0$), then \cref{eq:pseudoCon} are just the axioms of a metric connection. When $W\subseteq\mbb{E}$ is Lagrangian, then \cref{eq:NabDer} implies that $\nabla:\Omega^0(M,W)\to \Omega^1(M, W^*)$ is tensorial, and \cref{eq:NabPairDer} implies that $\nabla\in\Omega^1(M,\wedge^2W^*).$
\end{remark}

For a subbundle $W\subseteq\mbb{E}$ of a Courant algebroid, modifying the Courant bracket using a pseudo-connection, 
\begin{equation}\label{eq:pcModBrk}[\sigma,\tau]:=\Cour{\sigma,\tau}-\mbf{a}^*\la \nabla\sigma,\tau\ra,\quad\sigma,\tau\in\Gamma(W),\end{equation} 
defines a  $\Gamma(\mbb{E})$-valued bracket on sections of $W$. 

Using the modified bracket, we define a `torsion' tensor for the pseudo-connection, \begin{equation}\label{eq:torsTens}T(\sigma,\tau,\upsilon)=\la\nabla_{\mbf{a}(\sigma)}\tau-\nabla_{\mbf{a}(\tau)}\sigma-[\sigma,\tau],\upsilon\ra,\quad\sigma,\tau,\upsilon\in\Gamma(W).\end{equation}
\begin{remark}
In \cref{sec:LieSubAlg} we will show that the bracket \labelcref{eq:pcModBrk} is skew symmetric (\cref{lem:ModBrkProp1}), and that \cref{eq:torsTens} defines a skew-symmetric tensor, i.e. $T\in\Gamma(\wedge^3W^*)$ (\cref{prop:TorsTens}).
\end{remark}

\begin{definition}\label{def:LieSubalg}
Suppose $\mbb{E}\to M$ is a Courant algebroid.
A pair, $(W,\nabla)$, consisting of a subbundle $W\subseteq \mbb{E}$ together with pseudo-connection $\nabla$ for $W\subseteq \mbb{E}$ (cf. \cref{def:pseuCon}) is called a \emph{pseudo-Dirac structure} in $\mbb{E}$
if
\begin{itemize}
\item the modified bracket
\labelcref{eq:pcModBrk} takes values in $\Gamma(W)$, and
\item  the following expression 
\begin{equation}\label{eq:Psi_L}
\begin{split}
\Psi(\sigma,\tau,\upsilon)=&\la[\sigma,\tau],\nabla\upsilon\ra+\la[\upsilon,\sigma],\nabla\tau\ra+\la[\tau,\upsilon],\nabla\sigma\ra\\
&+\iota_{\mbf{a}(\sigma)}\d\la\nabla\tau,\upsilon\ra+\iota_{\mbf{a}(\upsilon)}\d\la\nabla\sigma,\tau\ra+\iota_{\mbf{a}(\tau)}\d\la\nabla\upsilon,\sigma\ra\\
&+\d T(\sigma,\tau,\upsilon),
\end{split}\end{equation}
for $\sigma,\tau,\upsilon\in\Gamma(W)$ vanishes.
\end{itemize}
\end{definition}

\begin{remark}[$\Psi$ is a tensor]
Suppose that  the modified bracket \labelcref{eq:pcModBrk} takes values in $\Gamma(W)$. 
In \cref{sec:LieSubAlg} we shall prove that \cref{eq:Psi_L} defines a skew symmetric tensor on $W$ (\cref{lem:PsiProp}). That is,  $$\Psi\in\Omega^1(M, \wedge^3 W^*).$$
\end{remark}

\begin{remark}\label{rem:PsiAsCurv} If $W\subseteq\mbb{E}$ is quadratic, then a pseudo-connection is simply a metric connection on $W$. Using the identity $$d\la\nabla\tau,\upsilon\ra=\la R\tau,\upsilon\ra-\la\nabla\tau\wedge\nabla\upsilon\ra,$$ where $R\in\Omega^2(M,\mf{o}(W))$ is the curvature tensor, one may rewrite \cref{eq:Psi_L} as
\begin{equation}\label{eq:PsiMet}
\Psi(\sigma,\tau,\upsilon)=\iota_{\mbf{a}(\sigma)}\la R\tau,\upsilon\ra+\iota_{\mbf{a}(\upsilon)}\la R\sigma,\tau\ra+\iota_{\mbf{a}(\tau)}\la R\upsilon,\sigma\ra
+(\nabla T)(\sigma,\tau,\upsilon).
\end{equation}
So, in this case, $\Psi\in\Omega^1(M,\wedge^3 W^*)$ can be expressed entirely in terms of the curvature and torsion of the connection.
\end{remark}

\begin{remark} The tensor $\Psi$ can be understood as an obstruction to the modified bracket \labelcref{eq:pcModBrk} satisfying the Jacobi identity. Indeed, for $\sigma,\tau,\upsilon\in\Gamma(W)$, we have
$$[\sigma,[\tau,\upsilon]]+[\tau,[\upsilon,\sigma]]+[\upsilon,[\sigma,\tau]]=-\mbf{a}^*\Psi(\sigma,\tau,\upsilon).$$

\end{remark}
The following theorem, one of our main results, will justify our interest in pseudo-Dirac structures.
\begin{theorem}\label{thm:LieSubIsVBDirMinor}
Suppose $\mbb{E}\to M$ is a Courant algebroid.
If $(W,\nabla)$ is a pseudo-Dirac structure for $\mbb{E}$, then 
$$[\sigma,\tau]:=\Cour{\sigma,\tau}-\mbf{a}^*\la \nabla\sigma,\tau\ra,\quad\sigma,\tau\in\Gamma(W),$$
 defines a Lie algebroid bracket on $W$.
\end{theorem}

We will prove \cref{thm:LieSubIsVBDirMinor} in \cref{sec:LieSubAlg}, where we will also show that there is a one-to-one correspondence between pseudo-Dirac structures in $\mbb{E}\to M$ and  (linear) Dirac structures in the tangent prolongation, $T\mbb{E}\to TM$, of $\mbb{E}$.

\subsection{Examples}

\begin{example}
If $W\subseteq\mbb{E}$ is a Dirac structure, then 
   $(W,\nabla\equiv0)$ is a pseudo-Dirac structure. 
\end{example}

\begin{example}\label{ex:BarLieSub}
If $(W,\nabla)$ is a pseudo-Dirac structure for the Courant algebroid $\mbb{E}$, then $(W,-\nabla)$ is a pseudo-Dirac structure for $\overline{\mbb{E}}$.
\end{example}

\begin{example}
If $\mf{d}$ is a quadratic Lie algebra, then any Lie subalgebra $\g\subset\mf{d}$ is a pseudo-Dirac structure.
\end{example}


\begin{example}[Action Courant algebroids]\label{ex:LieSubAct}
Suppose $\mf{d}$ is a quadratic Lie algebra which acts on a manifold $M$ with coisotropic stabilizers. In this case, as explained in \cite{LiBland:2009ul}, the bundle $\mf{d}\times M$ is naturally a Courant algebroid (see \cref{ex:ActCourAlg} for details). If $\h\subseteq\mf{d}$ is any subalgebra, then $(\h\times M,\d)$ is a pseudo-Dirac structure in $\mf{d}\times M$, where $\d$ is the standard connection on the trivial bundle $\mf{d}\times M$. As a Lie algebroid, $(\h\times M,\d)$ is isomorphic to the action Lie algebroid, as can be seen by comparing \labelcref{eq:pcModBrk} with \labelcref{eq:actioncourant}.
\end{example}

The following is a special case of the last example.
\begin{example}[Dirac Lie groups]
Dirac Lie groups for which multiplication is a morphism of Manin pairs were classified in \cite{LiBland:2010wi, LiBland:2011vqa}  (see also \cite{Ortiz:2008bd, Jotz:2009va} for a different setting). There it was shown that the underlying Courant algebroid can be canonically trivialized as an action Courant algebroid  $\mbb{A}=\mf{d}\times H$ (see \cref{ex:ActCourAlg}), and the Dirac structure is a constant subbundle  $E=\g\times H$ under this trivialization.

As such, both $(\mf{d}\times H,\d)$ and $(\g\times H,\d)$  define pseudo-Dirac structures in $\mbb{A}$. Moreover, if $\mf{r}\subset\mf{q}$ is the Lie subalgebra transverse to $\g$ described in \cite[Section~3.2]{LiBland:2011vqa}, then $(\mf{r}\times H,\d)$ describes a pseudo-Dirac structure which does not correspond to any Dirac structure in $\mbb{A}$. 
\end{example}

\begin{example}[Cotangent Lie algebroids]
Suppose that $T^*M$ carries the structure of a Lie algebroid with anchor map $\mbf{a}':T^*M\to TM$. Then $W=\gr(\mbf{a}')\subset \mbb{T}M$ is a pseudo-Dirac structure, where the pseudo-connection is defined by \cref{eq:pcModBrk}. That is
$$\la\nabla\sigma,\tau\ra=\Lied_{\mbf{a}'(\sigma)}j(\tau)-\iota_{\mbf{a}'(\tau)}dj(\sigma)-j([\sigma,\tau]),$$
for $\sigma,\tau\in\Gamma(W)$, where $j:W\to T^*M$ is the inverse of $( \mbf{a}'\oplus \on{id}):T^*M\to W$.

This example arises in q-Poisson geometry. \v{S}evera and the author \cite{LiBland:2010wi} showed that given any q-Poisson $(\g\oplus\bar\g,\g_\Delta)$-structure on a manifold $M$ (see \cref{ex:qPCotLie} for details), $T^*M$ carries the structure of a Lie algebroid. Of significance is that $T^*M$ cannot generally be identified with a Dirac structure in any exact Courant algebroid, which contrasts the case of a Poisson structure on $M$ (cf. \cref{ex:StdCourAlgPoisSymp}). However, both Poisson and q-Poisson structures on $M$ endow $T^*M$ with the structure of pseudo-Dirac structure, as we shall soon explain in more detail in \cref{ex:qPCotLie}.
\end{example}

\begin{example}\label{ex:PseudoAndTwists}
Suppose $(W,\nabla)$ is a pseudo-Dirac structure in $\mbb{T}M$ and $\gamma\in\Omega^3_{cl}(M)$ is a closed 3-form, then $(W,\nabla')$ is a pseudo-Dirac structure in $\mbb{T}_\gamma M$, where 
$$\la\nabla'\sigma,\tau\ra=\la\nabla\sigma,\tau\ra+\iota_{\mbf{a}(\sigma)}\iota_{\mbf{a}(\tau)}\gamma,$$
and $\mbb{T}_\gamma M$ is as in \cref{ex:exactCA}.
\end{example}

\begin{example}\label{ex:MetCon}
Let $g\in S^2(TM)$ be a non-degenerate quadratic form, defining a pseudo-Riemannian structure on $M$, and let $\tilde\nabla$ be the corresponding Levi-Civita connection.
 Suppose $\omega\in\Omega^2(M)$ is a 2-form. 
 We will consider the subbundle 
$$W=\gr(\omega^\flat+g^\flat)=\{(v,g(v,\cdot)+\iota_v\omega)\mid v\in TM\}\subseteq \mbb{T}M.$$
For any closed 3-form $\eta\in\Omega^3_{cl}(M)$, We define the pseudo-connection $\nabla^\eta:\Omega^0(W)\to\Omega^1(W^*)$  by
$$\nabla^\eta_Z X=\tilde\nabla_Z X-(g^\flat)^{-1}(\frac{1}{2}\iota_Z\iota_X \eta),\quad X,Z\in\mf{X}(M)$$
where we have used the anchor map to identify $W\cong TM$ and the pseudo-Riemannian structure to identify $TM\cong W^*$. (This is the unique metric connection on $TM\cong W$ which has torsion $T(X,Y,Z)=2\iota_Z\iota_Y\iota_X\eta$ (cf. \cite[Theorem 6.26]{Gualtieri:2004wh} or \cite[Proposition XII]{Greub:1973vs}).)

We claim that $(W,\nabla^{d\omega})$ is a pseudo-Dirac structure in $\mbb{T}M$. It follows from \cref{ex:PseudoAndTwists}  that for any closed 3-form $\gamma\in\Omega^3_{cl}(M)$, 
 $$(\gr(\omega^\flat+g^\flat),\nabla^{d\omega-\gamma})$$
  is a pseudo-Dirac structure in $\mbb{T}_\gamma M$. 
  
  To prove this claim first note that the pseudo-connection defined above is uniquely determined by requiring that $\Gamma(W)$ be involutive with respect to the modified bracket \labelcref{eq:pcModBrk}. Once one uses \cref{rem:PsiAsCurv} to express $\Psi$ in terms of the curvature and torsion of the connection, the fact that $\Psi$ vanishes is equivalent to the Bianchi identity.
  
  \begin{remark}
  When $g$ is positive definite, these structures play an important role in generalized-K\"{a}hler and bi-hermitian geometry as explained by Gualtieri \cite[Theorem 6.28]{Gualtieri:2004wh}: a generalized-K\"{a}hler structure on $M$ is equivalent to 
  a pair of almost complex structures $J^\pm:TM\to TM$ which are  compatible with a pseudo-Dirac structure, $(\gr(\omega^\flat+g^\flat),\nabla^{d\omega})$, in the following sense:
  \begin{itemize}
  \item$g(J^\pm\cdot,J^\pm\cdot)=g(\cdot,\cdot)$,
  \item$d\omega$ is of type $(2,1)+(1,2)$ with respect to $J^\pm$, and
  \item$\nabla^{\pm d\omega}J^\pm=0.$
  \end{itemize}
  \end{remark}


%
\end{example}

In \cref{sec:LieSubAlg}, we shall explain the relationship between pseudo-Dirac structures in $\mbb{E}$ and Dirac structures in $T\mbb{E}$. This will give us a more conceptual framework with which to work with pseudo-Dirac structures. Using this framework, we will describe a number of important examples of pseudo-Dirac structures in \cref{sec:FBimage}.

\section{Relation to the tangent prolongation of the Courant algebroid}\label{sec:LieSubAlg}
The tangent prolongation, $T\mbb{E}\to TM$ of a Courant algebroid $\mbb{E}\to M$ is an example of a \emph{double vector bundle}, a structure introduced by Pradines \cite{Pradines:1974tc}. In the next subsection, we recall some important properties of double vector bundles, before explaining the relationship between pseudo-Dirac structures in $\mbb{E}$ and Dirac structures in $T\mbb{E}$.  

\subsection{Double structures}
In this section, we recall the concepts of a double vector bundle and an $\mc{LA}$-vector bundle. We define both structures here to highlight their conceptual similarity.

\begin{definition}\label{def:DLACAVB}
Suppose that $D\to A$ and $B\to M$ are vector bundles, and
\begin{equation}\label{eq:DVB}\begin{tikzpicture}
\mmat{m}{
D&B\\
A&M\\
};
\path[->] (m-1-1) edge  (m-1-2);
\path[->] (m-1-1) edge  (m-2-1);
\path[->] (m-1-2) edge  (m-2-2);
\path[->] (m-2-1) edge  (m-2-2);
\end{tikzpicture}\end{equation}
is a morphism of vector bundles. We let $\gr(+_{D/A})\subseteq D\times D\times D$ denote the graph of addition for $D\to A$, as in \cref{ex:GrAdd}.
\begin{itemize}
\item $D$ is called a \emph{double vector bundle} if $D$ is a vector bundle over $B$ and $\gr(+_{D/A})$ is a vector subbundle of $D^3\to B^3$.
\item $D$ is called an \emph{$\mc{LA}$-vector bundle} if $D$ is a Lie algebroid over $B$ and $\gr(+_{D/A})\subseteq D^3$ is Lie subalgebroid.
\end{itemize}

\end{definition}
%


\subsubsection{Double vector bundles}
In essence, double vector bundles are vector bundles in the category of vector bundles. They were first introduced by Pradines \cite{Pradines:1974tc} and further studied in \cite{Mackenzie:2005tc,Konieczna:1999vh,Grabowski:2009dc}.

\begin{example}\label{ex:DirectSumDVB}
Suppose that $A$, $B$ and $C$ are all vector bundles over the manifold $M$. Then $A\times_M B\times _M C$ is the total space of a vector bundle over $A$ and of a vector bundle over $B$ with the respective additions given by
$$(a,b_1,c_1)+_{D/A}(a,b_2,c_2)=(a,b_1+b_2,c_1+c_2),$$ and
$$(a_1,b,c_1)+_{D/B}(a_2,b,c_2)=(a_1+a_2,b,c_1+c_2),$$ where $a,a_1,a_2\in A$, $b,b_1,b_2\in B$ and $c,c_1,c_2\in C$ all lie over the same point in $M$. With these vector bundle structures, 
$$\begin{tikzpicture}
\mmat{m}{
A\times_M B\times _M C&B\\
A&M\\
};
\path[->] (m-1-1) edge  (m-1-2);
\path[->] (m-1-1) edge  (m-2-1);
\path[->] (m-1-2) edge  (m-2-2);
\path[->] (m-2-1) edge  (m-2-2);
\end{tikzpicture}$$
is a double vector bundle. All double vector bundles are (non-canonically) isomorphic to a double vector bundle of this form.
\end{example}

\begin{example}[Tangent bundle of a vector bundle]\label{ex:TngDVB}
The following example, which will be of central importance to this paper, is due to Pradines \cite{Pradines:1974tc}.
Suppose that $B\to M$ is a vector bundle, then 
\begin{equation}\label{eq:TngDVB}\begin{tikzpicture}
\mmat{m}{
TB&B\\
TM&M\\
};
\path[->] (m-1-1)	edge (m-1-2)
				edge (m-2-1);
\path[<-] (m-2-2)	edge  (m-1-2)
				edge (m-2-1);
\end{tikzpicture}\end{equation}
is a double vector bundle. The graph of the addition in the vector bundle $TB\to TM$ is obtained by applying the tangent functor $$\gr(+_{TB/TM})=T\gr(+{B/M})$$ to the graph of the addition in the vector bundle $B\to M$. We will revisit this example in more detail in \cref{ex:TngDVB2,ex:TngLAVB}  below.
\end{example}

 \cref{def:DLACAVB} shows the similarity between the notion of double vector bundles and $\mc{LA}$-vector bundles. Moreover, it facilitates certain calculations. However some of the properties of double vector bundles are more readily apparent from the following standard definition \cite{Pradines:1974tc,Mackenzie:2005tc,Mackenzie05}, for which the roles of the two vector bundle structures on $D$ are manifestly symmetric.
\begin{definition}[\cite{Pradines:1974tc,Mackenzie:2005tc,Mackenzie05}]\label{def:AltDVBdef}A double vector bundle
 $$\begin{tikzpicture}
\mmat{m}{
D&B\\
A&M\\
};
\path[->] (m-1-1) edge  (m-1-2);
\path[->] (m-1-1) edge  (m-2-1);
\path[->] (m-1-2) edge  (m-2-2);
\path[->] (m-2-1) edge  (m-2-2);
\end{tikzpicture}$$
is a commutative diagram of vector bundles for which the following equations hold
\begin{align}\label{eq:DVBinterchg}(d_1+_{D/B} d_2)+_{D/A} (d_3+_{D/B} d_4)&=(d_1+_{D/A} d_3)+_{D/B}(d_2+_{D/A} d_4),\\
 t\cdot_{D/A}(d_1+_{D/B} d_2)&=t\cdot_{D/A} d_1+_{D/B} t\cdot_{D/A} d_2,\\
 t\cdot_{D/B}(d_1+_{D/A} d_3)&=t\cdot_{D/B} d_1 +_{D/A} t\cdot_{D/B} d_3,\end{align}
for any $d_1,d_2,d_3,d_4\in D$  satisfying $(d_1,d_2)\in D\times_B D$, $(d_3,d_4)\in D\times_B D$, and $(d_1,d_3)\in D\times_A D$, $(d_2,d_4)\in D\times_A D$, and any $t\in\mbb{R}$.

Here $+_{D/A}$ and $+_{D/B}$ denote the additive operations on $D$, viewed as a vector bundle over $A$ and $B$, respectively, while $\cdot_{D/A}$ and $\cdot_{D/B}$ denote the scalar multiplication operations.
\end{definition}
The proof that these two definitions are equivalent is the content of \cite[Corollary 2.5.1]{LiBland:2012un}.

We refer to the vector bundles $A\to M$ and $B\to M$ as the \emph{horizontal} and \emph{vertical side bundles}, respectively. We let $0_{D/A}:A\to D$ and $0_{D/B}:B\to D$ denote the zero sections, and $\tilde 0_{D/M}:M\to D$ be the composition of $0_{D/B}:B\to D$ with the zero section $M\to B$.

We let $D^{flip}$
denote the reflection of \labelcref{eq:DVB} across the diagonal, 
$$\begin{tikzpicture}
\mmat{m}{
D&A\\
B&M\\
};
\path[->] (m-1-1) edge node {$q_{D/A}$} (m-1-2);
\path[->] (m-1-1) edge node[swap] {$q_{D/B}$} (m-2-1);
\path[->] (m-1-2) edge node {$q_{A/M}$} (m-2-2);
\path[->] (m-2-1) edge node[swap] {$q_{B/M}$} (m-2-2);
\end{tikzpicture}$$
which, as is apparent from \cref{def:AltDVBdef}, is also a double vector bundle.

\subsubsection{The core}
Consider the submanifold $C:=\on{ker}(q_{D/A})\cap\on{ker}(q_{D/B})$.
If in \cref{eq:DVBinterchg}, we let $d_1,d_4\in C$ and $d_2,d_3=\tilde 0_{D/M}$, then we get 
\begin{align*}
d_1+_{D/A} d_4 =& (d_1+_{D/B} \tilde 0_{D/M})+_{D/A}(\tilde 0_{D/M}+_{D/B} d_4)\\
=&(d_1+_{D/A} \tilde 0_{D/M})+_{D/B} (\tilde 0_{D/M}+_{D/A} d_4)\\
=&d_1+_{D/B} d_4.
\end{align*} 
So both $+_{D/A}$ and $+_{D/B}$ define the same additive structure on $C$. Similarly both $\cdot_{D/A}$ and $\cdot_{D/B}$ restrict to the same scalar multiplication on $C$. Therefore, with either choice of addition and scalar multiplication, $C$ is a vector bundle over $M$, called the \emph{core} of $D$. 

We may occasionally display the core explicitly in a double vector bundle diagram, as follows:
$$\begin{tikzpicture}
\mmat{m}{
D&B\\
A&M\\
};
\path[->] (m-1-1) edge  (m-1-2);
\path[->] (m-1-1) edge  (m-2-1);
\path[->] (m-1-2) edge  (m-2-2);
\path[->] (m-2-1) edge  (m-2-2);
\draw (0,0) node (c) {$C$} ;
\path[left hook->] (c) edge (m-1-1);
\end{tikzpicture}$$

\begin{example}
Suppose $A,B$ and $C$ are vector bundles over $M$. The core in the double vector bundle
$$\begin{tikzpicture}
\mmat{m}{
A\times_M B\times _M C&B\\
A&M\\
};
\path[->] (m-1-1) edge  (m-1-2);
\path[->] (m-1-1) edge  (m-2-1);
\path[->] (m-1-2) edge  (m-2-2);
\path[->] (m-2-1) edge  (m-2-2);
\end{tikzpicture}$$
 (c.f \cref{ex:DirectSumDVB}) is the subbundle $C\subseteq A\times_M B\times _M C$.
\end{example}
%

We let $i:C\to D$ denote the inclusion and 
$i_A:q_{A/M}^*C\to D$
denote  the composition 
$$\begin{tikzpicture}
\mmat[5em]{m}{A\times_M C&D\times_B D&D.\\};
\path[->] (m-1-1) edge node {$ 0_{D/A}\times i$} (m-1-2);
\path[->] (m-1-2) edge node {$+_{D/B}$} (m-1-3);
\end{tikzpicture}$$
The exact sequence of vector bundles 
\begin{equation}\label{eq:iA}\begin{tikzpicture}
\mmat{m}{q_{A/M}^* C & D& q_{A/M}^* B\\ A& A& A\\};
\path[->] (m-1-1) 	edge node {$i_A$} (m-1-2)
				edge (m-2-1);
\path[->] (m-1-2)	edge (m-1-3)
				edge (m-2-2);
\path[->] (m-1-3)	edge (m-2-3);
\path[->] (m-2-1)	edge node[swap] {$\on{id}$} (m-2-2);
\path[->] (m-2-2)	edge node[swap] {$\on{id}$} (m-2-3);
\end{tikzpicture}\end{equation} is called the \emph{core sequence}.

Among the sections of $D\to A$, $\Gamma(D,A)$, there are two subspaces of sections: the \emph{core} sections $\Gamma_C(D,A)$, and the \emph{linear} sections $\Gamma_l(D,A)$. 
For any section $\sigma\in\Gamma(C)$, we let  the section $\sigma_{c_A}:A\to D$ be given by
$$\begin{tikzpicture}
\mmat{m}{A&&q_{A/M}^*C&D\\};
\path[->] (m-1-1) edge node {$q_{A/M}^*\sigma$} (m-1-3);
\path[->] (m-1-3) edge node {$i_A$} (m-1-4);
\end{tikzpicture}$$
The map $\sigma\to\sigma_{c_A}$ embeds $\Gamma(C)$ into $\Gamma(D,A)$, as the space $\Gamma_C(D,A)$ of core sections.

Meanwhile a section $\sigma\in\Gamma(D,A)$ is called linear, if there is a section $\sigma_0\in\Gamma(B)$ such that the vertical map
$$\begin{tikzpicture}
\mmat{m}{
D&B\\
A&M\\
};
\path[->] (m-2-1) edge node {$\sigma$} (m-1-1);
\path[->] (m-1-1) edge (m-1-2);
\path[->] (m-2-1) edge (m-2-2);
\path[->] (m-2-2) edge node {$\sigma_0$} (m-1-2);
\end{tikzpicture}$$
is a morphism of vector bundles (in this case, $\sigma_0=q_{D/B}\circ \sigma\circ 0_{A/M}$ is unique). The subspace of linear sections is denoted by $\Gamma_l(D,A)\subset\Gamma(D,A)$. Moreover, as explained in \cite[Section 2.4]{gracia2010lie} $\Gamma_l(D,A)$ is a locally free $C^\infty(M)$ module which fits into the short exact sequence \begin{equation}\label{eq:ShrtLinSeq}0\to \Gamma(A^*\otimes C)\to\Gamma_l(D,A)\to\Gamma(B)\to0.\end{equation} Here the injection $\Gamma(A^*\otimes C)\to\Gamma_l(D,A)$ maps $\tau\in \Gamma(A^*\otimes C)$ to the section $$a\to i_A(a,\tau(a)), \quad a\in A,$$ while the surjection $\Gamma_l(D,A)\to\Gamma(B)$ maps $\sigma\in\Gamma_l(D,A)$ to $q_{D/B}\circ \sigma\circ 0_{A/M}$. In the sequel, we will abuse notation and denote this surjection simply by $$q_{D/B}:\Gamma_l(D,A)\to\Gamma(B).$$

It is clear that there are analogous notions with the roles of $A$ and $B$ replaced.

\begin{example}[Tangent bundle of a vector bundle (cont.)]\label{ex:TngDVB2}
Suppose that $B\to M$ is a vector bundle, then as explained in \cref{ex:TngDVB},
$$\begin{tikzpicture}
\mmat{m}{
TB&B\\
TM&M\\
};
\path[->] (m-1-1)	edge (m-1-2)
				edge (m-2-1);
\path[<-] (m-2-2)	edge node[swap] {$q_{B/M}$} (m-1-2)
				edge node {$q_{TM/M}$}(m-2-1);
\end{tikzpicture}$$
is a double vector bundle. For a point $x\in M$, the core fibre over $x$ consists of vectors based at $x$ which are tangent to the fibre, $q_{B/M}^{-1}(x)$, of $B$ over $x$ (see \cref{fig:TngCore}). So the core is canonically isomorphic to $B$:
$$\begin{tikzpicture}
\mmat{m}{
TB&B\\
TM&M\\
};
\path[->] (m-1-1)	edge (m-1-2)
				edge (m-2-1);
\path[<-] (m-2-2)	edge (m-1-2)
				edge (m-2-1);
\draw (0,0) node (c) {$B$};
\path[left hook->] (c) edge (m-1-1);
\end{tikzpicture}$$
In the sequel, for $\sigma\in\Gamma(B)$, we will typically denote $\sigma_{c_{TM}}\in\Gamma_C(TB,TM)$ by $\sigma_C$ to simplify notation. 

\begin{figure}
\begin{tabular}{cc}
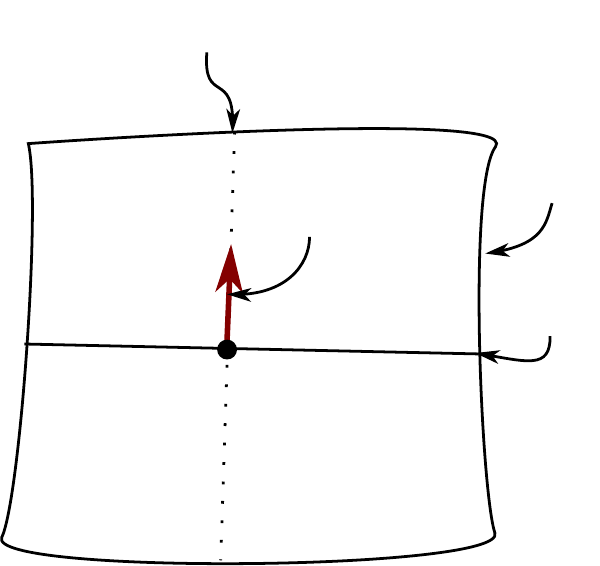 & 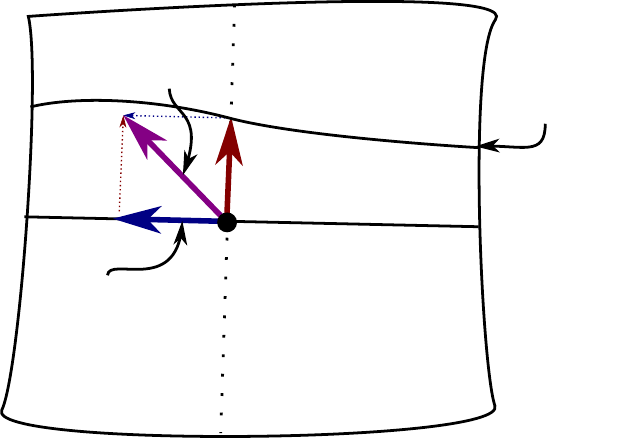
\end{tabular}
\caption{\label{fig:TngCore}The vector $\xi\in T_xB$ is tangent to the fibre $q_{B/M}^{-1}(x)$, and therefore defines a core element of $TB$ at $x\in M$. If $\sigma\in\Gamma(B)$ is such that $\xi=\sigma(x)$, then for any $X\in TM$, $\sigma_C(X)=X+_{TB/B}\xi$, (here the addition takes place in the vector space $T_xB$).}
\end{figure}

Meanwhile,  the differential $\d\sigma:TM\to TB$ is a linear section of $TB\to TM$ canonically associated to the section $\sigma:M\to B$, which we call the tangent lift $\sigma_T$ of $\sigma$ (see \cref{fig:TngLiftSec}). The map $$\Gamma(B)\xrightarrow{\sigma\to\sigma_T}\Gamma_l(TB,TM)$$ defines a splitting of the exact sequence \cref{eq:ShrtLinSeq} (but not of the underlying vector bundles).

\begin{figure}
\begin{center}
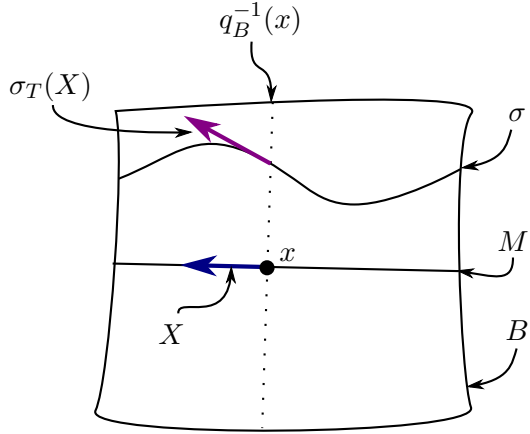
\end{center}
\caption{\label{fig:TngLiftSec}For any section $\sigma\in\Gamma(B)$, the tangent lift $\sigma_T\in\Gamma(TB,TM)$ of $\sigma$ takes $X\in TM$ to $\d\sigma(X)\in TB$.}
\end{figure}

For $f\in C^\infty(M)$, we introduce the notation $f_C:=q_{TM/M}^*f\in C^\infty(TM)$ and $f_T:=\d\! f\in C^\infty(TM)$.\footnote{Here, we understand the 1-form, $\d\! f\in\Omega^1(M)$, as defining a function on $TM$.} Then we have the following rules 
\begin{subequations}\label{eq:TCDer}
\begin{align}
(f\cdot\sigma)_C&=f_C\cdot\sigma_C,\\
(f\cdot\sigma)_T&=f_T\cdot\sigma_C+f_C\cdot\sigma_T.
\end{align}\end{subequations}

Finally, we point out that if $\{\sigma^i\}\subset\Gamma(B)$ is a local basis of sections of $B$, then $\{\sigma^i_C,\sigma^i_T\}\subset\Gamma(TB,TM)$ is a local basis of sections of $TB$.

\end{example}

\subsubsection{Duals of double vector bundles}
A double vector bundle
\begin{equation}\label{eq:DVBdualpic}\begin{tikzpicture}
\mmat{m}{
D&B\\
A&M\\
};
\path[->] (m-1-1) edge  (m-1-2);
\path[->] (m-1-1) edge  (m-2-1);
\path[->] (m-1-2) edge  (m-2-2);
\path[->] (m-2-1) edge  (m-2-2);
\draw (0,0) node (c) {$C$} ;
\path[left hook->] (c) edge (m-1-1);
\end{tikzpicture}\end{equation}
is the total space for two ordinary vector bundles: 
$$\begin{tikzpicture}
\draw (-3,0) node (d) {$D$};
\draw (-2,0) node (b) {$B$};
\draw (2,.5) node (d') {$D$};
\draw (2,-.5) node (a) {$A$};
\draw (0,0) node {and};
\path[->] (d) edge (b);
\path[->] (d') edge (a);
\end{tikzpicture}$$
 We denote the duals of these two ordinary vector bundles by
$$\begin{tikzpicture}
\draw (-3,0) node (d) {$D^{*_{\!x}}$};
\draw (-2,0) node (b) {$B$};
\draw (2,.5) node (d') {$D^{*_{\!y}}$};
\draw (2,-.5) node (a) {$A$};
\draw (0,0) node {and};
\path[->] (d) edge (b);
\path[->] (d') edge (a);
\end{tikzpicture}$$
The notation $*_{\!x}$ and $*_{\!y}$, which we have adapted from \cite{GraciaSaz:2009ck},\footnote{Gracia-Saz and Mackenzie denote $*_{\!x}$ and $*_{\!y}$ by $X$ and $Y$, respectively.} identifies the arrow along which we dualize with the appropriate axis:
$$\begin{tikzpicture}
\path[->] (0,0) edge node {$x$} (1,0)
		edge node[swap] {$y$} (0,-1);
\end{tikzpicture}$$
We call $D^{*_{\!y}}$  the \emph{vertical dual} or \emph{dual over $A$}, and $D^{*_{\!x}}$ the \emph{horizontal dual} or \emph{dual over $B$}.

Pradines \cite{Pradines:1988td}, and later Konieczna-Urba\'{n}ski \cite{Konieczna:1999vh}, Mackenzie \cite{Mackenzie:2005tc,Mackenzie:1999vk} and  Grabowski-Rotkiewicz \cite{Grabowski:2009dc}, studied the duals $D^{*_{\!x}}$ and $D^{*_{\!y}}$, proving that they form the total space for double vector bundles themselves. 
\begin{proposition}[\cite{Konieczna:1999vh,Mackenzie:1999vk}]
\begin{equation}\label{eq:DVBDual}\begin{tikzpicture}
\mmat{m}{D^{*_{\!y}}&C^*\\ A& M\\};
\path[->] (m-1-1)	edge (m-1-2)
				edge (m-2-1);
\path[<-] (m-2-2)	edge (m-1-2)
				edge (m-2-1);
\draw (0,0) node (b) {$B^*$};
\path[left hook->] (b) edge (m-1-1);
\end{tikzpicture}\end{equation}
is a double vector bundle, where the inclusion $B^*\to D^{*_{\!y}}$ and the projection $D^{*_{\!y}}\to C^*$ are given by dualizing the core sequence \labelcref{eq:iA}.
Addition and scalar multiplication for the vector bundle $ D^{*_{\!y}}\to C^*$ is defined by the relations \begin{align}\gr(+_{D^{*_{\!y}}/C^*})&:=\ann^\natural(\gr(+_{D/B})):D^{*_{\!y}}\times D^{*_{\!y}}\dasharrow D^{*_{\!y}}\\
\gr(t\cdot_{D^{*_{\!y}}/C^*})&:=\ann^\natural(\gr(t\cdot_{D/B})):D^{*_{\!y}}\dasharrow D^{*_{\!y}},\quad t\in\mbb{R}.\end{align} 
\end{proposition}
\begin{proof}
This follows from \cref{def:DLACAVB} and \cite[Lemma A.2.]{LiBland:2011vqa}.
\end{proof}
A similar proposition holds for the dual $D^{*_{\!x}}\to B$.

\begin{example}[Cotangent bundle of a vector bundle]\label{ex:CotDVB}
Taking the dual of the tangent double vector bundle \labelcref{eq:TngDVB} over $B$, and using the identification $(TB)^{*_{\!x}}\cong T^*B$, we get the cotangent double vector bundle
\begin{equation}\label{eq:CotDVB}\begin{tikzpicture}
\mmat{m}{
T^*B&B\\
B^*&M\\
};
\path[->] (m-1-1)	edge (m-1-2)
				edge (m-2-1);
\path[<-] (m-2-2)	edge (m-1-2)
				edge (m-2-1);
\draw (0,0) node (c) {$T^*M$};
\path[left hook->] (c) edge (m-1-1);
\end{tikzpicture}\end{equation}

\end{example}

\begin{example}\label{ex:SecDualTng}
If we take the dual of the tangent double vector bundle \labelcref{eq:TngDVB} over $TM$ instead, the result is canonically isomorphic to 
$$\begin{tikzpicture}
\mmat{m}{
TB^*&B^*\\
TM&M\\
};
\path[->] (m-1-1)	edge (m-1-2)
				edge (m-2-1);
\path[<-] (m-2-2)	edge (m-1-2)
				edge (m-2-1);
\draw (0,0) node (c) {$B^*$};
\path[left hook->] (c) edge (m-1-1);
\end{tikzpicture}$$
That is, $(TB)^{*_{\!y}}=T(B^*)$.
\end{example}

\subsubsection{Linear connections and double vector bundles}

An Ehresmann connection on a fibre bundle $B\xrightarrow{q_{B/M}} M$ is defined to be a subbundle $K\subseteq TB$ such that $\d q_{B/M}:K\to TM$ is a fibrewise isomorphism \cite{Ehresmann:1995ts}. Inverting this isomorphism defines an injection $\lambda_K:\mf{X}(M)\to \mf{X}(B)$, i.e. a way of lifting vector fields on $M$ to vector fields on $B$. When $B\to M$ is a vector bundle, then $K$ defines a linear connection precisely when $K\to TM$ is a vector subbundle of $TB\to TM$, i.e. 
$$\begin{tikzpicture}
\mmat{m1} at (-2,0){K&B\\ TM& M\\};
\path[->] (m1-1-1)	edge (m1-1-2)
				edge (m1-2-1);
\path[<-] (m1-2-2)	edge (m1-1-2)
				edge (m1-2-1);
\mmat{m2} at (2,0) {TB&B\\ TM& M\\};
\path[->] (m2-1-1)	edge (m2-1-2)
				edge (m2-2-1);
\path[<-] (m2-2-2)	edge (m2-1-2)
				edge (m2-2-1);
\draw (0,0) node {$\subseteq$} (m2);
\draw (2,0) node (c2) {$B$};
\path[left hook->] (c2) edge (m2-1-1);
\end{tikzpicture}
$$
is a double vector subbundle \cite{Konieczna:1999vh,Pradines:1974tc}. More generally, we have
\begin{proposition}\label{prop:DVSBandConnect}
Suppose $S\subseteq M$ is a submanifold, $W\to S$ and $C\to S$ are vector subbundles of $B\rvert_S\to S$, and $E\to S$ is a vector subbundle of $TM\rvert_S\to S$.
Let $$\widetilde{\Gamma(W)}:=\{\sigma\in\Gamma(B)\text{ such that } \sigma\rvert_S\in\Gamma(W)\}.$$
Then double vector subbundles 
\begin{equation}\label{eq:DVBforConn}\begin{tikzpicture}
\mmat[5em]{m1} at (-2.5,0){K&W\\ E& S\\};
\path[->] (m1-1-1)	edge node {$q_{K/W}$} (m1-1-2)
				edge node[swap] {$q_{K/E}$} (m1-2-1);
\path[<-] (m1-2-2)	edge (m1-1-2)
				edge (m1-2-1);
\mmat[5em]{m2} at (3.5,0) {TB&B\\ TM& M\\};
\path[->] (m2-1-1)	edge node {$q_{TB/B}$} (m2-1-2)
				edge node[swap] {$q_{\scriptscriptstyle TB/TM}$} (m2-2-1);
\path[<-] (m2-2-2)	edge (m2-1-2)
				edge (m2-2-1);
\draw (0,0) node {$\subseteq$} (m2);
\draw (-2.5,0) node (c) {$C$};
\path[left hook->] (c) edge (m1-1-1);
\draw (3.5,0) node (c2) {$B$};
\path[left hook->] (c2) edge (m2-1-1);
\end{tikzpicture}
\end{equation}
are in one-to-one correspondence with linear maps \begin{equation}\label{eq:NabForK}\nabla:\widetilde{\Gamma(W)}\to \Gamma(E^*\otimes \frac{B\rvert_S}{C})\end{equation} satisfying the Leibniz rule 
\begin{equation}\label{eq:NabForKLeib}\nabla f\sigma=f\nabla\sigma+\d\!f\otimes \sigma,\end{equation}
 for any $f\in C^\infty(M)$.
\end{proposition}

\begin{remark}As pointed out by a referee, when $W=B$,
then $K$ is a \emph{linear distribution} on $B$, and a similar one-to-one correspondence is described in Drummond-Jotz-Ortiz \cite{Drummond:2013wr}.

When $W=B$ and $E=TM$, then $K$ can be seen as a multiplicative distribution on the linear groupoid $B\rightrightarrows M$ (the vector bundle seen as a Lie groupoid). In this case, $\nabla$ is the Spencer operator studied by Crainic-Salazar-Struchiner \cite{Crainic:2012th,Salazar:2013tk}.  Jotz-Ortiz also study such \emph{foliated groupoids} in \cite{Jotz:2011wq}, where they classify them in terms of similar connection type data.

If, additionally, $C$ is trivial, then \cref{prop:DVSBandConnect} reduces to the usual relationship between linear Ehresmann connections on $B$ and covariant differential operators $\nabla:\Gamma(B)\to \Omega^1(M,B)$.
\end{remark}

The proof of this proposition is essentially the same as for linear Ehresmann connections, but we include it here for the curious reader.

\begin{proof}
Suppose first that \labelcref{eq:DVBforConn} is a double vector bundle. The space $$TB\rvert_{\scriptscriptstyle E\!\times_{\!S} \!W}:=q_{TB/B}^{-1}(W)\cap q_{TB/TM}^{-1}(E)$$ is a double vector bundle with the same side bundles as $K$, but with a different core:
$$\begin{tikzpicture}
\mmat[4em]{m1} at (-2.5,0){K&W\\ E& S\\};
\path[->] (m1-1-1)	edge  (m1-1-2)
				edge  (m1-2-1);
\path[<-] (m1-2-2)	edge (m1-1-2)
				edge (m1-2-1);
\mmat[4em]{m2} at (2.5,0) {TB\rvert_{\scriptscriptstyle E\!\times_{\!S} \!W}&W\\ E& S\\};
\path[->] (m2-1-1)	edge (m2-1-2)
				edge (m2-2-1);
\path[<-] (m2-2-2)	edge (m2-1-2)
				edge (m2-2-1);
\draw (0,0) node {$\subseteq$} (m2);
\draw (-2.5,0) node (c) {$C$};
\path[left hook->] (c) edge (m1-1-1);
\draw (2.5,0) node (c2) {$B\rvert_S$};
\path[left hook->] (c2) edge (m2-1-1);
\end{tikzpicture}
$$
Thus the double quotient (in both the vertical and horizontal directions) of $TB\rvert_{\scriptscriptstyle E\!\times_{\!S} \!W}$ by $K$ is canonically isomorphic to the quotient of the core bundles, $B\rvert_S/C$ (see \cref{app:TotQuo,rem:DblQuot} for more details). The quotient map is a morphism of double vector bundles
\begin{equation}\label{eq:DVBtotQuo}\begin{tikzpicture}
\mmat[4em]{m1} at (-2.5,0){TB\rvert_{\scriptscriptstyle E\!\times_{\!S} \!W}&W\\ E& S\\};
\path[->] (m1-1-1)	edge  (m1-1-2)
				edge  (m1-2-1);
\path[<-] (m1-2-2)	edge (m1-1-2)
				edge (m1-2-1);
\mmat[4em]{m2} at (2.5,0) {(B\rvert_S)/C&S\\ S& S\\};
\path[->] (m2-1-1)	edge (m2-1-2)
				edge (m2-2-1);
\path[<-] (m2-2-2)	edge (m2-1-2)
				edge (m2-2-1);
\draw (-2.5,0) node (c) {$B\rvert_S$};
\path[left hook->] (c) edge (m1-1-1);
\draw (2.5,0) node (c2) {$(B\rvert_S)/C$};
\path[left hook->] (c2) edge (m2-1-1);
\path[->] (m1) edge node {$q_K$}(m2);
\end{tikzpicture}
\end{equation}
canonically associated to $K$. In fact one may recover $K$ as the kernel of $q_K$. In this way, there is a one-to-one correspondence between double vector subbundles of the form \labelcref{eq:DVBforConn}, and morphisms of double vector bundles of the form \labelcref{eq:DVBtotQuo} whose restriction to the core
\begin{equation}\label{eq:qKCoreRest}q_K\rvert_{(B\rvert_S)}:B\rvert_S\to (B\rvert_S)/C\end{equation}
is the canonical projection.

As explained in \cref{ex:TngDVB2}, the space $\Gamma(TB,TM)$ is spanned by sections of the form $\sigma_T$ and $\tau_C$, for $\sigma,\tau\in\Gamma(B)$. Similarly, $\Gamma(TB\rvert_{\scriptscriptstyle E\!\times_{\!S} \!W},E)$ is spanned by sections of the form $\sigma_T\rvert_E$ and $\tau_C\rvert_E$ for $\sigma\in\widetilde{\Gamma(W)}$ and $\tau\in\Gamma(B)$. Now the composition of $q_K$ with any core section $\tau_C\rvert_E$ is just determined by the canonical projection \labelcref{eq:qKCoreRest}, and does not depend on $K$.
Therefore $q_K$ (and thus $K$) is entirely determined by the set of vector bundle maps 
$$\{q_K\circ \sigma_T\rvert_E:E\to (B\rvert_S)/C\text{ for }\sigma\in\widetilde{\Gamma(W)}\}.$$

We define $$\nabla:\widetilde{\Gamma(W)}\to \Gamma(E^*\otimes \frac{B\rvert_S}{C})$$ by the formula 
\begin{equation}\label{eq:NabFormFromK}\nabla\sigma:=q_K\circ\sigma_T\rvert_E,\quad \sigma\in\widetilde{\Gamma(W)},\end{equation}
 and we conclude that $K$ is entirely determined by $\nabla$.

Note that for any $f\in C^\infty(M)$, 
$$\nabla f\sigma=q_K\big((f\sigma)_T\big)=q_K\big(f_C\;\sigma_T+f_T\;\sigma_C\big)=f\nabla\sigma+\d\!f\otimes\sigma,$$
so $\nabla$ satisfies the requisite Leibniz rule, \cref{eq:NabForKLeib}.

 Conversely any map \labelcref{eq:NabForK} satisfying the Leibniz rule, \cref{eq:NabForKLeib}, arises from formula \labelcref{eq:NabFormFromK} for a unique double vector subbundle of the form \labelcref{eq:DVBforConn}, which proves the proposition.
\end{proof}

\subsubsection{$\mc{LA}$-vector bundles}
$\mc{LA}$-vector bundles are a concept due to Mackenzie \cite{Mackenzie:1998te,Mackenzie:1998ge} which encode the notion of a vector bundle in the category of Lie algebroids, or equivalently a Lie algebroid in the category of vector bundles. In \cite{gracia2010lie} it was shown that $\mc{LA}$-vector bundles are the correct context from which to study 2-term representations of Lie algebroids `up to homotopy' (note that \cite{gracia2010lie} also provides a very nice summary of $\mc{LA}$-vector bundle theory).

Suppose that 
\begin{equation}\label{eq:LAvb}\begin{tikzpicture}
\mmat{m}{D&B\\ A& M\\};
\path[->] (m-1-1)	edge (m-1-2)
				edge node[swap] {$q^D_A$} (m-2-1);
\path[<-] (m-2-2)	edge (m-1-2)
				edge (m-2-1);
\end{tikzpicture}\end{equation}
is an $\mc{LA}$-vector bundle. That is, $D\to A$ and $B\to M$ are vector bundles, \labelcref{eq:LAvb} is a morphism of vector bundles, $D\to B$ is a Lie algebroid, and $\gr(+_{D/A})\subset D^3$ is a Lie subalgebroid. 

\begin{proposition}\label{prop:AisLA}
There is a unique Lie algebroid structure on $A\to M$ such that the inclusion $0_{D/A}:A\to D$ and the projection $q_{D/A}:D\to A$ are both morphisms of Lie algebroids.
\end{proposition}
\begin{proof}

We let $\gr(-_{D/A}):D\times D\dasharrow D$ be the relation defined by $( d_1, d_2)\sim_{\gr(-_{D/A})} d$ if and only if $ d+_{D/A} d_2= d_1$. Since $\on{gr}(-_{D/A})=\{( d; d_1, d_2)\mid d+_{D/A} d_2= d_1\}\subseteq D\times D\times D$ is obtained from $\on{gr}(+_{D/A})$ by permuting factors, it is a Lie subalgebroid of $D^3\to B^3$. Moreover, the diagonal $D_\Delta\subset D\times D$ is Lie subalgebroid of $D^2\to B^2$. Therefore, the composition $\gr(-_{D/A})\circ D_\Delta$ is a Lie subalgebroid of $D\to B$. Since $0_{D/A}:A\to D$ embeds $A$ as $\gr(-_{D/A})\circ D_\Delta$, there exists a unique Lie algebroid structure on $A$ such that $0_{D/A}:A\to D$ is a morphism of Lie algebroids.

Moreover, since diagonal embedding $\Delta_D:D\to D\times D$, given by $ d\to( d, d)$ is a morphism of Lie algebroids from $D\to B$ to $D^2\to B^2$, the composition of relations $\gr(-_{D/A})\circ \gr(\Delta_D):D\dasharrow D$ is a $\mc{LA}$-relation. However $\gr(-_{D/A})\circ \gr(\Delta_D)=\gr(0_{D/A}\circ q_{D/A})$, which shows that $q_{D/A}:D\to A$ is a morphism of Lie algebroids.
\end{proof}

\begin{remark}
\Cref{prop:AisLA} shows that the Lie algebroid structure on $A$ is completely determined by the Lie algebroid structure on $D$. Indeed, if $\sigma_D,\tau_D\in\Gamma_l(D,B)$ and $\sigma_A,\tau_A\in\Gamma(A)$ are such that $\sigma_A=q_{D/A}(\sigma_D)$ and $\tau_A=q_{D/A}(\tau_D)$, then $$[\sigma_A,\tau_A]=q_{D/A}[\sigma_D,\tau_D].$$
\end{remark}

\begin{example}[Tangent prolongation]\label{ex:TngLAVB}
If $B$ is any vector bundle then writing the structural equations for $B$ as diagrams and applying the tangent functor we get corresponding diagrams in the category of Lie algebroids. For example, applying the tangent functor to the addition relation $\gr(+_{B/M}):B\times B\dasharrow B$ yields addition for $TB\to TM$, $\gr(+_{TB/TM}):=T\gr(+_{B/M}):TB\times_{TM} TB\to TB$. Since $T\gr(+_{B/M})\subset (TB)^3$ is a subalgebroid, 
 it follows directly from \cref{def:DLACAVB} that then the tangent double vector bundle \labelcref{eq:TngDVB}, $TB$ is a $\mc{LA}$-vector bundle. 

Meanwhile, if $A\to M$ is a Lie algebroid, then the flip, $(TA)^{flip}$, of the tangent double vector bundle is naturally an $\mc{LA}$-vector bundle, called the \emph{tangent prolongation} (or \emph{tangent lift}) of $A$.
$$\begin{tikzpicture}
\mmat{m}{
TA&TM\\
A&M\\
};
\path[->] (m-1-1)	edge (m-1-2)
				edge (m-2-1);
\path[<-] (m-2-2)	edge (m-1-2)
				edge (m-2-1);
\end{tikzpicture}$$
The Lie bracket on $\Gamma(TA,TM)$ can be defined as follows. For $\sigma,\tau\in\Gamma(A)$ we define
\begin{subequations}\label{eq:LATngPro}
\begin{align}
[\sigma_T,\tau_T]&=[\sigma,\tau]_T,\\
[\sigma_T,\tau_C]&=[\sigma,\tau]_C,\\
[\sigma_C,\tau_C]&=0.
\end{align}\end{subequations}
The bracket of arbitrary sections is defined by means of the Leibniz rule.

As an example, when $A=\g$ is a Lie algebra, then $(TA)^{flip}=\g\ltimes\g$, as a Lie algebra.
\end{example}

\subsection{The tangent prolongation of a Courant algebroid}\label{sec:TngPro}

Let $\mbb{E}$ be a Courant algebroid over $M$. In \cite{Boumaiza:2009eg}  Boumaiza and  Zaalani showed that $T\mbb{E}\to TM$ carries a canonical Courant algebroid structure.  In this section we recall this so-called \emph{tangent prolongation of Courant algebroids}.

Recall from \cref{ex:TngDVB2} that  any section $\sigma\in\Gamma(\mbb{E},M)$, defines two sections $\sigma_C,\sigma_T\in\Gamma(T\mbb{E},TM)$ called the core and tangent lift of $\sigma$, respectively. Note also that $\{\sigma^i_C,\sigma^i_T\}$ is a local basis for $T\mbb{E}\to TM$ whenever $\{\sigma^i\}$ is a local basis for $\mbb{E}\to M$.

\begin{proposition}\label{prop:TngProCA}
The tangent bundle $T\mbb{E}$ of a Courant algebroid $\mbb{E}\to M$,
$$\begin{tikzpicture}
\mmat{m}{T\mbb{E} &\mbb{E}\\ TM&M\\};
\path[->]
	(m-1-1) edge (m-1-2)
		edge (m-2-1);
\path[<-] 
	(m-2-2) edge (m-1-2)
		edge (m-2-1);
\end{tikzpicture}$$
carries a unique Courant algebroid structure over $TM$ such that the pairing and bracket satisfy
\begin{subequations}\label[pluralequation]{eq:TEPairBrk}
\begin{align}
\label{eq:TEPair}\la \sigma_T,\tau_T\ra&=\la\sigma,\tau\ra_T & \la\sigma_T,\tau_C\ra&=\la\sigma,\tau\ra_C &
\la\sigma_C,\tau_T\ra&=\la\sigma,\tau\ra_C & \la\sigma_C,\tau_C\ra&=0\\
\label{eq:TEBrk}\Cour{\sigma_T,\tau_T}&=\Cour{\sigma,\tau}_T & \Cour{\sigma_T,\tau_C}&=\Cour{\sigma,\tau}_C  & 
\Cour{\sigma_C,\tau_T}&=\Cour{\sigma,\tau}_C& \Cour{\sigma_C,\tau_C}&=0
\end{align}\end{subequations}
 and the anchor map satisfies
$$\mbf{a}(\sigma_T)=\mbf{a}(\sigma)_T \quad\quad \mbf{a}(\sigma_C) = \mbf{a}(\sigma)_C,$$
for any sections $\sigma,\tau\in\Gamma(\mbb{E},M)$.
\end{proposition}
\begin{proof}
The fact that $T\mbb{E}\to TM$ is a Courant algebroid follows  from \cite[Proposition A.0.3]{LiBland:2012un}. 
\end{proof}

\begin{remark}\label{rem:TangLiftTngProDef}
One may define the pairing and anchor map without reference to core and tangent lift sections. If 
$$\begin{tikzpicture}
\mmat{m1} at (-1.5,0) {\mbb{E} \\ M\\};
\path[->]
	(m1-1-1) edge (m1-2-1);
\mmat{m2} at (1.5,0) {TM\\ M\\};
\path[->]
	(m2-1-1) edge (m2-2-1);
		
\path[->,bend left= 15] (m1-1-1) edge node {$\mbf{a}$} (m2-1-1);
\end{tikzpicture}$$
 is the anchor map for $\mbb{E}\to M$, the anchor map for $T\mbb{E}\to TM$ is obtained by applying the tangent functor,
$$
\begin{tikzpicture}[cross line/.style={preaction={draw=white, -,
           line width=6pt}},]
\mmat{m1} at (-2.5,0) {T\mbb{E}&\mbb{E}\\ TM& M\\};
\path[->] (m1-1-1)	edge (m1-1-2)
				edge (m1-2-1);
\path[<-] (m1-2-2)	edge (m1-1-2)
				edge (m1-2-1);
\mmat{m2} at (2.5,0) {T^2M&TM\\ TM& M\\};
\path[->] (m2-1-1)	edge (m2-1-2)
				edge (m2-2-1);
\path[<-] (m2-2-2)	edge (m2-1-2)
				edge (m2-2-1);
				
\path[->] (m1) edge node {$\mbf{a}_{T\mbb{E}}$} (m2);
\path[->, bend left =30] (m1-1-2) edge node {$\mbf{a}$} (m2-1-2);
\path[->, bend left =30] (m1-1-1) edge [cross line] node {$\d\mbf{a}$} (m2-1-1);
\end{tikzpicture}
$$

Meanwhile, as reviewed in \cref{ex:SecDualTng}, $(T\mbb{E})^{*_{\!y}}\cong T(\mbb{E}^*)$. The fibre metric on $\mbb{E}$ defines an isomorphism $Q_{\la\cdot,\cdot\ra}:\mbb{E}\to \mbb{E}^*$, which lifts to an isomorphism of double vector bundles $$\d Q_{\la\cdot,\cdot\ra}:T\mbb{E}\to T(\mbb{E}^*)\cong (T\mbb{E})^{*_{\!y}}$$ via the tangent functor. This latter isomorphism defines the metric on the fibres of $T\mbb{E}\to TM$.

\end{remark}

\begin{example}[Quadratic Lie algebra]\label{ex:quadLie}
Suppose that $\mbb{E}=\mf{d}$ is a quadratic Lie algebra. Then $T\mbb{E}=T\mf{d}:=\mf{d}\ltimes\mf{d}$, where the bracket and pairing are given by 
\begin{align*}
[(\xi,\xi'),(\eta,\eta')]_{\mf{d}\ltimes\mf{d}}&=([\xi,\eta]_{\mf{d}},[\xi,\eta']_{\mf{d}}+[\xi',\eta]_{\mf{d}})\\
\la(\xi,\xi'),(\eta,\eta')\ra_{\mf{d}\ltimes\mf{d}}&=\la \xi,\eta'\ra_{\mf{d}}+\la \xi',\eta\ra_{\mf{d}}
\end{align*}
for $\xi,\xi',\eta,\eta'\in\mf{d}$.
\end{example}

\begin{example}[Action Courant algebroids]
Suppose that the quadratic Lie algebra $\mf{d}$ acts on $M$ with coisotropic stabilizers, and $\mbb{E}=\mf{d}\times M$ is the corresponding action Courant algebroid \cite{LiBland:2009ul}. The natural action of the Lie algebra $T\mf{d}$ from \cref{ex:quadLie} on $TM$ has coisotropic stabilizers, so $T\mf{d}\times TM$ carries the structure of an action Courant algebroid. Moreover, the canonical identification $$T\mbb{E}\cong T\mf{d}\times TM$$ as double vector bundles is an isomorphism of Courant algebroids.
\end{example}

\begin{example}[Exact Courant algebroids]
The tangent lift $\alpha_T\in\Omega^k(TM)$ of a $k$-form $\alpha\in\Omega^k(M)$ is defined as follows \cite{Grabowski:1997vy,Bursztyn:2010wb,Yano:1973wm}:
Let $\uptau:\Omega^k(M)\to\Omega^{k-1}(TM)$ be defined by $$\uptau(\alpha)_X=q_{TM/M}^*\iota_X\alpha, \quad X\in TM,$$ where $q_{TM/M}:TM\to M$ is the bundle projection. Then 
\begin{equation}\label{eq:TngLiftForm}\alpha_T=\d(\uptau(\alpha))+\uptau(\d\alpha).\end{equation}
Note that when $\alpha\in\Omega^0(M)$ or $\alpha\in\Omega^1(M)$ this definition agrees with the `tangent lift' construction for a function or section described in \cref{ex:TngDVB2}.

For any vector field $X\in\mf{X}(M)$, we have $$\iota_{X_C}(\alpha_C)=q^*_{TM/M}\iota_X\alpha,\quad\iota_{X_T}(\alpha_T)=(\iota_X\alpha)_T.$$

Suppose that $\mbb{E}\to M$ is an exact Courant algebroid, and $s:TM\to \mbb{E}$ an isotropic splitting, defining an isomorphism $\mbb{E}\cong \mbb{T}_\gamma M$ of Courant algebroids, where $\gamma\in\Omega^3_{cl}(M)$ is given by \cref{eq:SplitExtGam}. Since the anchor map for $T\mbb{E}$ is surjective, a dimension count shows that $T\mbb{E}$ is an exact Courant algebroid, while
\cref{eq:TEPair} shows that map $Ts:TTM\to T\mbb{E}$ defines a isotropic splitting. 

For $X,Y\in\mf{X}(M)$, the calculation
\begin{align*}
\Cour{Ts(X_T),Ts(Y_T)}&=\Cour{s(X)_T,s(Y)_T}\\
&=\Cour{s(X),s(Y)}_T\\
&=(s[X,Y]+\mbf{a}^*\iota_X\iota_Y\gamma)_T\\
&=Ts([X,Y]_T)+\mbf{a}^*\iota_{X_T}\iota_{Y_T}\gamma_T
\end{align*}
shows that the associated \v{S}evera 3-form is $\gamma_T\in\Omega^3(TM)$. Thus there is a canonical identification
 $$T\mbb{T}_\gamma M\cong \mbb{T}_{\gamma_T} TM.$$  
Since the 3-form, $\gamma\in\Omega^3_{cl}(M)$ was assumed to be closed, \cref{eq:TngLiftForm} shows that $\gamma_T=\d(\uptau(\gamma))$ is exact. Thus, all the Courant algebroids $T\mbb{T}_\gamma M$ are canonically isomorphic to $\mbb{T}TM$.

\end{example}

\subsubsection{$\mc{VB}$-Dirac structures in $T\mbb{E}$}

\begin{definition}

Suppose that $\mbb{E}\to M$ is a Courant algebroid and \begin{equation*}\label{eq:VBDirWSup}\begin{tikzpicture}
\mmat{m1} at (-2,0){L&W\\ E& S\\};
\path[->] (m1-1-1)	edge (m1-1-2)
				edge (m1-2-1);
\path[<-] (m1-2-2)	edge (m1-1-2)
				edge (m1-2-1);
\mmat{m2} at (2,0) {T\mbb{E}&\mbb{E}\\ TM& M\\};
\path[->] (m2-1-1)	edge (m2-1-2)
				edge (m2-2-1);
\path[<-] (m2-2-2)	edge (m2-1-2)
				edge (m2-2-1);
\draw (0,0) node {$\subseteq$};
\end{tikzpicture}\end{equation*}
is a double vector subbundle.
If $L\subset T\mbb{E}$ is also a Dirac structure with support on $E$, we call it a \emph{$\mc{VB}$-Dirac structure with support on $E$}. When $E=TM$, we simply call it a \emph{$\mc{VB}$-Dirac structure}.
\end{definition}

We automatically have at least one $\mc{VB}$-Dirac structure:
\begin{proposition}\label{prop:FreeVBDir} Consider the tangent prolongation $$\begin{tikzpicture}
\mmat{m}{T\mbb{E}&\mbb{E}\\ TM& M\\};
\path[->] (m-1-1)	edge node{$q_{T\mbb{E}/\mbb{E}}$} (m-1-2)
				edge (m-2-1);
\path[<-] (m-2-2)	edge (m-1-2)
				edge (m-2-1);
\end{tikzpicture}$$ of a Courant algebroid $\mbb{E}\to M$. Then
$T\mbb{E}_C:=\on{ker}(q_{T\mbb{E}/\mbb{E}})$ is a $\mc{VB}$-Dirac structure. Moreover, for any $\sigma\in\Gamma_l(T\mbb{E},TM)$ and $\tau\in\Gamma_l(T\mbb{E}_C,TM)$, we have \begin{equation}\label{eq:ECalmostideal}\Cour{\sigma,\tau}\in\Gamma_l(T\mbb{E}_C,TM).\end{equation}
\end{proposition}
\begin{proof}
 Since $T\mbb{E}_C\subseteq T\mbb{E}$ is spanned by the core sections, the fact that it is a Dirac structure follows immediately from \cref{eq:TEBrk}.

 Given $\sigma\in\Gamma_l(T\mbb{E},TM)$ and $\tau\in\Gamma_l(T\mbb{E}_C,TM)$, choose a decomposition $$\tau=\sum_i f_i\tau_i,$$ where each $f_i$ is a linear function on $TM$ and each $\tau_i\in\Gamma_C(T\mbb{E},TM)$ is a core section. Then $$\Cour{\sigma,\tau}=\sum_i\big(f_i\Cour{\sigma,\tau_i}+(\mbf{a}(\sigma)f_i)\tau_i\big).$$ Since  $\tau_i$ is a core section, the second term on the right hand side certainly lies in $\Gamma_l(T\mbb{E}_C,TM)$. Moreover,  since \cref{eq:TEBrk} implies that $\Cour{\sigma,\tau_i}$ is a core section, we also have $f_i\Cour{\sigma,\tau_i}\in\Gamma_l(T\mbb{E}_C,TM)$.
 
\end{proof}

As a corollary of the above, we get a formula for the bracket of two linear sections which we will find useful later:
\begin{corollary}\label{cor:linSecBrk}
Suppose $\tilde\sigma,\tilde\tau\in\Gamma_l(T\mbb{E},TM)$. Since $$\Gamma(\mbb{E})\xrightarrow{\sigma\to\sigma_T}\Gamma_l(T\mbb{E},TM)$$ splits the exact sequence \labelcref{eq:ShrtLinSeq}, $$0\to\Gamma(T^*M\otimes \mbb{E})\xrightarrow{i}\Gamma_l(T\mbb{E},TM)\xrightarrow{q_{T\mbb{E}/\mbb{E}}}\Gamma(\mbb{E})\to0,$$ there exist unique sections $\sigma',\tau'\in\Gamma(T^*M\otimes \mbb{E})$  such that 
$$\tilde\sigma=i\sigma'+\sigma_T,\quad\tilde\tau=i\tau'+\tau_T,$$
where $\sigma=q_{T\mbb{E}/\mbb{E}}(\tilde\sigma)\in\Gamma(\mbb{E})$ and $\tau=q_{T\mbb{E}/\mbb{E}}(\tilde\tau)\in\Gamma(\mbb{E})$. 

We have
\begin{equation}\label{eq:LinBrk}q_{T\mbb{E}/\mbb{E}}\big(\Cour{\tilde\sigma,\tilde\tau}\big)=\Cour{\sigma,\tau}+\mbf{a}^*\la \sigma',\tau\ra.\end{equation}
\end{corollary}
\begin{proof}
Note that the image of the injection $i:\Gamma(T^*M\otimes \mbb{E})\to\Gamma_l(T\mbb{E},TM)$ is precisely, $\Gamma_l(T\mbb{E}_C,TM)$.
Therefore, for any $\upsilon\in\Gamma(\mbb{E})$,  
\begin{align*}
\la q_{T\mbb{E}/\mbb{E}}\big( \Cour{\tilde\sigma,\tilde\tau}\big),\upsilon\ra_C&=\la\Cour{\tilde\sigma,\tilde\tau},\upsilon_C\ra\\
&=\la\Cour{i\sigma'+\sigma_T,i\tau'+\tau_T},\upsilon_C\ra\\
&=\la\Cour{\sigma,\tau}_T+\mbf{a}^*d\la i\sigma',\tau_T\ra-\Cour{\tau_T,i\sigma'}+\Cour{\sigma_T,i\tau'}+\Cour{i\sigma',i\tau'},\upsilon_C\ra\\
&=\la\Cour{\sigma,\tau},\upsilon\ra_C+\mbf{a}(\upsilon)_C\la i\sigma',\tau_T\ra
\end{align*}
The last line follows from \cref{prop:FreeVBDir}.
To simplify the final expression, note that $\mbf{a}(\upsilon)_C\in \mf{X}(TM)$ is just the vectical vector field corresponding to translation by the section $\mbf{a}(\upsilon)\in\Gamma(TM)$. Since $\la i\sigma',\tau_T\ra\in C^\infty(M)$ is a linear function, it follows that  
$$\mbf{a}(\upsilon)_C\la i\sigma',\tau_T\ra=\la\la \sigma',\tau\ra,\mbf{a}(\upsilon)\ra,$$ (where we interpret the right hand side, $\la\la \sigma',\tau\ra,\mbf{a}(\upsilon)\ra\in\Gamma(T^*M)$ as a linear function on $TM$). This concludes the proof.
\end{proof}

\begin{example}[\cite{Courant:1999ho} Tangent Lift of a Dirac structure (with support)]\label{ex:DirStrTngLf}
Suppose that $R\subseteq\mbb{E}$ is a Dirac structure with support on a submanifold $S\subset M$, then it follows immediately from \cref{eq:TEPair,eq:TEBrk} that $TR\subseteq T\mbb{E}$ is a $\mc{VB}$-Dirac structure with support on $TS\subseteq TM$. This was already observed by Courant in \cite{Courant:1999ho} (in the case where $S=M$).
\end{example}

\begin{proposition}\label{prop:VBDirIsLaVB}
Suppose that $$\begin{tikzpicture}
\mmat{m1} at (-2,0){L&W\\ E& M\\};
\path[->] (m1-1-1)	edge (m1-1-2)
				edge (m1-2-1);
\path[<-] (m1-2-2)	edge (m1-1-2)
				edge (m1-2-1);
\mmat{m2} at (2,0) {T\mbb{E}&\mbb{E}\\ TM& M\\};
\path[->] (m2-1-1)	edge (m2-1-2)
				edge (m2-2-1);
\path[<-] (m2-2-2)	edge (m2-1-2)
				edge (m2-2-1);
\draw (0,0) node {$\subseteq$};
\end{tikzpicture}$$
is a $\mc{VB}$-Dirac structure.
Then $L^{flip}$ is an $\mc{LA}$-vector bundle. In particular, $L\to E$ and $W\to M$ are Lie algebroids, and $L\to W$ is a Lie algebroid morphism.
\end{proposition}
\begin{proof}
By \cref{prop:TngProCA}, 
$$\gr(+_{T\mbb{E}/\mbb{E}}):T\mbb{E}\times T\mbb{E}\dasharrow T\mbb{E}$$ is a Courant morphism. 
It follows that $$\gr(+_{L/W})=\gr(+_{T\mbb{E}/\mbb{E}})\cap L^3:L\times L\to L$$ is an involutive subbundle of $L$, and hence a Lie algebroid relation. Hence $L^{flip}$ is an $\mc{LA}$-vector subbundle.
\end{proof}

\subsubsection{$\mc{VB}$-Dirac structures and pseudo-connections}

\begin{lemma}\label{lem:CoreOfLagDVB}
Suppose that $$\begin{tikzpicture}
\mmat{m1} at (-2,0){L&W\\ E& S\\};
\path[->] (m1-1-1)	edge (m1-1-2)
				edge (m1-2-1);
\path[<-] (m1-2-2)	edge (m1-1-2)
				edge (m1-2-1);
\mmat{m2} at (2,0) {T\mbb{E}&\mbb{E}\\ TM& M\\};
\path[->] (m2-1-1)	edge (m2-1-2)
				edge (m2-2-1);
\path[<-] (m2-2-2)	edge (m2-1-2)
				edge (m2-2-1);
\draw (0,0) node {$\subseteq$};
\end{tikzpicture}$$
is a Lagrangian double vector subbundle.
Then the core of $L$ is  $W^\perp\subseteq \mbb{E}$.
\end{lemma}
\begin{proof}

 For simplicity, we prove this for the case that $E=TM$. To prove the more general case, one merely replaces $T\mbb{E}$ by $T\mbb{E}\rvert_E$ in what follows.

Let $C_L$ denote the core of $L$. Dualizing the inclusion
 $$\begin{tikzpicture}
\mmat{m1} at (-2,0){L&W\\ TM& M\\};
\path[->] (m1-1-1)	edge (m1-1-2)
				edge (m1-2-1);
\path[<-] (m1-2-2)	edge (m1-1-2)
				edge (m1-2-1);
\mmat{m2} at (2,0) {T\mbb{E}&\mbb{E}\\ TM& M\\};
\path[->] (m2-1-1)	edge (m2-1-2)
				edge (m2-2-1);
\path[<-] (m2-2-2)	edge (m2-1-2)
				edge (m2-2-1);
\draw (0,0) node {$\subseteq$};
\draw (-2,0) node (c1) {$C_L$};
\draw (2,0) node (c2) {$\mbb{E}$};
\draw[left hook->] (c1) edge (m1-1-1);
\draw[left hook->] (c2) edge (m2-1-1);
\end{tikzpicture}$$
we get the projection
 $$\begin{tikzpicture}
\mmat{m1} at (2,0){L^{*_{\!y}}&C^*_L\\ TM& M\\};
\path[->] (m1-1-1)	edge (m1-1-2)
				edge (m1-2-1);
\path[<-] (m1-2-2)	edge (m1-1-2)
				edge (m1-2-1);
\mmat{m2} at (-2,0) {T\mbb{E}&\mbb{E}\\ TM& M\\};
\path[->] (m2-1-1)	edge (m2-1-2)
				edge (m2-2-1);
\path[<-] (m2-2-2)	edge (m2-1-2)
				edge (m2-2-1);
\draw[->] (m2) edge (m1);
\draw (2,0) node (c1) {$W^*$};
\draw (-2,0) node (c2) {$\mbb{E}$};
\draw[left hook->] (c1) edge (m1-1-1);
\draw[left hook->] (c2) edge (m2-1-1);
\end{tikzpicture}$$
 whose kernel is $L$. Restricting to the cores yields the map $\mbb{E}\to W^*$, whose kernel is $W^\perp$, the core of $L$.
\end{proof}

\begin{lemma}\label{lem:restPair}
Suppose that $$\begin{tikzpicture}
\mmat[4em]{m1} at (-2.5,0){L&W\\ E& S\\};
\path[->] (m1-1-1)	edge (m1-1-2)
				edge  (m1-2-1);
\path[<-] (m1-2-2)	edge (m1-1-2)
				edge (m1-2-1);
\mmat[4em]{m2} at (2.5,0) {T\mbb{E}&\mbb{E}\\ TM& M\\};
\path[->] (m2-1-1)	edge node {$q_{T\mbb{E}/\mbb{E}}$} (m2-1-2)
				edge (m2-2-1);
\path[<-] (m2-2-2)	edge (m2-1-2)
				edge (m2-2-1);
\draw (0,0) node {$\subseteq$} (m2);
\draw (-2.5,0) node (c) {$W^\perp$};
\path[left hook->] (c) edge (m1-1-1);
\draw (2.5,0) node (c2) {$\mbb{E}$};
\path[left hook->] (c2) edge (m2-1-1);
\end{tikzpicture}$$
is a double vector subbundle. Consider the double quotient \cref{eq:DVBtotQuo} described in \cref{prop:DVSBandConnect},
 \begin{equation}\label{eq:DVBtotLagQuo}\begin{tikzpicture}
\mmat[4em]{m1} at (-2.5,0){T\mbb{E}\rvert_{\scriptscriptstyle E\!\times_{\!S} \!W}&W\\ E& S\\};
\path[->] (m1-1-1)	edge  (m1-1-2)
				edge  (m1-2-1);
\path[<-] (m1-2-2)	edge (m1-1-2)
				edge (m1-2-1);
\mmat[4em]{m2} at (2.5,0) {W^*&S\\ S& S\\};
\path[->] (m2-1-1)	edge (m2-1-2)
				edge (m2-2-1);
\path[<-] (m2-2-2)	edge (m2-1-2)
				edge (m2-2-1);
\draw (-2.5,0) node (c) {$\mbb{E}\rvert_S$};
\path[left hook->] (c) edge (m1-1-1);
\draw (2.5,0) node (c2) {$W^*$};
\path[left hook->] (c2) edge (m2-1-1);
\path[->] (m1) edge node {$q_L$}(m2);
\end{tikzpicture}
\end{equation}
The following are equivalent:
\begin{enumerate}
\item $L\subseteq T\mbb{E}$ is Lagrangian.
\item For any $\sigma,\tau\in \Gamma(T\mbb{E}\rvert_{\scriptscriptstyle E\!\times_{\!S} \!W},E)$, we have
\begin{equation}\label{eq:restPair}\la \sigma,\tau\ra=\la q_L(\sigma),q_{T\mbb{E}/\mbb{E}}(\tau)\ra+\la q_{T\mbb{E}/\mbb{E}}(\sigma),q_L(\tau)\ra,\end{equation}
where the right hand side denotes the natural pairing between $W$ and $W^*$.
\item For any $$\sigma,\tau\in\widetilde{\Gamma(W)}:=\{\sigma\in\Gamma(\mbb{E})\text{ such that } \sigma\rvert_S\in\Gamma(W)\},$$
we have
\begin{equation}\label{eq:restPair2}\la \sigma,\tau\ra_T\rvert_E=\la q_L(\sigma_T),\tau\ra\rvert_E+\la \sigma,q_L(\tau_T)\ra\rvert_E,\end{equation}
where the right hand side denotes the natural pairing between $W$ and $W^*$.
\end{enumerate}
\end{lemma}
\begin{proof}$\quad$
\begin{description}

\item[$1\Rightarrow 2$]

Suppose $\sigma,\tau\in \Gamma(T\mbb{E}\rvert_{\scriptscriptstyle E\!\times_{\!S} \!W},E)$  and choose $\sigma_L,\tau_L\in \Gamma(L,E)$ so that
$$q_{T\mbb{E}/\mbb{E}}\circ\sigma_L=q_{T\mbb{E}/\mbb{E}}\circ\sigma, \quad q_{T\mbb{E}/\mbb{E}}\circ\tau_L=q_{T\mbb{E}/\mbb{E}}\circ\tau.$$
Exactness of the core sequence \labelcref{eq:iA} for $T\mbb{E}$, 
\begin{equation}\label{eq:TngProCoreSeq}\begin{tikzpicture}
\mmat{m}{q_{TM/M}^*\mbb{E}  & T\mbb{E}& q_{TM/M}^* \mbb{E}\\};
\path[->] (m-1-1) 	edge node {$i$} (m-1-2);
\path[->] (m-1-2)	edge node {$q_{T\mbb{E}/\mbb{E}}$} (m-1-3);
\end{tikzpicture}\end{equation} implies that there exists sections $\sigma',\tau'\in\Gamma(q_{E/S}^*\mbb{E},E)$ such that
$$\sigma-\sigma_L=i\circ \sigma',\quad\tau-\tau_L=i\circ\tau'.$$

As explained in \cref{rem:TangLiftTngProDef} the metric on $T\mbb{E}$ defines an isomorphism of double vector bundles 
$$T\mbb{E}\cong (T\mbb{E})^{*_{\!y}}.$$
Since the core sequence for $(T\mbb{E})^{*_{\!y}}$ is dual to \labelcref{eq:TngProCoreSeq}, we have
$$\la i\circ\sigma',\tau_L\ra=\la \sigma',q_{T\mbb{E}/\mbb{E}}\tau_L\ra=\la\tilde q\circ\sigma',q_{T\mbb{E}/\mbb{E}}\tau\ra,$$
where $\tilde q:\mbb{E}\rvert_S\to \mbb{E}\rvert_S/W^\perp\cong W^*$ is the natural projection, and the right hand side denotes the natural pairing between $W$ and $W^*$. However (as required by it's definition) the restriction of $q_L$ to the core is $\tilde q$, i.e. $\tilde q=q_L\circ i$. Thus 
$$\la i\circ\sigma',\tau_L\ra=\la q_L\circ i\circ\sigma',q_{T\mbb{E}/\mbb{E}}\tau\ra=\la q_L\circ (\sigma-\sigma_L),q_{T\mbb{E}/\mbb{E}}\tau\ra=\la q_L\circ \sigma,q_{T\mbb{E}/\mbb{E}}\tau\ra,$$
and similarly with the roles of $\sigma$ and $\tau$ replaced. Therefore
 \begin{align*}
\la\sigma,\tau\ra&=\la\sigma_L+(\sigma-\sigma_L),\tau_L+(\tau-\tau_L)\ra\\
&=\la\sigma_L,\tau_L\ra+\la\sigma_L,i\circ\tau'\ra+\la i\circ\sigma',\tau_L\ra\\
&=\la\sigma_L,\tau_L\ra+\la q_{T\mbb{E}/\mbb{E}} (\sigma),q_L(\tau)\ra+\la q_L (\sigma),q_{T\mbb{E}/\mbb{E}}(\tau)\ra.
 \end{align*}
Thus \cref{eq:restPair} holds whenever $L$ is Lagrangian. 

\item[$2\Rightarrow 3$]Combining \cref{eq:restPair,eq:TEPair} yields \cref{eq:restPair2}.

\item[$3\Rightarrow 1$]
Note that the rank of $L\to E$ is the sum of the ranks of its core and side bundles, i.e. 
$$\on{rank}(L\to E)=\on{rank}(W\to S)+\on{rank}(W^\perp\to S)
=\frac{1}{2}\on{rank}(T\mbb{E}\to TM).$$
So $L$ is Lagrangian if and only if it is isotropic. 


Suppose $\sigma_L,\tau_L\in\Gamma_l(L,E)$,  and let 
$\sigma,\tau\in\widetilde{\Gamma(W)}$ be chosen so that 
$$\sigma\rvert_S=q_{T\mbb{E}/\mbb{E}}(\sigma_L),\quad \tau\rvert_S=q_{T\mbb{E}/\mbb{E}}(\tau_L).$$
Exactness of the core sequence \labelcref{eq:TngProCoreSeq} implies that there exists sections $\sigma',\tau'\in\Gamma_l(q_{E/S}^*\mbb{E},E)$ such that
$$\sigma_L-\sigma_T=i\circ \sigma',\quad\tau_L-\tau_T=i\circ\tau'.$$
Thus, as above,
$$\la i\circ\sigma',\tau_T\ra\rvert_E=\la\sigma',q_{T\mbb{E}/\mbb{E}}\circ\tau_T\ra\rvert_E=\la q_L(\sigma_L-\sigma_T),\tau\ra\rvert_E=-\la q_L\circ\sigma_T,\tau\ra\rvert_E,$$
and similarly with the roles of $\sigma$ and $\tau$ replaced.
Therefore
 \begin{align*}
\la\sigma_L,\tau_L\ra\rvert_E&=\la\sigma_T+(\sigma_L-\sigma_T),\tilde\tau_T+(\tau_L-\tau_T)\ra\rvert_E\\
&=\la\tilde\sigma_T,\tilde\tau_T\ra\rvert_E+\la\sigma_T,i\circ\tau'\ra\rvert_E+\la i\circ\sigma',\tau_T\ra\rvert_E\\
&=\la\tilde\sigma,\tilde\tau\ra_T\rvert_E-\la\sigma,q_L(\tau_T)\ra\rvert_E-\la q_L(\sigma_T),\tau\ra\rvert_E\\
&=0
 \end{align*}
 which shows that $L$ is isotropic.
\end{description}
\end{proof}

\begin{proposition}\label{prop:clasLagTDSB}
Suppose that $\mbb{E}$ is a Courant algebroid, $S\subseteq M$ is a submanifold, $W\to S$ is a vector subbundle of $\mbb{E}\rvert_S\to S$, and $E\to S$ is a vector subbundle of $TM\rvert_S\to S$.
Let $$\widetilde{\Gamma(W)}:=\{\sigma\in\Gamma(\mbb{E})\text{ such that } \sigma\rvert_S\in\Gamma(W)\}.$$
The correspondence described in \cref{prop:DVSBandConnect} restricts to a one-to-one correspondence between  Lagrangian double vector subbundles 
\begin{equation}\label{eq:LagDVBforConn}\begin{tikzpicture}
\mmat[4em]{m1} at (-2.5,0){L&W\\ E& S\\};
\path[->] (m1-1-1)	edge (m1-1-2)
				edge  (m1-2-1);
\path[<-] (m1-2-2)	edge (m1-1-2)
				edge (m1-2-1);
\mmat[4em]{m2} at (2.5,0) {T\mbb{E}&\mbb{E}\\ TM& M\\};
\path[->] (m2-1-1)	edge node {$q_{T\mbb{E}/\mbb{E}}$} (m2-1-2)
				edge (m2-2-1);
\path[<-] (m2-2-2)	edge (m2-1-2)
				edge (m2-2-1);
\draw (0,0) node {$\subseteq$} (m2);
\draw (-2.5,0) node (c) {$W^\perp$};
\path[left hook->] (c) edge (m1-1-1);
\draw (2.5,0) node (c2) {$\mbb{E}$};
\path[left hook->] (c2) edge (m2-1-1);
\end{tikzpicture}
\end{equation}
and linear maps \begin{equation}\label{eq:NabForL}\nabla:\widetilde{\Gamma(W)}\to \Gamma(E^*\otimes W^*)\end{equation} satisfying the Leibniz rule 
\begin{equation}\label{eq:NabForLLeib}\nabla f\sigma=f\nabla\sigma+\d\!f\otimes \la\sigma,\cdot\ra,\end{equation}
 for any $f\in C^\infty(M)$, $\sigma\in\widetilde{\Gamma(W)}$, and the metric compatibility condition
 \begin{equation}\label{eq:NabForLMetComp}\d\la\sigma,\tau\ra\rvert_E=\la\nabla\sigma,\tau\ra+\la\sigma,\nabla\tau\ra,\end{equation}
 for any $\sigma,\tau\in\widetilde{\Gamma(W)}$.
\end{proposition}
\begin{proof}
Since \cref{eq:NabFormFromK} states that $\nabla\sigma=q_L\circ \sigma_T$,
\cref{eq:restPair2} is the same equation as \cref{eq:NabForLMetComp}. 
Therefore \cref{lem:restPair} shows that under the correspondence described in \cref{prop:DVSBandConnect}, Lagrangian double vector subbundles of the form \labelcref{eq:LagDVBforConn} correspond to linear maps \labelcref{eq:NabForL} satisfying \cref{eq:NabForLMetComp}.

\end{proof}

\Cref{prop:VBDirIsLaVB} shows that if $L\subseteq T\mbb{E}$ is a $\mc{VB}$-Dirac structure, then the side bundle $W$ carries a Lie algebroid structure. The following proposition describes the Lie algebroid bracket on $W$ in terms of the Courant bracket on $\mbb{E}$ and the pseudo-connection 
$$\nabla:\Gamma(W)\to \Gamma(T^*M\otimes W^*)$$
 described in \cref{prop:clasLagTDSB}.

\begin{proposition}\label{prop:qIndBrk}
Suppose that 
$$\begin{tikzpicture}
\mmat{m1} at (-2,0){L&W\\ TM& M\\};
\path[->] (m1-1-1)	edge (m1-1-2)
				edge (m1-2-1);
\path[<-] (m1-2-2)	edge (m1-1-2)
				edge (m1-2-1);
\mmat{m2} at (2,0) {T\mbb{E}&\mbb{E}\\ TM& M\\};
\path[->] (m2-1-1)	edge node {$q_{\mbb{E}/V}$} (m2-1-2)
				edge (m2-2-1);
\path[<-] (m2-2-2)	edge (m2-1-2)
				edge (m2-2-1);
\draw (0,0) node {$\subseteq$};
\end{tikzpicture}$$
 is a $\mc{VB}$-Dirac structure, and let $$\nabla:\Gamma(W)\to \Gamma(T^*M\otimes W^*)$$ be the pseudo-connection \labelcref{eq:NabForL} described in \cref{prop:clasLagTDSB}.
 Suppose that $\sigma,\tau\in\Gamma(W)$ are two sections.
  Then \begin{equation}\label{eq:VBDirIndBrk}[\sigma,\tau]=\Cour{\sigma,\tau}-\mbf{a}^*\la \nabla\sigma,\tau\ra,\end{equation} where the left hand side denotes the Lie algebroid bracket on $W$ described in \cref{prop:VBDirIsLaVB}.
\end{proposition}
\begin{proof}Choose $\tilde\sigma,\tilde\tau\in\Gamma_l(L,TM)$ to be $q_{L/W}$-related to $\sigma$ and $\tau$, respectively.
That is,
$$q_{T\mbb{E}/\mbb{E}}(\tilde\sigma)=\sigma,\quad q_{T\mbb{E}/\mbb{E}}(\tilde\tau)=\tau.$$
By definition of the Lie algebroid structure on $W$, we have 
\begin{subequations}\label[pluralequation]{eq:IndBrk}
\begin{equation}\label{eq:IndBrk1}[\sigma,\tau]=q_{T\mbb{E}/\mbb{E}}\big(\Cour{\tilde\sigma,\tilde\tau}\big).\end{equation}
But \cref{cor:linSecBrk} implies that 
\begin{equation}\label{eq:IndBrk2}q_{T\mbb{E}/\mbb{E}}\big( \Cour{\tilde\sigma,\tilde\tau}\big)=\Cour{\sigma,\tau}+\mbf{a}^*\la \sigma',\tau\ra.\end{equation}
Now, by definition $\nabla\sigma=q_L(\sigma_T),$ where 
$$q_L:T\mbb{E}\rvert_W\to T\mbb{E}\rvert_W/L\cong W^*$$ 
is the double quotient map \labelcref{eq:DVBtotLagQuo}.
Since $\tilde\sigma\in\Gamma(L,TM)$, it follows that 
$$\nabla\sigma=q_L(\sigma_T)=q_L(\sigma_T-\tilde\sigma)=-q_L\sigma'.$$
Since the restriction of $q_L$ to the core, $\mbb{E}$ is just the canonical projection $\mbb{E}\to\mbb{E}/W^\perp$, and $\tau\in\Gamma(W)$, we get 
\begin{equation}\label{eq:IndBrk3}\la \sigma',\tau\ra=-\la\nabla\sigma,\tau\ra.\end{equation}
\end{subequations}
Combining \cref{eq:IndBrk1,eq:IndBrk2,eq:IndBrk3} yields \cref{eq:VBDirIndBrk}.

\end{proof}

\subsection{Pseudo-Dirac structures and the tangent prolongation}


In this section, we show that there is a one-to-one correspondence between $\mc{VB}$-Dirac structures in $T\mbb{E}$ and pseudo-Dirac structures in $\mbb{E}$. In particular, we will prove the main result of this paper, that
\begin{theorem}\label{thm:LieSubIsVBDir}
Suppose $\mbb{E}\to M$ is a Courant algebroid.
\begin{itemize}
\item If $(W,\nabla)$ is a pseudo-Dirac structure for $\mbb{E}$, then 
$$[\sigma,\tau]:=\Cour{\sigma,\tau}-\mbf{a}^*\la \nabla\sigma,\tau\ra,\quad\sigma,\tau\in\Gamma(W),$$
 defines a Lie algebroid bracket on $W$.
\item The correspondence described in \cref{prop:clasLagTDSB}
%
 between Lagrangian double vector subbundles
\begin{equation}\label{eq:VBDirTngThm}
\begin{tikzpicture}
\mmat{m1} at (-2,0) {L&W\\ TM& M\\};
\path[->] (m1-1-1)	edge node {$q_{L/W}$} (m1-1-2)
				edge (m1-2-1);
\path[<-] (m1-2-2)	edge (m1-1-2)
				edge (m1-2-1);
\mmat{m2} at (2,0) {T\mbb{E}&\mbb{E}\\ TM& M\\};
\path[->] (m2-1-1)	edge (m2-1-2)
				edge (m2-2-1);
\path[<-] (m2-2-2)	edge (m2-1-2)
				edge (m2-2-1);
				
\draw (0,0) node {$\subseteq$};
\end{tikzpicture}
\end{equation} 
and subbundles $W\subseteq \mbb{E}$ carrying a pseudo-connection $$\nabla:\Omega^0(W)\to \Omega^1(W^*)$$ restricts to a one-to-one correspondence between $\mc{VB}$-Dirac structures of the form \labelcref{eq:VBDirTngThm} and 
pseudo-Dirac structures $(W,\nabla)$ in $\mbb{E}$. 
Under this correspondence, the map $q_{L/W}:L\to W$ is a Lie algebroid morphism. 
\end{itemize}

\end{theorem}

Before proving this theorem, we will first prove some preliminary results.
%
%
%

\begin{lemma}\label{lem:ModBrkProp1}
Suppose that $\mbb{E}\to M$ is a Courant algebroid and $\nabla$ is a pseudo-connection  for the subbundle $W\subseteq \mbb{E}$. Then the modified bracket \labelcref{eq:pcModBrk} is skew symmetric.
\end{lemma}
\begin{proof}
 Axioms (c3) for the Courant bracket (see \cref{def:CA}) implies that, 
for any $\sigma,\tau\in\Gamma(W)$, we have
\begin{align*}
[\sigma,\tau]+[\tau,\sigma]&=\Cour{\sigma,\tau}+\Cour{\tau,\sigma}-\mbf{a}^*\la \nabla\sigma,\tau\ra-\mbf{a}^*\la \nabla\tau,\sigma\ra\\
&=\mbf{a}^*\d\la\sigma,\tau\ra-\mbf{a}^*\la \nabla\sigma,\tau\ra-\mbf{a}^*\la \nabla\tau,\sigma\ra\\
&=0.
\end{align*}
Here the final equality follows from \cref{eq:NabPairDer}.
\end{proof}

\begin{proposition}\label{prop:TorsTens}
\Cref{eq:torsTens}, which we restate here, $$T(\sigma,\tau,\upsilon)=\la\nabla_{\mbf{a}(\sigma)}\tau-\nabla_{\mbf{a}(\tau)}\sigma-[\sigma,\tau],\upsilon\ra,\quad\sigma,\tau,\upsilon\in\Gamma(W),$$ defines a skew symmetric tensor. That is, $T\in\Gamma(\wedge^3W^*)$.
\end{proposition}
\begin{proof}
First we show that $T$ is cyclic in its arguments. Combining \cref{eq:pcModBrk,eq:torsTens} yields
\begin{align*}
T(\sigma,\tau,\upsilon)&=\la\nabla_{\mbf{a}(\sigma)}\tau-\nabla_{\mbf{a}(\tau)}\sigma-[\sigma,\tau],\upsilon\ra\\
&=\la\nabla_{\mbf{a}(\sigma)}\tau-\nabla_{\mbf{a}(\tau)}\sigma-\Cour{\sigma,\tau} +\mbf{a}^*\la\nabla\sigma,\tau\ra,\upsilon\ra\\
&=\la\nabla_{\mbf{a}(\sigma)}\tau-\nabla_{\mbf{a}(\tau)}\sigma+\mbf{a}^*\la\nabla\sigma,\tau\ra,\upsilon\ra  -\mbf{a}(\sigma)\la\tau,\upsilon\ra+\la\tau,\Cour{\sigma,\upsilon}\ra,
\end{align*}
where we used axiom (c2) for the Courant bracket in the last line. Hence
\begin{align*}
T(\sigma,\tau,\upsilon)
&=\la\nabla_{\mbf{a}(\sigma)}\tau-\nabla_{\mbf{a}(\tau)}\sigma,\upsilon\ra+\la\nabla_{\mbf{a}(\upsilon)}\sigma,\tau\ra  -\mbf{a}(\sigma)\la\tau,\upsilon\ra+\la\tau,\Cour{\sigma,\upsilon}\ra\\
&=\la\nabla_{\mbf{a}(\sigma)}\tau-\nabla_{\mbf{a}(\tau)}\sigma,\upsilon\ra+\la\nabla_{\mbf{a}(\upsilon)}\sigma,\tau\ra  -\la\nabla_{\mbf{a}(\sigma)}\tau,\upsilon\ra-\la\tau,\nabla_{\mbf{a}(\sigma)}\upsilon\ra+\la\tau,\Cour{\sigma,\upsilon}\ra\\
&=\la\nabla_{\mbf{a}(\upsilon)}\sigma-\nabla_{\mbf{a}(\sigma)}\upsilon+\Cour{\sigma,\upsilon},\tau\ra-\la\nabla_{\mbf{a}(\tau)}\sigma,\upsilon\ra,
\end{align*}
where we used \cref{eq:NabPairDer} in the second line. Finally, using \cref{eq:pcModBrk} again, we get
\begin{align*}
T(\sigma,\tau,\upsilon)
&=\la\nabla_{\mbf{a}(\upsilon)}\sigma-\nabla_{\mbf{a}(\sigma)}\upsilon+[\sigma,\upsilon],\tau\ra\\
&=\la\nabla_{\mbf{a}(\upsilon)}\sigma-\nabla_{\mbf{a}(\sigma)}\upsilon-[\upsilon,\sigma],\tau\ra\\
&=T(\upsilon,\sigma,\tau),
\end{align*}
where we used \cref{lem:ModBrkProp1} to obtain the second equality.
 
\Cref{lem:ModBrkProp1} implies that $T$ is skew symmetric in the first two arguments, and since it is also cyclic, it must be totally skew symmetric. Moreover, it is manifestly tensorial in the last argument, and hence tensorial in all its arguments.
\end{proof}

\begin{lemma}\label{lem:ModBrkProp2}
Suppose that $\mbb{E}\to M$ is a Courant algebroid and $\nabla$ is a pseudo-connection  for the subbundle $W\subseteq \mbb{E}$. Let $L\subseteq T\mbb{E}$ be the  Lagrangian double vector subbundle corresponding to $(W,\nabla)$ via \cref{prop:clasLagTDSB}.
Let $L_C=T\mbb{E}_C\cap L$, where $T\mbb{E}_C$ is the vertical subbundle of $T\mbb{E}\to TM$, as defined in \cref{prop:FreeVBDir}.
The modified bracket \labelcref{eq:pcModBrk} takes values in $\Gamma(W)$ if and only if $$\Cour{\Gamma_l(L,TM),\Gamma_l(L_C,TM)}\subseteq\Gamma_l(L_C,TM).$$
\end{lemma}
\begin{proof}
Let $\sigma,\tau\in\Gamma(W)$, and suppose $\tilde\sigma,\tilde\tau\in\Gamma_l(L,TM)$ are $q_{L/W}$-related to $\sigma$ and $\tau$, respectively. That is
$$q_{T\mbb{E}/\mbb{E}}( \tilde\sigma)=\sigma,\quad q_{T\mbb{E}/\mbb{E}}(\tilde\tau)=\tau.$$
Then \cref{cor:linSecBrk} and \cref{eq:IndBrk} imply that 
$$q_{T\mbb{E}/\mbb{E}}(\Cour{\tilde\sigma,\tilde\tau})=[\sigma,\tau],$$
where the bracket on the right hand side is the modified bracket \labelcref{eq:pcModBrk}.

The short exact sequence \labelcref{eq:ShrtLinSeq} defines an isomorphism $\Gamma(T^*M\otimes W^\perp)\cong\Gamma_l(L_C,TM)$. Thus 
$[\sigma,\tau]\in\Gamma(W)$ if and only if for any $\upsilon\in\Gamma(T^*M\otimes W^\perp)\cong\Gamma_l(L_C,TM)$, we have
$$\la [\sigma,\tau],\upsilon\ra=\la \Cour{\tilde\sigma,\tilde\tau},\upsilon\ra=0,$$
where the first equality follows from \cref{eq:TEPair}.
But Axiom (c2) for the Courant algebroid implies
$$\la \Cour{\tilde\sigma,\tilde\tau},\upsilon\ra=\la\tilde\tau,\Cour{\tilde\sigma,\upsilon}\ra.$$
The right hand side vanishes for arbitrary $\tilde\sigma,\tilde\tau\in\Gamma_l(L,TM)$ and $\upsilon\in\Gamma_l(L_C,TM)$ if and only if 
$$\Cour{\Gamma_l(L,TM),\Gamma_l(L_C,TM)}\subseteq\Gamma_l(L,TM).$$
However, \cref{eq:ECalmostideal} implies that the left hand side necessarily lies in $\Gamma_l(T\mbb{E}_C,TM)$. Therefore, we conclude that 
$$\Cour{\Gamma_l(L,TM),\Gamma_l(L_C,TM)}\subseteq\Gamma_l(L_C,TM)$$
if and only if 
$$[\Gamma(W),\Gamma(W)]\subseteq \Gamma(W).$$ 

%
%
%
\end{proof}

\begin{proposition}\label{lem:PsiProp}
Suppose that $\nabla$ is a pseudo-connection  for the subbundle $W\subseteq \mbb{E}$, and that the modified bracket \labelcref{eq:pcModBrk} takes values in $\Gamma(W)$. 
Let $L\subseteq T\mbb{E}$ be the Lagrangian double vector subbundle corresponding to $(W,\nabla)$ via \cref{prop:clasLagTDSB}.
\begin{itemize}
\item
The following expression: \begin{equation*}\begin{split}
\Psi(\sigma,\tau,\upsilon)=&\la[\sigma,\tau],\nabla\upsilon\ra+\la[\upsilon,\sigma],\nabla\tau\ra+\la[\tau,\upsilon],\nabla\sigma\ra\\
&+\iota_{\mbf{a}(\sigma)}\d\la\nabla\tau,\upsilon\ra+\iota_{\mbf{a}(\upsilon)}\d\la\nabla\sigma,\tau\ra+\iota_{\mbf{a}(\tau)}\d\la\nabla\upsilon,\sigma\ra\\
&+\d T(\sigma,\tau,\upsilon),
\end{split}\end{equation*}
for $\sigma,\tau,\upsilon\in\Gamma(W)$, defines a skew symmetric tensor on $W$. That is,  $$\Psi\in\Omega^1(M, \wedge^3 W^*).$$
\item $L$ is a Dirac structure if and only if $\Psi=0$.
\end{itemize}
\end{proposition}
\begin{proof}
Let $L_C=T\mbb{E}_C\cap L$, where $T\mbb{E}_C$ is the vertical subbundle of $T\mbb{E}\to TM$, as defined in \cref{prop:FreeVBDir}. Then \cref{lem:ModBrkProp2} implies that 
\begin{subequations}\label[pluralequation]{eq:LCTM}
\begin{equation}\label{eq:LCTM1}\Cour{\Gamma_l(L,TM),\Gamma_l(L_C,TM)}\subseteq \Gamma_l(L_C,TM).\end{equation}
 Since $L$ is Lagrangian, we also have 
 \begin{equation}\label{eq:LCTM2}\Cour{\Gamma_l(L_C,TM),\Gamma_l(L,TM)}\subseteq \Gamma_l(L_C,TM).\end{equation}  
 Now,  
 \begin{equation}\label{eq:LCTM3}q_L\big(\Gamma_l(L_C,TM)\big)=0,\end{equation}
 \end{subequations}
 where $$q_L:T\mbb{E}\rvert_W\to T\mbb{E}\rvert_W/L\cong W^*$$ is the double quotient map \labelcref{eq:DVBtotLagQuo}
defined \cref{prop:DVSBandConnect}. 

Combining \cref{eq:LCTM1,eq:LCTM2,eq:LCTM3} we see that for $\tilde\sigma,\tilde\tau\in\Gamma_l(L,TM)$, the expression $q_L\Cour{\tilde\sigma,\tilde\tau}$ only depends on $q_{L/W}\circ\tilde\sigma$ and $q_{L/W}\circ\tilde\tau$.
 Consequently, if $\tilde\sigma,\tilde\tau,\tilde\upsilon\in\Gamma_l(L,TM)$ are $q_{L/W}$ related to $\sigma,\tau,\upsilon\in\Gamma(W)$,  the expression 
\begin{equation}\label{eq:LinvLinTens}
-\la\Cour{\tilde\sigma,\tilde\tau},\tilde\upsilon\ra\in \Gamma_l(TM\times \mbb{R},TM)\cong\Omega^1(M)
\end{equation}
 only depends on $\sigma,\tau,\upsilon\in\Gamma(W)$. Hence \labelcref{eq:LinvLinTens} defines a tensor  $\Psi\in\Omega^1(M, \wedge^3 W^*)$ measuring the involutivity of $L$, which we shall now calculate directly. 
 
 To simplify notation, we let $\sigma'=\tilde\sigma-\sigma_T$, $\tau'=\tilde\tau-\tau_T$ and $\upsilon'=\tilde\upsilon-\upsilon_T$, and remark that $\sigma',\tau',\upsilon'\in\Gamma_l(T\mbb{E}_C,TM)$. 
Since $q_L(\tilde\sigma)=q_L(\tilde\tau)=q_L(\tilde\upsilon)=0$,  \cref{eq:NabFormFromK} implies that
\begin{align}\label{eq:qsig'}
q_L(\sigma')&=-\nabla(\sigma), &
q_L(\tau')&=-\nabla(\tau),&
q_L(\upsilon')&=-\nabla(\upsilon).
\end{align}
Plugging the last of these into $\Psi$ and simplifying, 
we get
\begin{align}
-\Psi(\sigma,\tau,\upsilon)
=&\la\Cour{\tilde\sigma,\tilde\tau},\upsilon'+\upsilon_T\ra\notag\\
=&-\la[\sigma,\tau],\nabla\upsilon\ra+\la\Cour{\sigma,\tau}_T+\Cour{\sigma_T,\tau'}+\Cour{\sigma',\tau_T}+\Cour{\sigma',\tau'},\upsilon_T\ra\notag\\
=&-\la[\sigma,\tau],\nabla\upsilon\ra+\la\Cour{\sigma,\tau}_T,\upsilon_T\ra-\la\Cour{\sigma,\upsilon}_T,\tau'\ra+\la\Cour{\tau,\upsilon}_T,\sigma'\ra\notag\\
&+\mbf{a}(\sigma_T)\la\tau',\upsilon_T\ra-\mbf{a}(\tau_T)\la\sigma',\upsilon_T\ra+\mbf{a}(\upsilon_T)\la\sigma',\tau_T\ra+\la\Cour{\sigma',\tau'},\upsilon_T\ra. \label{eq:PsiLstLine1}
\end{align}
To obtain the last line we rearranged terms using axioms (c2) and (c3) for a Courant algebroid (see \cref{def:CA}). Using (c2) and (c3) again, we notice that $$\la\Cour{\sigma',\tau'},\upsilon_T\ra=\la\mbf{a}^*\d\la\tau',\upsilon_T\ra,\sigma'\ra-\la\mbf{a}^*\d\la\sigma',\upsilon_T\ra,\tau'\ra+\la\tau',\Cour{\upsilon_T,\sigma'}\ra,$$ and the third term vanishes since $\Cour{\upsilon_T,\sigma'}\in\Gamma(T\mbb{E}_C,TM)$. Next, using \cref{eq:restPair,eq:qsig'} notice that $\la\sigma',\upsilon_T\ra=-\la\nabla\sigma,\upsilon\ra$ and hence $\mbf{a}^*\d\la\sigma',\upsilon_T\ra=-\mbf{a}^*\la\nabla\sigma,\upsilon\ra_T$.
Substituting the corresponding expressions for various permutations of $\sigma,\tau$ and $\upsilon$ into \cref{eq:PsiLstLine1}, we get
\begin{align*}
-\Psi(\sigma,\tau,\upsilon)=&-\la[\sigma,\tau],\nabla\upsilon\ra+\la\Cour{\sigma,\tau}_T,\upsilon_T\ra-\la\Cour{\sigma,\upsilon}_T,\tau'\ra+\la\Cour{\tau,\upsilon}_T,\sigma'\ra\\
&-\mbf{a}(\sigma_T)\la\nabla\tau,\upsilon\ra+\mbf{a}(\tau_T)\la\nabla\sigma,\upsilon\ra-\mbf{a}(\upsilon_T)\la\nabla\sigma,\tau\ra\\
&-\la\mbf{a}^*\la\nabla\tau,\upsilon\ra_T,\sigma'\ra+\la\mbf{a}^*\la\nabla\sigma,\upsilon\ra_T,\tau'\ra.
\end{align*}
Using the definition \labelcref{eq:pcModBrk} of the bracket, we get
\begin{align}
-\Psi(\sigma,\tau,\upsilon)=&-(\la[\sigma,\tau],\nabla\upsilon\ra+\la[\upsilon,\sigma],\nabla\tau\ra+\la[\tau,\upsilon],\nabla\sigma\ra)\notag\\
&+\d\la[\sigma,\tau],\upsilon\ra+\d\la\nabla_{\mbf{a}(\upsilon)}\sigma,\tau\ra\notag\\
&-\mbf{a}(\sigma_T)\la\nabla\tau,\upsilon\ra+\mbf{a}(\tau_T)\la\nabla\sigma,\upsilon\ra-\mbf{a}(\upsilon_T)\la\nabla\sigma,\tau\ra.\label{eq:PsiLstLine2}
\end{align}

Now $\mbf{a}(\sigma_T)\la\nabla\tau,\upsilon\ra=\Lied_{\mbf{a}(\sigma)}\la\nabla\tau,\upsilon\ra=\d\la\nabla_{\mbf{a}(\sigma)}\tau,\upsilon\ra+\iota_{\mbf{a}(\sigma)}\d\la\nabla\tau,\upsilon\ra$. Substituting the corresponding expression for various permutations of $\sigma,\tau$ and $\upsilon$ into \cref{eq:PsiLstLine2}, we get
%
%
\begin{equation}\label{eq:Psi_La}
\begin{split}
\Psi(\sigma,\tau,\upsilon)=&\la[\sigma,\tau],\nabla\upsilon\ra+\la[\upsilon,\sigma],\nabla\tau\ra+\la[\tau,\upsilon],\nabla\sigma\ra\\
&+\iota_{\mbf{a}(\sigma)}\d\la\nabla\tau,\upsilon\ra+\iota_{\mbf{a}(\upsilon)}\d\la\nabla\sigma,\tau\ra+\iota_{\mbf{a}(\tau)}\d\la\nabla\upsilon,\sigma\ra\\
&+\d T(\sigma,\tau,\upsilon),
\end{split}\end{equation}
\end{proof}

\begin{proof}[Proof of \cref{thm:LieSubIsVBDir}]
\Cref{prop:clasLagTDSB,lem:PsiProp} establish a one-to-one correspondence between $\mc{VB}$-Dirac structures of the form \labelcref{eq:VBDirTngThm}, and pseudo-Dirac structures $(W,\nabla)$ in $\mbb{E}$. Meanwhile \cref{prop:qIndBrk} states that the Lie bracket on $\Gamma(W)$ is given by the formula
$$[\sigma,\tau]:=\Cour{\sigma,\tau}-\mbf{a}^*\la \nabla\sigma,\tau\ra,\quad\sigma,\tau\in\Gamma(W).$$

\end{proof}

\begin{remark}
Suppose $A$ is a Lie algebroid (with bracket $[\cdot,\cdot]$ and anchor $\mbf{a}$),  $\la\cdot,\cdot\ra$ is a metric on the fibres of $A$, and $\nabla$ is a metric connection on $A$. Then
$$\Cour{\sigma,\tau}:=[\sigma,\tau]+\mbf{a}^*\la \nabla\sigma,\tau\ra,\quad\sigma,\tau\in\Gamma(A),$$
 defines a Courant bracket on $A$ if and only if
\begin{itemize}
\item $\mbf{a}\circ\mbf{a}^*=0$ ($A$ acts with coisotropic stabilizers),
\item the `torsion' tensor \labelcref{eq:torsTens} is skew symmetric, and
\item $\mbf{a}^*\Psi=0$, (where $\Psi$  is defined  in terms of the curvature and torsion by \cref{eq:PsiMet}).
\end{itemize}
Indeed axiom (c3) for a Courant bracket (see \cref{def:CA}) holds since $\nabla$ is a metric connection, axiom (c2) holds if and only if the `torsion' tensor is skew symmetric, and axiom (c1) holds if and only if $\mbf{a}\circ\mbf{a}^*=0$ and $\mbf{a}^*\Psi=0$.

It is perhaps more natural to require that $\Psi=0$, in which case $(A,\nabla)$ is embedded as a pseudo-Dirac structure in the corresponding Courant algebroid.
\end{remark}

\begin{proposition}\label{prop:fltSec}
Suppose that $(W,\nabla)$ is a pseudo-Dirac structure  in the Courant algebroid $\mbb{E}$. If $\sigma,\tau\in\Gamma(W)$ satisfy $\nabla\sigma=\nabla\tau=0$  then 
\begin{itemize}
\item $\nabla\Cour{\sigma,\tau}=0$, and
\item $\Cour{\sigma,\tau}=[\sigma,\tau]$.
\end{itemize} That is, the `flat' sections of $W$ form a Lie algebra with respect to the Courant bracket.
\end{proposition}
\begin{proof}
First, if $\nabla\sigma=0$, then $[\sigma,\tau]:=\Cour{\sigma,\tau}-\mbf{a}^*\la\nabla\sigma,\tau\ra=\Cour{\sigma,\tau}$.

Next, let $L\subset T\mbb{E}$ be the $\mc{VB}$-Dirac structure corresponding to $(W,\nabla)$. By \cref{eq:NabFormFromK}, we see that $\nabla\sigma=0$ if and only if $\sigma_T\in\Gamma(L,TM)$. Therefore, $\sigma_T,\tau_T\in\Gamma(L,TM)$, and thus  $\Cour{\sigma_T,\tau_T}=\Cour{\sigma,\tau}_T\in \Gamma(L,TM)$. In turn, this shows that $\nabla\Cour{\sigma,\tau}=0$.
\end{proof}

\section{Forward and backward images of  pseudo-Dirac structures}\label{sec:FBimage}

The properties of Dirac structures which allow you to compose them with Courant relations (defined below) extend to pseudo-Dirac structures, as we shall explain in this section. Indeed, morally, 
any procedure for Dirac structures carries over to the more general pseudo-Dirac structures, since the latter \emph{are in fact just Dirac structures} in the tangent prolongation of the Courant algebroid.
Special cases of composing with Courant relations include forward and backward Dirac maps, or restricting a Dirac structure to a submanifold. 

\subsubsection{Composition of Courant relations with pseudo-Dirac structures}

Suppose that $(W,\nabla)$ is a pseudo-Dirac structure in the Courant algebroid $\mbb{E}$, and $R:\mbb{E}\dasharrow \mbb{F}$ is a Courant relation with support on $S:M\dasharrow N$. Let $L\subset T\mbb{E}$ be the $\mc{VB}$-Dirac structure corresponding to $(W,\nabla)$ (c.f. \cref{thm:LieSubIsVBDir}). From \cref{ex:DirStrTngLf} we see that $TR:T\mbb{E}\dasharrow T\mbb{F}$ is a Courant relation with support on $TS:TM\dasharrow TN$. 
\begin{definition}
We say that $R:\mbb{E}\dasharrow \mbb{F}$ \emph{composes cleanly} with the pseudo-Dirac structure $(W,\nabla)$ if $TR:T\mbb{E}\dasharrow T\mbb{F}$ composes cleanly with $L$.
\end{definition}
Assume that the composition $TR\circ L$ is clean, and equal to a subbundle $L'\subseteq T\mbb{F}$ supported on all of $TN$. Then the Proposition above shows that $L'$ is a $\mc{VB}$-Dirac structure, which in turn corresponds to a pseudo-Dirac structure $(W',\nabla')$ in $\mbb{F}$ (c.f. \cref{thm:LieSubIsVBDir}). In this case, we write $$(W',\nabla')=R\circ(W,\nabla),$$ and call it the \emph{forward image} along $R:\mbb{E}\dasharrow \mbb{F}$.

Conversely, suppose we instead start with a pseudo-Dirac structure $(W',\nabla')$ in $\mbb{E}$ corresponding to a $\mc{VB}$-Dirac structure $L'\subseteq T\mbb{F}$ (c.f. \cref{thm:LieSubIsVBDir}). 
\begin{definition}
We say that $R:\mbb{E}\dasharrow \mbb{F}$ \emph{composes cleanly} with the pseudo-Dirac structure $(W',\nabla')$ if $TR:T\mbb{E}\dasharrow T\mbb{F}$ composes cleanly with $L'$.
\end{definition}
Assume that the composition $L'\circ TR$ is clean, and equal to a subbundle $L\subseteq T\mbb{E}$ supported on all of $TM$. Then the Proposition above shows that $L$ is a $\mc{VB}$-Dirac structure, which in turn corresponds to a pseudo-Dirac structure $(W,\nabla)$ in $\mbb{E}$ (c.f. \cref{thm:LieSubIsVBDir}). In this case, we write $$(W,\nabla)=(W',\nabla')\circ R,$$ and call it the \emph{backward image} along $R:\mbb{E}\dasharrow \mbb{F}$.

The following proposition describes a useful equation satisfied by the pseudo-connections associated  to forward and backward images of a pseudo-Dirac structure along a Courant relation.

\begin{proposition}
Suppose $(W,\nabla)$ and $(W',\nabla')$ are pseudo-Dirac structures in the Courant algebroids $\mbb{E}\to M$ and $\mbb{F}\to N$, respectively.
 Let $L\subseteq T\mbb{E}$ and $L'\subseteq T\mbb{F}$ be the respective $\mc{VB}$-Dirac structures corresponding to $(W,\nabla)$ and $(W',\nabla')$ (c.f. \cref{thm:LieSubIsVBDir}).
Suppose $R:\mbb{E}\dasharrow\mbb{F}$ is a Courant relation over $S:M\dasharrow N,$ and that the intersection $$(L'\times L)\cap TR$$ is clean.
Then 
for any section $(\sigma',\sigma)\in\Gamma(W'\times W)$ satisfying 
 $$(\sigma',\sigma)\rvert_S\in\Gamma\big((W'\times W)\cap R\big),$$ and any  element $$\big(X;(\lambda',\lambda)\big)\in TS\times_M \big((W'\times W)\cap R\big),$$ the following equation holds:
\begin{equation}\label{eq:nablaComposition}\la (p_M^*\nabla')_{X}\sigma',\lambda'\ra=\la (p_N^*\nabla)_{X}\sigma,\lambda\ra.\end{equation} Here $p_N^*\nabla$ and $p_M^*\nabla'$ are the pull-backs of the respective pseudo-connections along the natural projections $p_N:N\times M\to N$ and $p_M:N\times M\to M$ (c.f. \cref{def:pseuCon}).

\end{proposition}
\begin{proof}
%
%
%

%
%
 Let $$q_{L'\times L}:T(\mbb{F}\times\overline{\mbb{E}})\rvert_{W'\times W}\to (W'\times W)^*$$ be the double quotient map \cref{eq:DVBtotLagQuo} corresponding to the Lagrangian double vector subbundle $L'\times L\subseteq T(\mbb{F}\times\overline{\mbb{E}})$.
Now the $\mc{VB}$-Dirac structure $L'\times L$ corresponds to the pseudo-Dirac structure $\big(W'\times W,(-\nabla')\oplus\nabla\big)$ (c.f. \cref{thm:LieSubIsVBDir,ex:BarLieSub}). Thus, by \cref{eq:NabFormFromK}, for any section $(\sigma',\sigma)\in\Gamma(W'\times W)$ and any $X\in T(N\times M)$, we have \begin{equation}\label{eq:qLieSubCom}q_{L'\times L}(\sigma'_T,\sigma_T)\rvert_{X}=\big((-\nabla')\oplus\nabla\big)_{X}(\sigma',\sigma)=\big(-(p_N^*\nabla')_{X}\sigma',(p_M^*\nabla)_{X}\sigma\big).\end{equation}

Next, by the clean intersection assumption, $$\begin{tikzpicture} \mmat{m2}{\big(L'\times L\big)\cap TR &(W'\times W)\cap R\\ TS&S\\};
\path[->]
	(m2-1-1) edge  (m2-1-2)
		edge (m2-2-1);
\path[<-] 
	(m2-2-2) edge  (m2-1-2)
		edge (m2-2-1);
\end{tikzpicture}$$ is a double vector subbundle of $T(\mbb{F}\times\overline{\mbb{E}})$. Thus for any $$\big(X;(\lambda',\lambda)\big)\in TS\times_M \big((W'\times W)\cap R\big),$$ there exists a pair $(\tilde\lambda',\tilde\lambda)\in (L'\times L)\cap TR$ which is simultaneously mapped to both $X$ and $(\lambda',\lambda)$ by the respective bundle maps, as pictured in the following diagram:
$$\begin{tikzpicture}
\mmat{m1} at (-3,0) {(\tilde\lambda',\tilde\lambda) &(\lambda',\lambda)\\ X&\\};
\path[->]
	(m1-1-1) edge  (m1-1-2)
		edge (m1-2-1);

\draw (0,0) node {$\in$};

\mmat{m2} at (4,0) {\big(L'\times L\big)\cap TR &(W'\times W)\cap R\\ TS&S\\};
\path[->]
	(m2-1-1) edge  (m2-1-2)
		edge (m2-2-1);
\path[<-] 
	(m2-2-2) edge  (m2-1-2)
		edge (m2-2-1);
\end{tikzpicture}$$
Additionally, for any $(\sigma',\sigma)\in\Gamma(W'\times W)$ satisfying $(\sigma',\sigma)\rvert_S\in\Gamma\big((W'\times W)\cap R\big)$,  $$(\sigma'_T,\sigma_T)\rvert_{TS}\in\Gamma(TR,TS).$$ Consequently, since $TR$ is Lagrangian, \begin{equation}\label{eq:LieSubCompZero}\la(\sigma'_T,\sigma_T),(\tilde\lambda',\tilde\lambda)\ra=0.\end{equation}

 Since $(\tilde\lambda',\tilde\lambda)\in (L'\times L)$, we have $q_{L'\times L}(\tilde\lambda',\tilde\lambda)=0$. Hence, \cref{eq:LieSubCompZero,eq:restPair} imply that 
$$\la q_{L'\times L}(\sigma'_T,\sigma_T),(\lambda',\lambda)\ra=0.$$
 
 Combining this equality with \cref{eq:qLieSubCom} yields
 $$\la\big(-(p_N^*\nabla')_{X}\sigma',(p_M^*\nabla)_{X}\sigma\big), (\lambda',\lambda)\ra=0,$$
 which proves the proposition.
\end{proof}

Now we specialize to the case where $R:\mbb{E}\dasharrow \mbb{F}$ is a Courant morphism supported on the graph of a map $\phi:M\to N$. Suppose that $(W',\nabla')$ is a pseudo-Dirac structure in $\mbb{F}$ which composes cleanly with $R:\mbb{E}\dasharrow\mbb{F}$, and $\on{ran}(R)+ \phi^*W'^\perp=\phi^*\mbb{F}$. 

Since $(W',\nabla')$ composes cleanly with $R$, the intersection $\phi^*W'\cap \on{ran}(R)\subseteq \phi^*\mbb{F}$ is a smooth subbundle. Since $R$ is supported on the graph of a map $\phi:M\to N$, the subbundle
$$W:=W'\circ R\subseteq \mbb{E}$$ is well defined with support on all of $M$.
\begin{lemma}\label{lem:PullBackPseu}
There is a bundle map $\Psi:W\to \phi^* W'$, uniquely determined by the equation $$(\Psi(\lambda),\lambda)\in R_{(\phi(x),x)},$$ for every $\lambda\in W_x.$

Moreover, \cref{eq:nablaComposition} specializes to 
$$\nabla=\Psi^*\circ(\phi^*\nabla')\circ\Psi.$$ 
\end{lemma}
\begin{proof}
The proof of the first claim is entirely analogous to that of \cite[Lemma 7.2]{Bursztyn03-1} for Dirac realizations.
Consider the relation $R':W\dasharrow W'$, supported on the graph of $\phi:M\to N$, defined by
$$R':=R\cap (\phi^* W'\times W).$$ We will show that $R'$ is the graph of a bundle map $\Psi:W\to \phi^*W'$, i.e.  the natural projection $R'\to W$ is an isomorphism.

First, we establish surjectivity: let $\lambda\in W_x$. Since $W:=W'\circ R$, there exists some $\lambda'\in W'_{\phi(x)}$ such that the pair $(\lambda',\lambda)\in R_{(\phi(x),x)}$.

Next we establish injectivity. Suppose there exists $\lambda''\in W'_{\phi(x)}$ such that the pair $(\lambda'',\lambda)\in R_{(\phi(x),x)}$. Then $(\lambda''-\lambda',0)\in R\cap (\phi^*W'\times 0)$. However, 
$$\big(R\cap (\phi^*W'\times 0)\big)^\perp=(\on{ran}(R)+ \phi^*W'^\perp)\times \overline{\mbb{E}}=\mbb{F}\times\overline{\mbb{E}}.$$
Hence $R\cap (\phi^*W'\times 0)=0$ and $\lambda''=\lambda'$.

The second claim is an immediate consequence of the first claim.
\end{proof}

%


Let $L\subseteq T\mbb{E}$ and $L'\subseteq T\mbb{F}$ be the $\mc{VB}$-Dirac structures corresponding to the pseudo-Dirac structures $(W,\nabla)$ and $(W',\nabla')$. As an intersection of involutive subbundles, it is clear that $$TR\cap (L'\times L)=\gr(T\Psi)$$ is a Lie subalgebroid of $L'\times L$. In particular, the intersection $\gr(\Psi)=\gr(T\Psi)\cap(W'\times W)$ is a Lie subalgebroid of $W'\times W$. Thus we have shown:
\begin{corollary}\label{cor:PsiMorpLieAlg}
$$\Psi:W\to W'$$ is a morphism of Lie algebroids.
\end{corollary}

\subsection{Further Examples}

\begin{example}[q-Poisson structures and Cotangent Lie algebroids]\label{ex:qPCotLie}

Suppose that $\mf{d}$ is a quadratic Lie algebra, $\g\subset\mf{d}$ is a Lagrangian Lie subalgebra, and 
\begin{equation}\label{eq:qPstr}R:(\mbb{T}M,TM)\dasharrow (\mf{d},\g)\end{equation}
 is a morphism of Manin pairs \cite{Bursztyn:2009wi}.  That is,
 \begin{enumerate} 
\item[m1)] $TM\cap \on{ker}(R)=0$, 
\item[m2)] $R\circ TM\subseteq \g$. 
\end{enumerate}
Axioms (m1) and (m2) imply that for any $\xi\in\g$, there is a unique vector field $\rho(\xi)\in\mf{X}(M)$ such that
 $$(\rho(\xi),0)\sim_R \xi,\quad \xi\in\g.$$
Moreover, as explained in \cite{Bursztyn:2009wi}, $\rho:\g\times M\to TM$ defines an action of $\g$ on $M$.

It was shown in \cite[\S~3.2]{Bursztyn:2009wi} that once a Lagrangian complement $\mf{p}\subset\mf{d}$ to $\g$ has been chosen, the morphism of Manin pairs \labelcref{eq:qPstr} defines a q-Poisson structure on $M$ in the sense of \cite{Alekseev99}, for the Manin quasi-triple $(\mf{d},\g,\mf{p})$.
 Moreover, \cite[Proposition~6.1]{Xu95}   shows that whenever $\mf{p}$ is a Lagrangian subalgebra, $\mf{p}\circ R\subseteq \mbb{T}M$ is a Dirac structure (in fact it is the graph of a Poisson structure on $M$, cf. \cref{ex:StdCourAlgPoisSymp}). Using pseudo-Dirac geometry, we can generalize this fact by relaxing the requirement that $\mf{p}$ be Lagrangian:

If $\h\subset\mf{d}$ is any subalgebra transverse to $\g$, then $\on{ran}(R)+\h^\perp=\mf{d}$, so $F=\h\circ R\subseteq \mbb{T}M$ is a subbundle transverse to $TM$. Moreover, \cref{lem:PullBackPseu} shows that $(F,\nabla)$ is a pseudo-Dirac structure, where
\begin{equation}\label{eq:qPpseCon}\nabla\sigma=\rho \d\rho^*\sigma,\quad\sigma\in\Gamma(F),\end{equation} and we have used the metric to identify $\h\cong \g^*$.
Thus, any choice of Lie subalgebra $\h\subset\mf{d}$ transverse to $\g$ endows $T^*M\cong F$ with the structure of a Lie algebroid so that both 
\begin{align*}
\rho^*:T^*M&\to \h,\\
\rho:\g\times M&\to TM
\end{align*}
are morphisms of Lie algebroids.


Note that in the special case where $\mf{k}$ is a quadratic Lie algebra, $\mf{d}=\mf{k}\oplus\overline{\mf{k}}$, $\g=\mf{k}_\Delta$ is the diagonal subalgebra, and $\h=0\oplus\overline{\mf{k}}$,
this result was already proven in \cite[Theorem~1]{LiBland:2010wi}.
\end{example}

\begin{example}[pseudo-Dirac structure on the moduli-space of flat connections]
Suppose that $\mf{k}$ is a quadratic Lie algebra, and $\Sigma$ is a oriented surface with boundary, with one marked point on each boundary component.
\begin{center}
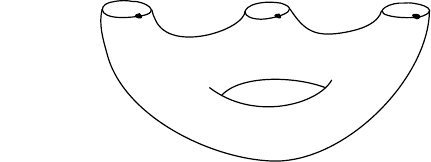
\end{center}

Let $\mc{M}_{\Sigma,\mf{k}}$  denote the space of flat $\mf{k}$ connections on $\Sigma$ modulo gauge transformations which act as the identity at the marked points. Then, as explained in \cite{Alekseev97,Bursztyn:2009wi}, there is a morphism of Manin pairs
$$R:(\mbb{T}\mc{M}_{\Sigma,\mf{k}}, T\mc{M}_{\Sigma,\mf{k}})\dasharrow(\mf{k}\oplus\overline{\mf{k}},\mf{k}_\Delta).$$
Let $K$ be the simply-connected Lie group corresponding to $\mf{k}$. If $\Sigma$ has $n$ boundary components labelled by $i=1,\dots,n$, then we have a `moment map' $$\Phi:\mc{M}_{\Sigma,\mf{k}}\to K^n,$$
whose $i^{th}$ component is the holonomy of a given connection around the $i^{th}$ boundary \cite{Alekseev97}.

Thus, as explained in \cref{ex:qPCotLie}, the subbundle
$$W_{\Sigma,\mf{k}}:=(0\oplus\overline{\mf{k}})\circ R\subseteq \mbb{T}\mc{M}_{\Sigma,\mf{k}}$$
is a pseudo-Dirac structure, with pseudo-connection given by \cref{eq:qPpseCon}, where 
$$\rho:\mf{k}\times\mc{M}_{\Sigma,\mf{k}}\to T\mc{M}_{\Sigma,\mf{k}}$$ 
is the residual action of the gauge group at the marked points.

The theory of \cref{sec:FBimage} shows that one can restrict this pseudo-Dirac structure to submanifolds $\mc{N}\subseteq\mc{M}_{\Sigma,\mf{k}}$. Indeed, $$(W_{\Sigma,\mf{k}},\nabla)\circ R_{\gr(\phi)}$$ is a pseudo-Dirac structure in $\mbb{T}\mc{N}$, where $$R_{\gr(\phi)}:\mbb{T}\mc{N}\dasharrow\mbb{T}\mc{M}_{\Sigma,\mf{k}}$$ is the Courant morphism described in \cref{ex:StdDiracStr} and $\phi:\mc{N}\to\mc{M}_{\Sigma,\mf{k}}$ is the embedding. 

Suppose for example we wish to study the moduli space $\mc{M}_{\Sigma',\mf{k}}$, where $\Sigma'$ is the surface constructed from $\Sigma$ by sewing the $i^{th}$ boundary component along the $j^{th}$ (in an orientation preserving way). Let 
$$\mc{N}=\Phi^{-1}K^n_{\Delta_{ij}},$$ where $$K^n_{\Delta_{i,j}}:=\{(k_1,\cdots,k_n)\in K^n\mid k_i=k_j^{-1}\}.$$ That is, $\mc{N}$ is the subspace of $\mc{M}$, determined by requiring that the holonomy around the $i^{th}$ and $j^{th}$ boundary components coincide up to inversion. Acting diagonally by the residual gauge group at the $i^{th}$ and $j^{th}$ marked points gives a surjective submersion $q:\mc{N}\to \mc{M}_{\Sigma',\mf{k}}$.
What is interesting is that one may show the two pseudo-Dirac structures are related by
\begin{equation}\label{eq:ModSpceRed}(W_{\Sigma',\mf{k}},\nabla)=(W_{\Sigma,\mf{k}},\nabla)\circ R_{\gr(\phi)}\circ R_{\gr(q)}^\top.\end{equation}
Moreover, $$(W_{\Sigma,\mf{k}}\circ R_\phi)\cap T\mc{N}$$ is precisely the tangent space to the fibres of the submersion $q:\mc{N}\to \mc{M}_{\Sigma',\mf{k}}$.

One may interpret \cref{eq:ModSpceRed} as saying that the pseudo-Dirac structure on $\mc{M}_{\Sigma',\mf{k}}$ is the reduction of the pseudo-Dirac structure on $\mc{M}_{\Sigma,\mf{k}}$ at moment level $K^n_{\Delta_{ij}}$.

 Generalizing this, suppose $S\subseteq K^n$ is a submanifold. Let $\mc{N}=\Phi^{-1}(S)$, the map $\phi:\mc{N}\to\mc{M}$ denote the inclusion, and $$\mc{F}= (W_{\Sigma,\mf{k}}\circ R_\phi)\cap T\mc{N}.$$  Let $\mc{Q}$ denote the leaf space of $\mc{F}$, and $q:\mc{N}\to\mc{Q}$ the quotient map. Then (assuming everything exists and is  smooth etc.) we define 
$$(W_{\Sigma,\mf{k}},\nabla)\circ R_{\gr(\phi)}\circ R_{\gr(q)}^\top$$
to be the reduction of the pseudo-Dirac structure on $\mc{M}_{\Sigma,\mf{k}}$ at moment level $S\subseteq K^n$.

One may construct interesting moduli spaces by choosing different moment levels. For instance, one may require that the holonomy around the $i^{th}$ boundary component lie in a specific conjugacy class $C_i\subseteq K$, i.e. $S=C_1\times\cdots \times C_n$. Then the reduction at moment level $S$ is the symplectic structure on the moduli space of flat $\mf{k}$ connections on $\Sigma$ whose holonomies around the various boundary components lie in the specified conjugacy classes \cite{Alekseev97}.



We will generalize this example, and explain it in greater detail in a future paper.
\end{example}

\begin{example}[Transverse pseudo-Dirac structures]\label{ex:TransvPsDir}
Suppose that $\mbb{E}\to M$ is a Courant algebroid and $(E,\nabla^E)$ and $(F,\nabla^F)$ are two complementary pseudo-Dirac structures in $\mbb{E}$ (that is, $\mbb{E}=E\oplus F$ as a vector bundle). Then $(E\oplus F,\nabla^E\oplus(-\nabla^F))$ is a pseudo-Dirac structure in $\mbb{E}\times\overline{\mbb{E}}$ transverse to the diagonal $\mbb{E}_\Delta$. It follows that the Courant morphism
$$R_{\on{diag}}:\mbb{T}M\dasharrow \mbb{E}\times\overline{\mbb{E}}$$
described in \cref{ex:DiagMorph} composes cleanly with $(E\oplus F,\nabla^E\oplus(-\nabla^F))$. 
Therefore, $$(W,\nabla):=(E\oplus F,\nabla^E\oplus(-\nabla^F))\circ R_{\on{diag}}$$ is a pseudo-Dirac structure in $\mbb{T}M$. 

Suppose $\mu\in T^*M$, if
$$\mbf{a}^*\mu=x-y,\quad x\in E,\quad y\in F,$$
then $x$ is the image of $\mu$ under the map 
\begin{equation}\label{eq:TransMorph1}T^*M\xrightarrow{\mbf{a}^*}\mbb{E}\to \mbb{E}/F\cong E.\end{equation}
Since $$\mbb{E}=E\oplus F\quad\Leftrightarrow \quad\mbb{E}=E^\perp\oplus F^\perp,$$ we have $E^*\cong F^\perp$ and $F^*\cong E^\perp$. 
In particular, $E\cong\mbb{E}/F\cong (F^\perp)^*$. 
Therefore, the composition \cref{eq:TransMorph1} is dual to the map $\mbf{a}\rvert_{F^\perp}:F^\perp\to TM$.
Thus, $$\mbf{a}(x)=\mbf{a}\rvert_E\circ \mbf{a}\rvert_{F^\perp}^*(\mu).$$
Comparison with \cref{eq:DiagMorph2} shows that $$\big( \mbf{a}\rvert_E\circ \mbf{a}\rvert_{F^\perp}^*(\mu),\mu\big)\sim_{R_{\on{diag}}} (x,y),$$
where $x=\mbf{a}\rvert_{F^\perp}^*(\mu)\in E$ and $y=\mbf{a}\rvert_{E^\perp}^*(\mu)\in F$.

It follows that 
$$W=\gr\bigg((\mbf{a}\rvert_E)\circ (\mbf{a}\rvert_{F^\perp})^*:T^*M\to TM\bigg).$$
Or, $$W=\{(\mbf{a}(e_i)\la\mbf{a}(e^i),\mu\ra,\mu)\mid \mu\in T^*M\},$$
where $\{e_i\}$ is a local basis of sections for $E$ and $\{e^i\}$ is a local basis for $F^\perp$ in duality.

We may compute the pseudo-connection $\nabla$ in terms of $\nabla^E$ and $\nabla^F$ using \cref{lem:PullBackPseu}. In this context, the bundle map $\Psi:W\to E\times_M F$ is simply
$$\Psi:\mu\to (\mbf{a}\rvert_{F^\perp}^*\mu,\mbf{a}\rvert_{E^\perp}^*\mu),\quad \mu\in T^*M$$
where we have used the identifications $W\cong T^*M$, $E\cong (F^\perp)^*$ and $F\cong (E^\perp)^*$. 
Thus \cref{lem:PullBackPseu} implies that for any section $\alpha\in\Gamma(T^*M)\cong \Gamma(W)$, and $X\in \mf{X}(M)$,
$$\nabla_X\alpha=\mbf{a}_{F^\perp}\big(\nabla^E_X(\mbf{a}\rvert_{F^\perp}^*\alpha)\big)+\mbf{a}_{E^\perp}\big(\nabla^F_X(\mbf{a}\rvert_{E^\perp}^*\alpha)\big).$$

Moreover, \cref{cor:PsiMorpLieAlg} implies that 
$$\Psi:W\to E\times F$$ is a morphism of Lie algebroids over the diagonal embedding $M\to M\times M$. In particular, the leaves of $W$ must lie in the intersections between the leaves of $E$ with the leaves of $F$.

\end{example}

\begin{example}[Decompositions of action Courant algebroids]\label{ex:ActCourDec}
Suppose a quadratic Lie algebra, $\mf{d}$, acts on $M$ with coisotropic stablizers, so that
$$\mbb{E}:=\mf{d}\times M$$ 
is an action Courant algebroid \cite{LiBland:2009ul}. 
Suppose we have a vector space decomposition $\mf{d}=\mf{e}\oplus\mf{f}$, where $\mf{e},\mf{f}\subseteq\mf{d}$ are subalgebras. Note that we do not assume that $\mf{e},\mf{f}\subseteq\mf{d}$ are either ideals or Lagrangian subalgebras. In the literature, one says that $\mf{e}$ and $\mf{f}$ form a matched pair, and $\mf{d}=\mf{e}\bowtie\mf{f}$ is the double. 

\Cref{ex:LieSubAct} shows that
$(E,\d)$ and $(F,\d)$ are complementary pseudo-Dirac structures in $\mbb{E}$, where
$$E:=\mf{e}\times M,\quad F:=\mf{f}\times M,$$
and $\d$ is the standard flat connection on these trivial bundles.

Let $\rho:\mf{d}\times  M\to T M$ be the action map. Using the identification $ \mf{e}^*\cong \mf{f}^\perp$,  let $\{e_i\}$ and $\{e^i\}$ be bases for $\mf{e}$ and $\mf{f}^\perp$ in duality. Similarly, using the identification $ \mf{f}^*\cong \mf{e}^\perp$, let $\{f_i\}$ and $\{f^i\}$ be bases for $\mf{f}$ and $\mf{e}^\perp$ in duality.
 Let $$\tilde e_i:= \sum_j e^j\la e_j,e_i\ra$$
 denote the image of $e_i\in\mf{e}$ under the projection of $\mf{e}$ to $\mf{f}^\perp$ along $\mf{e}^\perp$. Similarly, let 
 $$\tilde f_i:= \sum_j f^j\la f_j,f_i\ra$$ 
 denote the image of $f_i\in\mf{f}$ under the projection of $\mf{f}$ to $\mf{e}^\perp$ along $\mf{f}^\perp$.
 Then \cref{ex:TransvPsDir} shows that the subbundle
$$W:=\big\{\big(\sum_i \rho(e_i)\mu(\rho(e^i)),\mu\big)\mid \mu\in T^* M\big\}
\subseteq \mbb{T} M$$
together with the pseudo-connection
$$\nabla\alpha=
\sum_i\rho(\tilde e_i)\d\iota_{\rho(e^i)}\alpha+\sum_i\rho(\tilde f_i) \d\iota_{\rho(f^i)}\alpha
\quad \alpha\in \Omega^1( M),$$
defines a pseudo-Dirac structure in $\mbb{T} M$ (we have implicitly used the canonical identification $W\cong T^*M$ is the formula above). In particular,  this endows $T^* M\cong W$ with a Lie-algebroid structure.

Moreover, \cref{cor:PsiMorpLieAlg} implies that the map
$$\Psi:T^*M\to (\mf{e}\times M)\times (\mf{f}\times M)$$ given by
$$\Psi:\mu\to (\rho\rvert_{\mf{f}^\perp}^*\mu,\rho\rvert_{\mf{e}^\perp}^*\mu)$$
 is a morphism of Lie algebroids over the diagonal embedding $M\to M\times M$.

When $\mf{e}=\mf{e}^\perp$ and $\mf{f}=\mf{f}^\perp$, $\tilde e_i=0$ and $\tilde f_i=0$, so $\nabla=0$. Therefore $W\subset \mbb{T}M$ is just a Dirac structure. In fact it is the graph of the Poisson structure constructed by Lu and Yakimov for the action of the Manin triple $(\mf{d},\mf{e},\mf{f})$ on $M$ \cite{Lu06,LiBland:2009ul}.

When one takes $\mf{e}=0$ and $\mf{f}=\mf{d}$, 
$$W=\{(0,\mu)\mid \mu \in T^*M\}.$$
 This shows that $T^*M$ is a bundle of Lie algebras, and $\rho^*:T^*_mM\to \mf{d}$ is a morphism of Lie algebras for every $m\in M$ - a result first proven in \cite{LiBland:2010wi}. 
\end{example}

\begin{example}[pseudo-Dirac structures on the variety of Lagrangian subalgebras]
Let $\mf{d}$ be a quadratic Lie algebra. In \cite{Evens:2001ue,Evens:2006kk} it was shown any vector space decomposition $\mf{d}=\mf{e}\oplus\mf{f}$ where $\mf{e},\mf{f}\subseteq\mf{d}$ are Lagrangian subalgebras induces a Poisson structure on the variety
$$\mc{L}_\mf{d}:=\{\mf{l}\subseteq\mf{d}\mid\mf{l} \text{ is a Lagrangian subalgebra}\},$$
of Lagrangian subalgebras of $\mf{d}$. A more general result is possible using pseudo-Dirac geometry:

Since the stabilizers of the natural $\mf{d}$ action on $\mc{L}_\mf{d}$ are coisotropic, 
$$\mbb{E}:=\mf{d}\times\mc{L}_\mf{d}$$ 
is an action Courant algebroid \cite{LiBland:2009ul}. 
Therefore \cref{ex:ActCourDec} shows that  any vector space decomposition $\mf{d}=\mf{e}\oplus\mf{f}$, where $\mf{e},\mf{f}\subseteq\mf{d}$ are subalgebras, endows $T^*\mf{L}_\mf{d}$ with the structure of a Lie algebroid. (For example, one can always take $\mf{e}=0$, $\mf{f}=\mf{d}$). In the special case where $\mf{e}$ and $\mf{f}$ are both Lagrangian, this is the Lie algebroid structure on $T^*M$ arising from Lu and Evens' Poisson structure.

\end{example}

\begin{example}[pseudo-Dirac Lie groups]
Suppose that $\mf{d}$ is a quadratic Lie algebra, and $\mf{g},\mf{h}\subseteq\mf{d}$ are two subalgebras such that $\mf{d}=\mf{g}\oplus\mf{h}$. Let $G$ and $D$ be the simply connected Lie groups integrating $\mf{g}$ and $\mf{d}$ respectively, and $$\phi:G\to D$$ the immersion of Lie groups integrating the inclusion $\mf{g}\subseteq\mf{d}$.

The quadratic Lie algebra $\overline{\mf{d}}\oplus\mf{d}$ acts on $D$ via the map $$\rho:(\xi,\eta)\to -\xi^R+\eta^L,$$ where $\xi^R,\xi^L\in\mf{X}(D)$ are the right/left invariant vector fields with value $\xi\in\mf{d}$ at the identity. This action has coisotropic stabilizers, and hence 
$$(\overline{\mf{d}}\oplus\mf{d})\times D$$
is an action Courant algebroid \cite[Example 4.2]{LiBland:2009ul}. Meanwhile \cref{ex:LieSubAct} shows that $$\big((\h\oplus\h)\times D,\d\big)$$ is a pseudo-Dirac structure, where $\d$ is the standard flat connection on the trivial bundle $(\h\oplus\h)\times D$.

%
 As shown in \cite{Alekseev:2009tg}, there is a canonical isomorphism
  $$(\overline{\mf{d}}\oplus\mf{d})\times D\cong \mbb{T}_\eta D,$$
 where $\eta=\frac{1}{12}\la[\theta^L,\theta^L],\theta^L\ra\in\Omega^3_{cl}(D)$ is the Cartan 3-form (here $\theta^L\in\Omega^1(D,\mf{d})$ is the left invariant Maurer-Cartan form).
 Recall the Courant morphism 
 $$R_{\gr(\phi)}:\mbb{T}_{\phi^*\eta}G\dasharrow \mbb{T}_\eta D$$
 described in \cref{ex:StdDiracStr}. The theory of \cref{sec:FBimage} shows that
 $$(W,\nabla):=\big((\h\oplus\h)\times D,\d\big)\circ R_{\gr(\phi)}$$ is a pseudo Dirac structure in $\mbb{T}_{\phi^*\eta}G$.
 
 In conclusion, we have constructed a pseudo-Dirac structure $(W,\nabla)$ in $\mbb{T}_{\phi^*\eta}G$, corresponding to the matched pair $\mf{d}=\g\bowtie\h$ of Lie algebras and the invariant quadratic form on $\mf{d}$. 
 
 When $(\mf{d},\g,\h)$ is a Manin triple,
  $\phi^*\eta=0$, $\nabla=0$, and $W=\gr(\pi^\sharp)$ is the graph of the Poisson structure $\pi\in\mf{X}^2(G)$, which Drinfel'd associates to the Lie bialgebra $\g,\g^*\cong\h$ \cite{Drinfeld83} (our construction of this Poisson structure is identical to the one found in \cite{LiBland:2009ul}).
 
 When only $\h\subseteq\mf{d}$ is Lagrangian then the pseudo-connection $\nabla$ is still trivial, and $W\subseteq \mbb{T}_{\phi^*\eta}G$ coincides with the Dirac Lie group structure described in \cite{LiBland:2011vqa,LiBland:2010wi} corresponding to the Dirac Manin triple $(\mf{d},\h;\g)_{\la\cdot,\cdot\ra}$.
\end{example}

\begin{appendix}
\section{Total kernels and quotients of double vector bundles}\label{app:TotQuo}

Suppose that the  map
\begin{equation}\label{eq:TotQuo}
\begin{tikzpicture}
\mmat{m1} at (-2,0){D&B\\ A& M\\};
\path[->] (m1-1-1)	edge (m1-1-2)
				edge (m1-2-1);
\path[<-] (m1-2-2)	edge (m1-1-2)
				edge (m1-2-1);
\mmat{m2} at (2,0) {Q&M\\ M& M\\};
\path[->] (m2-1-1)	edge (m2-1-2)
				edge (m2-2-1);
\path[<-] (m2-2-2)	edge node[swap]{$\on{id}$} (m2-1-2)
				edge node{$\on{id}$} (m2-2-1);
\path[->] (m1) edge node {$q$} (m2);
\draw (-2,0) node (c) {$C$};
\path[left hook->] (c) edge (m1-1-1);
\draw (2,0) node (c2) {$Q$};
\path[left hook->] (c2) edge (m2-1-1);
\end{tikzpicture}
\end{equation}
is a morphism of double vector bundles
such that the restriction of $q$ to the core $C$ of $D$ is a surjection.
Since $Q$ has a common zero section over both side bundles, it makes sense to define the \emph{total kernel} $$\on{ker}(q):=\{x\in D \mid q(x)=0\},$$
a double vector subbundle of $D$. 

On the other hand, suppose we have a double vector subbundle of $D$,
\begin{equation}\label{eq:TotKer}
\begin{tikzpicture}
\mmat{m1} at (-2,0){D'&B\\ A& M\\};
\path[->] (m1-1-1)	edge (m1-1-2)
				edge (m1-2-1);
\path[<-] (m1-2-2)	edge (m1-1-2)
				edge (m1-2-1);
\mmat{m2} at (2,0) {D&B\\ A& M\\};
\path[->] (m2-1-1)	edge (m2-1-2)
				edge (m2-2-1);
\path[<-] (m2-2-2)	edge (m2-1-2)
				edge (m2-2-1);
\draw (0,0) node {$\subseteq$} (m2);
\draw (-2,0) node (c) {$C'$};
\path[left hook->] (c) edge (m1-1-1);
\draw (2,0) node (c2) {$C$};
\path[left hook->] (c2) edge (m2-1-1);
\end{tikzpicture}
\end{equation}
%
containing both side bundles $A$ and $B$. 

Note that the abelian groupoid $C$ acts on $D$, via the formula
\begin{equation}\label{eq:coreaction}d+c:=(d+_{D/A}0_{D/A})+_{D/B}(c+_{D/A}0_{D/B})=(d+_{D/B}0_{D/B})+_{D/A}(c+_{D/B}0_{D/A}).\end{equation}
 for any $(d,c)\in D\times_M C$.
Consequently, we have a well defined map $q:D\to C/C'$, given by the condition 
\begin{equation}\label{eq:totQuotMap}q(d)=-\tilde c\Leftrightarrow d+c\in D',\end{equation} where $c\in C$ is any lift of $\tilde c\in C/C'$. The map $q:D\to C/C'$ is called the \emph{total quotient} of $D$ by $D'$. 

It is clear that these two constructions invert each other, yielding the following proposition.

\begin{proposition}\label{prop:totkerVtotquo}
The operations of taking the total kernel or the total quotient define a one-to-one correspondence between surjective morphisms of double vector bundles of the form \labelcref{eq:TotQuo} and inclusions of the form \labelcref{eq:TotKer}.
\end{proposition}

\begin{remark}\label{rem:DblQuot}
Alternatively, one may understand the map $q:D\to C/C'$ as taking the double quotient (in both the vertical and horizontal directions) of the double vector bundle $D$ by $D'$. Indeed, one should interpret the right hand side of \cref{eq:totQuotMap} as saying that $d$ and $-c$ differ by an element of $D'$, and hence project to the same element, $-\tilde c$, under the double quotient.
%
\end{remark}

\end{appendix}

\bibliography{basicbib}{}
\bibliographystyle{plain}
\end{document}

%% file: CoreElement.pdf_tex

\begingroup
  \makeatletter
  \providecommand\color[2][]{%
    \errmessage{(Inkscape) Color is used for the text in Inkscape, but the package 'color.sty' is not loaded}
    \renewcommand\color[2][]{}%
  }
  \providecommand\transparent[1]{%
    \errmessage{(Inkscape) Transparency is used (non-zero) for the text in Inkscape, but the package 'transparent.sty' is not loaded}
    \renewcommand\transparent[1]{}%
  }
  \providecommand\rotatebox[2]{#2}
  \ifx\svgwidth\undefined
    \setlength{\unitlength}{172.73549288pt}
  \else
    \setlength{\unitlength}{\svgwidth}
  \fi
  \global\let\svgwidth\undefined
  \makeatother
  \begin{picture}(1,0.94217106)%
    \put(0,0){\includegraphics[width=\unitlength]{CoreElement.pdf}}%
    \put(0.9254686,0.39781322){\color[rgb]{0,0,0}\makebox(0,0)[b]{\smash{$M$}}}%
    \put(0.92740634,0.6256717){\color[rgb]{0,0,0}\makebox(0,0)[b]{\smash{$B$}}}%
    \put(0.34627331,0.88233427){\color[rgb]{0,0,0}\makebox(0,0)[b]{\smash{$q_B^{-1}(x)$}}}%
    \put(0.42925149,0.3742551){\color[rgb]{0,0,0}\makebox(0,0)[b]{\smash{$x$}}}%
    \put(0.51612416,0.56120157){\color[rgb]{0,0,0}\makebox(0,0)[b]{\smash{$\xi$}}}%
  \end{picture}%
\endgroup

%% file: CoreSection.pdf_tex

\begingroup
  \makeatletter
  \providecommand\color[2][]{%
    \errmessage{(Inkscape) Color is used for the text in Inkscape, but the package 'color.sty' is not loaded}
    \renewcommand\color[2][]{}%
  }
  \providecommand\transparent[1]{%
    \errmessage{(Inkscape) Transparency is used (non-zero) for the text in Inkscape, but the package 'transparent.sty' is not loaded}
    \renewcommand\transparent[1]{}%
  }
  \providecommand\rotatebox[2]{#2}
  \ifx\svgwidth\undefined
    \setlength{\unitlength}{182.70631808pt}
  \else
    \setlength{\unitlength}{\svgwidth}
  \fi
  \global\let\svgwidth\undefined
  \makeatother
  \begin{picture}(1,0.69026244)%
    \put(0,0){\includegraphics[width=\unitlength]{CoreSection.pdf}}%
    \put(0.1672855,0.18999902){\color[rgb]{0,0,0}\makebox(0,0)[b]{\smash{$X$}}}%
    \put(0.30218905,0.5792288){\color[rgb]{0,0,0}\makebox(0,0)[b]{\smash{$\sigma_C(X)$}}}%
    \put(0.85949531,0.51730588){\color[rgb]{0,0,0}\makebox(0,0)[b]{\smash{$\sigma$}}}%
  \end{picture}%
\endgroup

%% file: TangentLiftSection.pdf_tex

\begingroup
  \makeatletter
  \providecommand\color[2][]{%
    \errmessage{(Inkscape) Color is used for the text in Inkscape, but the package 'color.sty' is not loaded}
    \renewcommand\color[2][]{}%
  }
  \providecommand\transparent[1]{%
    \errmessage{(Inkscape) Transparency is used (non-zero) for the text in Inkscape, but the package 'transparent.sty' is not loaded}
    \renewcommand\transparent[1]{}%
  }
  \providecommand\rotatebox[2]{#2}
  \ifx\svgwidth\undefined
    \setlength{\unitlength}{245.88376465pt}
  \else
    \setlength{\unitlength}{\svgwidth}
  \fi
  \global\let\svgwidth\undefined
  \makeatother
  \begin{picture}(1,0.6635267)%
    \put(0,0){\includegraphics[width=\unitlength]{TangentLiftSection.pdf}}%
    \put(0.88871755,0.27754201){\color[rgb]{0,0,0}\makebox(0,0)[b]{\smash{$M$}}}%
    \put(0.89549009,0.14094038){\color[rgb]{0,0,0}\makebox(0,0)[b]{\smash{$B$}}}%
    \put(0.50061312,0.62149083){\color[rgb]{0,0,0}\makebox(0,0)[b]{\smash{$q_B^{-1}(x)$}}}%
    \put(0.5424447,0.26243616){\color[rgb]{0,0,0}\makebox(0,0)[b]{\smash{$x$}}}%
    \put(0.89559663,0.47893739){\color[rgb]{0,0,0}\makebox(0,0)[b]{\smash{$\sigma$}}}%
    \put(0.36349836,0.13461733){\color[rgb]{0,0,0}\makebox(0,0)[b]{\smash{$X$}}}%
    \put(0.1775801,0.52039777){\color[rgb]{0,0,0}\makebox(0,0)[b]{\smash{$\sigma_T(X)$}}}%
  \end{picture}%
\endgroup

%% file: surfaceWBound.pdf_tex
\begingroup%
  \makeatletter%
  \providecommand\color[2][]{%
    \errmessage{(Inkscape) Color is used for the text in Inkscape, but the package 'color.sty' is not loaded}%
    \renewcommand\color[2][]{}%
  }%
  \providecommand\transparent[1]{%
    \errmessage{(Inkscape) Transparency is used (non-zero) for the text in Inkscape, but the package 'transparent.sty' is not loaded}%
    \renewcommand\transparent[1]{}%
  }%
  \providecommand\rotatebox[2]{#2}%
  \ifx\svgwidth\undefined%
    \setlength{\unitlength}{123.91615294bp}%
    \ifx\svgscale\undefined%
      \relax%
    \else%
      \setlength{\unitlength}{\unitlength * \real{\svgscale}}%
    \fi%
  \else%
    \setlength{\unitlength}{\svgwidth}%
  \fi%
  \global\let\svgwidth\undefined%
  \global\let\svgscale\undefined%
  \makeatother%
  \begin{picture}(1,0.37572951)%
    \put(0,0){\includegraphics[width=\unitlength]{surfaceWBound.pdf}}%
    \put(-0.00535897,0.16277782){\color[rgb]{0,0,0}\makebox(0,0)[lb]{\smash{$\Sigma=$}}}%
  \end{picture}%
\endgroup%

%% file: PseudoDirStr.bbl
\begin{thebibliography}{10}

\bibitem{Alekseev:2009tg}
Anton~Yu Alekseev, Henrique Bursztyn, and Eckhard Meinrenken.
\newblock {Pure spinors on Lie groups}.
\newblock {\em Ast\'erisque}, (327):131--199 (2010), 2009.

\bibitem{Alekseev99}
Anton~Yu Alekseev and Yvette Kosmann-Schwarzbach.
\newblock {Manin pairs and moment maps}.
\newblock {\em Journal of Differential Geometry}, 56(1):133--165, 2000.

\bibitem{Alekseev97}
Anton~Yu Alekseev, Anton~Z. Malkin, and Eckhard Meinrenken.
\newblock {Lie group valued moment maps}.
\newblock {\em Journal of Differential Geometry}, 48(3):445--495, 1998.

\bibitem{Alekseev:2002tn}
Anton~Yu Alekseev and Ping Xu.
\newblock {Derived brackets and Courant algebroids}.
\newblock Available at \url{http://www.math.psu.edu/ping/anton-final.pdf},
  2002.

\bibitem{Boumaiza:2009eg}
Mohamed Boumaiza and Nadhem Zaalani.
\newblock {Rel{\`e}vement d'une alg{\'e}bro{\"\i}de de Courant}.
\newblock {\em Comptes Rendus Math{\'e}matique. Acad{\'e}mie des Sciences.
  Paris}, 347(3-4):177--182, February 2009.

\bibitem{Bursztyn:2010wb}
Henrique Bursztyn and Alejandro Cabrera.
\newblock {Multiplicative forms at the infinitesimal level}.
\newblock {\em Mathematische Annalen}, 353(3):663--705, 2012.

\bibitem{Bursztyn03-1}
Henrique Bursztyn, Marius Crainic, Alan Weinstein, and Chenchang Zhu.
\newblock {Integration of twisted Dirac brackets}.
\newblock {\em Duke Mathematical Journal}, 123(3):549--607, 2004.

\bibitem{Bursztyn:2009wi}
Henrique Bursztyn, David Iglesias~Ponte, and Pavol {\v S}evera.
\newblock {Courant morphisms and moment maps}.
\newblock {\em Mathematical Research Letters}, 16(2):215--232, 2009.

\bibitem{Courant:1990uy}
Theodore~James Courant.
\newblock {Dirac manifolds}.
\newblock {\em Transactions of the American Mathematical Society},
  319(2):631--661, 1990.

\bibitem{Courant:1999ho}
Theodore~James Courant.
\newblock {Tangent Dirac structures}.
\newblock {\em Journal of Physics A: Mathematical and General},
  23(22):5153--5168, January 1999.

\bibitem{Courant:tm}
Theodore~James Courant and Alan Weinstein.
\newblock {Beyond Poisson structures}.
\newblock In {\em Action hamiltoniennes de groupes. Troisi\`eme th\'eor\`eme de
  Lie (Lyon, 1986)}, pages 39--49. Hermann, Paris, 1988.

\bibitem{Crainic:2012th}
Marius Crainic, Maria~Amelia Salazar, and Ivan Struchiner.
\newblock {Multiplicative Forms and Spencer Operators}.
\newblock October 2012.
\newblock Available at \url{http://arxiv.org/abs/1210.2277}.

\bibitem{Dorfman:1993us}
Irene Dorfman.
\newblock {\em {Dirac structures and integrability of nonlinear evolution
  equations}}.
\newblock Nonlinear Science - theory and applications. John Wiley {\&} Son Ltd,
  July 1993.

\bibitem{Drinfeld83}
Vladimir~Gershonovich Drinfel'd.
\newblock {Hamiltonian structures on Lie groups, Lie bialgebras and the
  geometric meaning of classical Yang-Baxter equations}.
\newblock {\em Doklady Akademii Nauk SSSR}, 268(2):285--287, 1983.

\bibitem{Drummond:2013wr}
Thiago Drummond, Madeleine Jotz, and Cristi{\'a}n Ortiz.
\newblock {VB-algebroid morphisms and representations up to homotopy}.
\newblock February 2013.
\newblock Available at \url{http://arxiv.org/abs/1302.3987}.

\bibitem{Ehresmann:1995ts}
Charles Ehresmann.
\newblock {Les connexions infinit\'esimales dans un espace fibr\'e
  diff\'erentiable}.
\newblock In {\em S\'eminaire Bourbaki. Vol. 1}, pages Exp.\ No.\ 24, 153--168.
  Soc. Math. France, Paris, 1995.

\bibitem{Evens:2001ue}
Sam Evens and Jiang-Hua Lu.
\newblock {On the variety of Lagrangian subalgebras. I}.
\newblock {\em Annales Scientifiques de l'\'Ecole Normale Sup\'erieure.
  Quatri\`eme S\'erie}, 34(5):631--668, 2001.

\bibitem{Evens:2006kk}
Sam Evens and Jiang-Hua Lu.
\newblock {On the variety of Lagrangian subalgebras. II}.
\newblock {\em Annales Scientifiques de l'\'Ecole Normale Sup\'erieure.
  Quatri\`eme S\'erie}, 39(2):347--379, 2006.

\bibitem{Grabowski:2009dc}
Janusz Grabowski and Mikolaj Rotkiewicz.
\newblock {Higher vector bundles and multi-graded symplectic manifolds}.
\newblock {\em Journal of Geometry and Physics}, 59(9):1285--1305, 2009.

\bibitem{Grabowski:1997vy}
Janusz Grabowski and Pawel Urba{\'n}ski.
\newblock {Tangent and cotangent lifts and graded Lie algebras associated with
  Lie algebroids}.
\newblock {\em Annals of Global Analysis and Geometry}, 15(5):447--486, 1997.

\bibitem{GraciaSaz:2009ck}
Alfonso Gracia-Saz and Kirill Charles~Howard Mackenzie.
\newblock {Duality Functors for Triple Vector Bundles}.
\newblock {\em Letters in Mathematical Physics}, 90(1-3):175--200, September
  2009.

\bibitem{gracia2010lie}
Alfonso Gracia-Saz and Rajan~Amit Mehta.
\newblock {Lie algebroid structures on double vector bundles and representation
  theory of Lie algebroids}.
\newblock {\em Advances in Mathematics}, 223(4):1236--1275, 2010.

\bibitem{Greub:1973vs}
Werner~Hildbert. Greub, Stephen Halperin, and Ray Vanstone.
\newblock {\em {Connections, curvature, and cohomology. Vol. II: Lie groups,
  principal bundles, and characteristic classes}}.
\newblock Academic Press [A subsidiary of Harcourt Brace Jovanovich,
  Publishers], New York-London, 1973.

\bibitem{Gualtieri:2004wh}
Marco Gualtieri.
\newblock {\em {Generalized complex geometry}}.
\newblock PhD thesis, University of Oxford, November 2003.

\bibitem{Gualtieri:2007wh}
Marco Gualtieri.
\newblock {Generalized complex geometry}.
\newblock {\em Annals of Mathematics. Second Series}, 174(1):75--123, 2011.

\bibitem{Higgins:1990gq}
Philip~J. Higgins and Kirill Charles~Howard Mackenzie.
\newblock {Algebraic constructions in the category of Lie algebroids}.
\newblock {\em Journal of Algebra}, 129(1):194--230, 1990.

\bibitem{Higgins:1993bc}
Philip~J. Higgins and Kirill Charles~Howard Mackenzie.
\newblock {Duality for base-changing morphisms of vector bundles, modules, Lie
  algebroids and Poisson structures}.
\newblock {\em Mathematical Proceedings of the Cambridge Philosophical
  Society}, 114(3):471--488, 1993.

\bibitem{PonteXu:08}
David Iglesias~Ponte and Ping Xu.
\newblock {Hamiltonian spaces for Manin pairs over manifolds}.
\newblock Available at \url{http://arxiv.org/abs/0809.4070}, 2008.

\bibitem{Jotz:2009va}
Madeleine Jotz.
\newblock {Dirac Lie groups, Dirac homogeneous spaces and the theorem of
  Drinfeld}.
\newblock {\em Indiana University Mathematics Journal}, 60(1):319--366, 2011.

\bibitem{Jotz:2011wq}
Madeleine Jotz and Cristi{\'a}n Ortiz.
\newblock {Foliated groupoids and their infinitesimal data}.
\newblock September 2011.
\newblock Available at \url{http://arxiv.org/abs/1109.4515}.

\bibitem{Jotz:2008wn}
Madeleine Jotz and Tudor~S. Ratiu.
\newblock {Dirac structures, nonholonomic systems and reduction}.
\newblock {\em Reports on Mathematical Physics}, 69(1):5--56, 2012.

\bibitem{Konieczna:1999vh}
Katarzyna Konieczna and Pawel Urba{\'n}ski.
\newblock {Double vector bundles and duality}.
\newblock {\em Universitatis Masarykianae Brunensis. Facultas Scientiarum
  Naturalium. Archivum Mathematicum}, 35(1):59--95, 1999.

\bibitem{LiBland:2012un}
David Li-Bland.
\newblock {\em {LA-Courant algebroids and their applications}}.
\newblock PhD thesis, University of Toronto, Toronto, April 2012.

\bibitem{LiBland:2009ul}
David Li-Bland and Eckhard Meinrenken.
\newblock {Courant algebroids and Poisson geometry}.
\newblock {\em International Mathematics Research Notices},
  2009(11):2106--2145, 2009.

\bibitem{LiBland:2011vqa}
David Li-Bland and Eckhard Meinrenken.
\newblock {Dirac Lie Groups}.
\newblock pages 1--46, October 2011.
\newblock Available at \url{http://arxiv.org/abs/1110.1525}.

\bibitem{LiBland:2010wi}
David Li-Bland and Pavol {\v S}evera.
\newblock {Quasi-Hamiltonian groupoids and multiplicative Manin pairs}.
\newblock {\em International Mathematics Research Notices},
  2011(10):2295--2350, 2011.

\bibitem{ManinTriplesBi}
Zhang-Ju Liu, Alan Weinstein, and Ping Xu.
\newblock {Manin triples for Lie bialgebroids}.
\newblock {\em Journal of Differential Geometry}, 45(3):547--574, 1997.

\bibitem{Lu06}
Jiang-Hua Lu and Milen Yakimov.
\newblock {Group orbits and regular partitions of Poisson manifolds}.
\newblock {\em Communications in Mathematical Physics}, 283(3):729--748, 2008.

\bibitem{Mackenzie:1998te}
Kirill Charles~Howard Mackenzie.
\newblock {Double Lie algebroids and the double of a Lie bialgebroid}.
\newblock Available at \url{http://arxiv.org/pdf/math/9808081}, 1998.

\bibitem{Mackenzie:1998ge}
Kirill Charles~Howard Mackenzie.
\newblock {Drinfel'd doubles and Ehresmann doubles for Lie algebroids and Lie
  bialgebroids}.
\newblock {\em Electronic Research Announcements of the American Mathematical
  Society}, 4:74--87 (electronic), 1998.

\bibitem{Mackenzie:1999vk}
Kirill Charles~Howard Mackenzie.
\newblock {On symplectic double groupoids and the duality of Poisson
  groupoids}.
\newblock {\em International Journal of Mathematics}, 10(4):435--456, 1999.

\bibitem{Mackenzie:2005tc}
Kirill Charles~Howard Mackenzie.
\newblock {Duality and triple structures}.
\newblock In {\em The breadth of symplectic and Poisson geometry}, pages
  455--481. Birkh\"auser Boston Inc., Boston, MA, 2005.

\bibitem{Mackenzie05}
Kirill Charles~Howard Mackenzie.
\newblock {\em {General theory of Lie groupoids and Lie algebroids}}, volume
  213 of {\em London Mathematical Society Lecture Note Series}.
\newblock Cambridge University Press, Cambridge, 2005.

\bibitem{moerdijk03}
Ieke Moerdijk and Janez Mr{\v c}un.
\newblock {\em {Introduction to foliations and Lie groupoids}}.
\newblock Cambridge University Press, 2003.

\bibitem{Ortiz:2008bd}
Cristi{\'a}n Ortiz.
\newblock {Multiplicative Dirac structures on Lie groups}.
\newblock {\em Comptes Rendus Math{\'e}matique}, 346(23-24):1279--1282,
  December 2008.

\bibitem{Pradines:1974tc}
Jean Pradines.
\newblock {Repr\'esentation des jets non holonomes par des morphismes
  vectoriels doubles soud\'es}.
\newblock {\em Comptes Rendus Hebdomadaires des S{\'e}ances de l'Acad{\'e}mie
  des Sciences. S{\'e}ries A}, 278:1523--1526, 1974.

\bibitem{Pradines:1988td}
Jean Pradines.
\newblock {Remarque sur le groupo\"\i de cotangent de Weinstein-Dazord}.
\newblock {\em Comptes Rendus des S\'eances de l'Acad\'emie des Sciences.
  S\'erie I. Math\'ematique}, 306(13):557--560, 1988.

\bibitem{Roytenberg99}
Dmitry Roytenberg.
\newblock {\em {Courant algebroids, derived brackets and even symplectic
  supermanifolds}}.
\newblock PhD thesis, University of California, Berkeley, 1999.

\bibitem{Roytenberg:2002}
Dmitry Roytenberg.
\newblock {On the structure of graded symplectic supermanifolds and Courant
  algebroids}.
\newblock In {\em Quantization, Poisson brackets and beyond (Manchester,
  2001)}, pages 169--185. Amer. Math. Soc., Providence, RI, 2002.

\bibitem{Salazar:2013tk}
Maria~Amelia Salazar.
\newblock {\em {Pfaffian groupoids}}.
\newblock PhD thesis, arXiv.org, Utrecht, June 2013.

\bibitem{LetToWein}
Pavol {\v S}evera.
\newblock {Letters to A. Weinstein}.
\newblock Available at \url{http://sophia.dtp.fmph.uniba.sk/~severa/letters/}.

\bibitem{Severa:2005vla}
Pavol {\v S}evera.
\newblock {Some title containing the words ``homotopy'' and ``symplectic'',
  e.g. this one}.
\newblock In {\em Travaux math\'ematiques. Fasc. XVI}, pages 121--137. Univ.
  Luxemb., Luxembourg, 2005.

\bibitem{Severa:2001}
Pavol {\v S}evera and Alan Weinstein.
\newblock {Poisson geometry with a 3-form background}.
\newblock {\em Progress of Theoretical Physics. Supplement}, (144):145--154,
  2001.

\bibitem{Uchino02}
Kyousuke Uchino.
\newblock {Remarks on the definition of a Courant algebroid}.
\newblock {\em Letters in Mathematical Physics}, 60(2):171--175, 2002.

\bibitem{Xu95}
Ping Xu.
\newblock {On Poisson groupoids}.
\newblock {\em International Journal of Mathematics}, 6(1):101--124, 1995.

\bibitem{Yano:1973wm}
Kentaro Yano and Shigeru Ishihara.
\newblock {\em {Tangent and cotangent bundles: differential geometry}}.
\newblock Marcel Dekker Inc., New York, 1973.
\newblock Pure and Applied Mathematics, No. 16.

\end{thebibliography}
